\documentclass[
a4paper,
11pt, 
DIV=13,
toc=bibliography, 
headsepline, 
parskip=half-,
abstract=true,
]{scrartcl}

\usepackage[english]{babel}
\usepackage[utf8]{inputenc}
\usepackage[T1]{fontenc}
\usepackage[english,cleanlook]{isodate}
\usepackage{lmodern}
\usepackage[babel=true]{microtype}
\usepackage{amsmath}
\usepackage{mathtools}
\usepackage{amssymb}
\usepackage{amsthm}
\usepackage{mathrsfs}
\usepackage{scrlayer-scrpage} 
\usepackage{enumitem}
\usepackage{tikz} 
\usepackage{tikz-3dplot} 
\usetikzlibrary{plotmarks,positioning} 

\ihead{M.\,B.\,Schulz}
\ohead{\csname @title\endcsname}
\setkomafont{pageheadfoot}{\footnotesize}

\providecommand{\R}{\mathbb{R}}
\providecommand{\Z}{\mathbb{Z}}
\providecommand{\N}{\mathbb{N}}
\providecommand{\B}{\mathbb{B}}
\providecommand{\dih}{\mathbb{D}}
\providecommand{\st}{\, :\ }
\providecommand{\rotation}{\textsf{R}}
\providecommand{\reflection}{\textsf{R}}
\providecommand{\translation}{\textsf{T}} 
\providecommand{\hsd}{\mathscr{H}}
\providecommand{\gsum}{\mathfrak{g}}  
\providecommand{\bsum}{\mathfrak{b}} 
\providecommand{\betti}{\beta}  
\providecommand{\symdiff}{\mathbin{\triangle}}
\providecommand{\heli}{X} 
\providecommand{\height}{\ell}  
\providecommand{\ambient}[1][\rho]{M^{#1}} 
\providecommand{\sol}{\mathcal{O}} 

\DeclarePairedDelimiter\abs{\lvert}{\rvert}
\DeclarePairedDelimiter\nm{\lVert}{\rVert}
\DeclarePairedDelimiter\sk{\langle}{\rangle}
\DeclarePairedDelimiter\interval{]}{[}
\DeclarePairedDelimiter\Interval{[}{[}
\DeclarePairedDelimiter\intervaL{]}{]}
\DeclarePairedDelimiter\IntervaL{[}{]}

\makeatletter
\tikzset{
    scale plot marks/.is choice,
    scale plot marks/true/.style={},	
    scale plot marks/false/.code={
        \def\pgfuseplotmark##1{\pgftransformresetnontranslations\csname pgf@plot@mark@##1\endcsname}
    },
every mark/.append style={scale plot marks=false},
plus/.style={mark=+,mark size=2.25pt},
vdash/.style={mark=|,mark size=2.25pt},
hdash/.style={mark=-,mark size=2.25pt},
bullet/.style={mark=*,mark size=1.125pt},
}
\makeatother
\definecolor{FarbeA}{cmyk}{0,0.5,1,0}
\tikzset{facet/.style={draw=black,fill=FarbeA!55}}
\pgfmathsetmacro{\hoehe}{pi/2}
\tikzset{axis/.style={red,ultra thick}}
\providecommand{\cylinder}[1]{
\begin{tikzpicture}[tdplot_main_coords,line cap=round,line join=round,semithick,scale=\textwidth/6.5cm]
\shade[right color=gray,left color=gray!5] 
({cos(\phiO)},{sin(\phiO)},\hoehe)
arc[start angle=\phiO,end angle=\phiO+90,radius=1]
--+(0,0,-\hoehe)
arc[start angle=\phiO+90,end angle=\phiO,radius=1]
--cycle;
\shade[left color=gray,right color=gray!5] 
({-cos(\phiO)},{-sin(\phiO)},\hoehe)
arc[start angle=\phiO+180,end angle=\phiO+90,radius=1]
--+(0,0,-\hoehe)
arc[start angle=\phiO+90,end angle=\phiO+180,radius=1]
--cycle;
\draw[shade,top color=gray!80,bottom color=gray!80,middle color=gray!25,shading angle=-cos(\thetaO)*\phiO]({cos(\phiO)},{sin(\phiO)},0)
arc[start angle=\phiO,end angle=\phiO+360,radius=1];
#1
\draw({cos(\phiO)},{sin(\phiO)},\hoehe)
arc[start angle=\phiO,end angle=\phiO+360,radius=1];
\draw
({ cos(\phiO)},{ sin(\phiO)},0)--++(0,0,\hoehe)
({-cos(\phiO)},{-sin(\phiO)},0)--++(0,0,\hoehe);
\end{tikzpicture}}

\pgfmathsetmacro{\picscale}{0.46}
\pgfmathsetmacro{\unitscale}{\picscale*\textwidth/6.5cm}
\providecommand{\FBMS}[2][\picscale]{\draw(0,0,0)node[inner sep=0]{\includegraphics[width=#1\textwidth]{figures-toroids/#2}};} 
\providecommand{\torusback}{
\path[facetback](  3.1265,  0.6515,  0.0000)..controls(  3.0901,  0.8232,  0.1581)and(  3.0348,  1.0646,  0.0146)..(  2.9430,  1.2190,  0.1671)..controls(  2.9519,  1.2227,  0.1120)and(  2.9564,  1.2246,  0.0560)..(  2.9564,  1.2246,  0.0000)..controls(  3.0254,  1.0581,  0.0000)and(  3.0841,  0.8571,  0.0000)..(  3.1265,  0.6515,  0.0000)--cycle;
\path[facetback](  2.9564,  1.2246,  0.0000)..controls(  2.9564,  1.2246,  0.0581)and(  2.9515,  1.2226,  0.1162)..(  2.9420,  1.2186,  0.1733)..controls(  2.7543,  1.5264,  0.4780)and(  2.5385,  1.9962,  0.2160)..(  2.1876,  2.1876,  0.4466)..controls(  2.2362,  2.2362,  0.3084)and(  2.2627,  2.2627,  0.1545)..(  2.2627,  2.2627,  0.0000)..controls(  2.5293,  1.9962,  0.0000)and(  2.8121,  1.5729,  0.0000)..(  2.9564,  1.2246,  0.0000)--cycle;
\path[facetback](  2.2627,  2.2627,  0.0000)..controls(  2.2627,  2.2627,  0.1566)and(  2.2355,  2.2355,  0.3125)..(  2.1856,  2.1856,  0.4522)..controls(  1.8609,  2.3585,  0.6620)and(  1.5359,  2.7078,  0.4782)..(  1.1475,  2.7702,  0.6015)..controls(  1.1967,  2.8891,  0.4293)and(  1.2246,  2.9564,  0.2152)..(  1.2246,  2.9564,  0.0000)..controls(  1.5729,  2.8121,  0.0000)and(  1.9962,  2.5293,  0.0000)..(  2.2627,  2.2627,  0.0000)--cycle;
\path[facetback](  1.2246,  2.9564,  0.0000)..controls(  1.2246,  2.9564,  0.2157)and(  1.1966,  2.8888,  0.4304)..(  1.1471,  2.7693,  0.6028)..controls(  0.7888,  2.8254,  0.7153)and(  0.3705,  2.9800,  0.6268)..(  0.0000,  2.9480,  0.6636)..controls(  0.0000,  3.1080,  0.4824)and(  0.0000,  3.2000,  0.2419)..(  0.0000,  3.2000,  0.0000)..controls(  0.3770,  3.2000,  0.0000)and(  0.8763,  3.1007,  0.0000)..(  1.2246,  2.9564,  0.0000)--cycle;
\path[facetback](  0.0000,  3.2000,  0.0000)..controls(  0.0000,  3.2000,  0.2419)and(  0.0000,  3.1080,  0.4824)..(  0.0000,  2.9480,  0.6636)..controls( -0.3445,  2.9177,  0.6974)and( -0.7983,  2.8274,  0.6974)..( -1.1281,  2.7236,  0.6636)..controls( -1.1894,  2.8714,  0.4824)and( -1.2246,  2.9564,  0.2419)..( -1.2246,  2.9564,  0.0000)..controls( -0.8763,  3.1007,  0.0000)and( -0.3770,  3.2000,  0.0000)..(  0.0000,  3.2000,  0.0000)--cycle;
\path[facetback]( -1.2246,  2.9564,  0.0000)..controls( -1.2246,  2.9564,  0.2419)and( -1.1894,  2.8714,  0.4824)..( -1.1281,  2.7236,  0.6636)..controls( -1.4827,  2.6114,  0.6268)and( -1.8100,  2.3085,  0.7153)..( -2.1196,  2.1196,  0.6028)..controls( -2.2110,  2.2110,  0.4304)and( -2.2627,  2.2627,  0.2157)..( -2.2627,  2.2627,  0.0000)..controls( -1.9962,  2.5293,  0.0000)and( -1.5729,  2.8121,  0.0000)..( -1.2246,  2.9564,  0.0000)--cycle;
\path[facetback]( -2.2627,  2.2627,  0.0000)..controls( -2.2627,  2.2627,  0.2152)and( -2.2112,  2.2112,  0.4293)..( -2.1202,  2.1202,  0.6015)..controls( -2.4552,  1.9140,  0.4782)and( -2.6218,  1.4668,  0.6620)..( -2.8556,  1.1828,  0.4522)..controls( -2.9208,  1.2098,  0.3125)and( -2.9564,  1.2246,  0.1566)..( -2.9564,  1.2246,  0.0000)..controls( -2.8121,  1.5729,  0.0000)and( -2.5293,  1.9962,  0.0000)..( -2.2627,  2.2627,  0.0000)--cycle;
\path[facetback]( -2.9564,  1.2246,  0.0000)..controls( -2.9564,  1.2246,  0.1545)and( -2.9218,  1.2102,  0.3084)..( -2.8582,  1.1839,  0.4466)..controls( -3.1092,  0.8728,  0.2160)and( -3.1288,  0.3561,  0.4780)..( -3.1844,  0.0000,  0.1733)..controls( -3.1947,  0.0000,  0.1162)and( -3.2000,  0.0000,  0.0581)..( -3.2000,  0.0000,  0.0000)..controls( -3.2000,  0.3770,  0.0000)and( -3.1007,  0.8763,  0.0000)..( -2.9564,  1.2246,  0.0000)--cycle;
\path[facetback]( -3.2000,  0.0000,  0.0000)..controls( -3.2000,  0.0000,  0.0560)and( -3.1951,  0.0000,  0.1120)..( -3.1855,  0.0000,  0.1671)..controls( -3.2113, -0.1784,  0.0142)and( -3.1695, -0.4240,  0.1586)..( -3.1374, -0.5972,  0.0000)..controls( -3.1771, -0.3901,  0.0000)and( -3.2000, -0.1810,  0.0000)..( -3.2000,  0.0000,  0.0000)--cycle;
\path[facetback](  1.2162,  0.0000,  0.1764)..controls(  1.2055,  0.0000,  0.1183)and(  1.2000,  0.0000,  0.0592)..(  1.2000,  0.0000,  0.0000)..controls(  1.2000,  0.0679,  0.0000)and(  1.1914,  0.1463,  0.0000)..(  1.1765,  0.2239,  0.0000)..controls(  1.1777,  0.2176, -0.0058)and(  1.2173,  0.0071,  0.1824)..(  1.2162,  0.0000,  0.1764)--cycle;
\path[facetback]( -1.1725, -0.2443,  0.0000)..controls( -1.1565, -0.3214,  0.0000)and( -1.1345, -0.3968,  0.0000)..( -1.1087, -0.4592,  0.0000)..controls( -1.1087, -0.4592,  0.0592)and( -1.1137, -0.4613,  0.1183)..( -1.1236, -0.4654,  0.1764)..controls( -1.1273, -0.4595,  0.1823)and( -1.1712, -0.2504, -0.0056)..( -1.1725, -0.2443,  0.0000)--cycle;
\path[facetback]( -1.1087, -0.4592,  0.0000)..controls( -1.1087, -0.4592, -0.2618)and( -1.2082, -0.5005, -0.5220)..( -1.3793, -0.5713, -0.7071)..controls( -1.3119, -0.7338, -0.7071)and( -1.1800, -0.9313, -0.7071)..( -1.0556, -1.0556, -0.7071)..controls( -0.9247, -0.9247, -0.5220)and( -0.8485, -0.8485, -0.2618)..( -0.8485, -0.8485,  0.0000)..controls( -0.9485, -0.7486,  0.0000)and( -1.0546, -0.5898,  0.0000)..( -1.1087, -0.4592,  0.0000)--cycle;
\path[facetback]( -0.8485, -0.8485,  0.0000)..controls( -0.8485, -0.8485, -0.2618)and( -0.9247, -0.9247, -0.5220)..( -1.0556, -1.0556, -0.7071)..controls( -0.9313, -1.1800, -0.7071)and( -0.7338, -1.3119, -0.7071)..( -0.5713, -1.3793, -0.7071)..controls( -0.5005, -1.2082, -0.5220)and( -0.4592, -1.1087, -0.2618)..( -0.4592, -1.1087,  0.0000)..controls( -0.5898, -1.0546,  0.0000)and( -0.7486, -0.9485,  0.0000)..( -0.8485, -0.8485,  0.0000)--cycle;
\path[facetback]( -0.4592, -1.1087,  0.0000)..controls( -0.4592, -1.1087, -0.2618)and( -0.5005, -1.2082, -0.5220)..( -0.5713, -1.3793, -0.7071)..controls( -0.4088, -1.4466, -0.7071)and( -0.1759, -1.4929, -0.7071)..( -0.0000, -1.4929, -0.7071)..controls( -0.0000, -1.3078, -0.5220)and( -0.0000, -1.2000, -0.2618)..( -0.0000, -1.2000,  0.0000)..controls( -0.1414, -1.2000,  0.0000)and( -0.3286, -1.1628,  0.0000)..( -0.4592, -1.1087,  0.0000)--cycle;
\path[facetback]( -0.0000, -1.2000,  0.0000)..controls( -0.0000, -1.2000, -0.2618)and( -0.0000, -1.3078, -0.5220)..( -0.0000, -1.4929, -0.7071)..controls(  0.1759, -1.4929, -0.7071)and(  0.4088, -1.4466, -0.7071)..(  0.5713, -1.3793, -0.7071)..controls(  0.5005, -1.2082, -0.5220)and(  0.4592, -1.1087, -0.2618)..(  0.4592, -1.1087,  0.0000)..controls(  0.3286, -1.1628,  0.0000)and(  0.1414, -1.2000,  0.0000)..( -0.0000, -1.2000,  0.0000)--cycle;
\path[facetback](  0.4592, -1.1087,  0.0000)..controls(  0.4592, -1.1087, -0.2618)and(  0.5005, -1.2082, -0.5220)..(  0.5713, -1.3793, -0.7071)..controls(  0.7338, -1.3119, -0.7071)and(  0.9313, -1.1800, -0.7071)..(  1.0556, -1.0556, -0.7071)..controls(  0.9247, -0.9247, -0.5220)and(  0.8485, -0.8485, -0.2618)..(  0.8485, -0.8485,  0.0000)..controls(  0.7486, -0.9485,  0.0000)and(  0.5898, -1.0546,  0.0000)..(  0.4592, -1.1087,  0.0000)--cycle;
\path[facetback](  0.8485, -0.8485,  0.0000)..controls(  0.8485, -0.8485, -0.2618)and(  0.9247, -0.9247, -0.5220)..(  1.0556, -1.0556, -0.7071)..controls(  1.1800, -0.9313, -0.7071)and(  1.3119, -0.7338, -0.7071)..(  1.3793, -0.5713, -0.7071)..controls(  1.2082, -0.5005, -0.5220)and(  1.1087, -0.4592, -0.2618)..(  1.1087, -0.4592,  0.0000)..controls(  1.0546, -0.5898,  0.0000)and(  0.9485, -0.7486,  0.0000)..(  0.8485, -0.8485,  0.0000)--cycle;
\path[facetback](  1.1087, -0.4592,  0.0000)..controls(  1.1087, -0.4592, -0.2618)and(  1.2082, -0.5005, -0.5220)..(  1.3793, -0.5713, -0.7071)..controls(  1.4466, -0.4088, -0.7071)and(  1.4929, -0.1759, -0.7071)..(  1.4929, -0.0000, -0.7071)..controls(  1.3078, -0.0000, -0.5220)and(  1.2000, -0.0000, -0.2618)..(  1.2000, -0.0000,  0.0000)..controls(  1.2000, -0.1414,  0.0000)and(  1.1628, -0.3286,  0.0000)..(  1.1087, -0.4592,  0.0000)--cycle;
\path[facetback](  1.4929,  0.0000, -0.7071)..controls(  1.6780,  0.0000, -0.8922)and(  1.9382,  0.0000, -1.0000)..(  2.2000,  0.0000, -1.0000)..controls(  2.2000,  0.2592, -1.0000)and(  2.1317,  0.6025, -1.0000)..(  2.0325,  0.8419, -1.0000)..controls(  1.7907,  0.7417, -1.0000)and(  1.5503,  0.6421, -0.8922)..(  1.3793,  0.5713, -0.7071)..controls(  1.4466,  0.4088, -0.7071)and(  1.4929,  0.1759, -0.7071)..(  1.4929,  0.0000, -0.7071)--cycle;
\path[facetback](  1.3793,  0.5713, -0.7071)..controls(  1.5503,  0.6421, -0.8922)and(  1.7907,  0.7417, -1.0000)..(  2.0325,  0.8419, -1.0000)..controls(  1.9334,  1.0814, -1.0000)and(  1.7389,  1.3724, -1.0000)..(  1.5556,  1.5556, -1.0000)..controls(  1.3705,  1.3705, -1.0000)and(  1.1865,  1.1865, -0.8922)..(  1.0556,  1.0556, -0.7071)..controls(  1.1800,  0.9313, -0.7071)and(  1.3119,  0.7338, -0.7071)..(  1.3793,  0.5713, -0.7071)--cycle;
\path[facetback](  1.0556,  1.0556, -0.7071)..controls(  1.1865,  1.1865, -0.8922)and(  1.3705,  1.3705, -1.0000)..(  1.5556,  1.5556, -1.0000)..controls(  1.3724,  1.7389, -1.0000)and(  1.0814,  1.9334, -1.0000)..(  0.8419,  2.0325, -1.0000)..controls(  0.7417,  1.7907, -1.0000)and(  0.6421,  1.5503, -0.8922)..(  0.5713,  1.3793, -0.7071)..controls(  0.7338,  1.3119, -0.7071)and(  0.9313,  1.1800, -0.7071)..(  1.0556,  1.0556, -0.7071)--cycle;
\path[facetback](  0.5713,  1.3793, -0.7071)..controls(  0.6421,  1.5503, -0.8922)and(  0.7417,  1.7907, -1.0000)..(  0.8419,  2.0325, -1.0000)..controls(  0.6025,  2.1317, -1.0000)and(  0.2592,  2.2000, -1.0000)..(  0.0000,  2.2000, -1.0000)..controls(  0.0000,  1.9382, -1.0000)and(  0.0000,  1.6780, -0.8922)..(  0.0000,  1.4929, -0.7071)..controls(  0.1759,  1.4929, -0.7071)and(  0.4088,  1.4466, -0.7071)..(  0.5713,  1.3793, -0.7071)--cycle;
\path[facetback](  0.0000,  1.4929, -0.7071)..controls(  0.0000,  1.6780, -0.8922)and(  0.0000,  1.9382, -1.0000)..(  0.0000,  2.2000, -1.0000)..controls( -0.2592,  2.2000, -1.0000)and( -0.6025,  2.1317, -1.0000)..( -0.8419,  2.0325, -1.0000)..controls( -0.7417,  1.7907, -1.0000)and( -0.6421,  1.5503, -0.8922)..( -0.5713,  1.3793, -0.7071)..controls( -0.4088,  1.4466, -0.7071)and( -0.1759,  1.4929, -0.7071)..(  0.0000,  1.4929, -0.7071)--cycle;
\path[facetback]( -0.5713,  1.3793, -0.7071)..controls( -0.6421,  1.5503, -0.8922)and( -0.7417,  1.7907, -1.0000)..( -0.8419,  2.0325, -1.0000)..controls( -1.0814,  1.9334, -1.0000)and( -1.3724,  1.7389, -1.0000)..( -1.5556,  1.5556, -1.0000)..controls( -1.3705,  1.3705, -1.0000)and( -1.1865,  1.1865, -0.8922)..( -1.0556,  1.0556, -0.7071)..controls( -0.9313,  1.1800, -0.7071)and( -0.7338,  1.3119, -0.7071)..( -0.5713,  1.3793, -0.7071)--cycle;
\path[facetback]( -1.0556,  1.0556, -0.7071)..controls( -1.1865,  1.1865, -0.8922)and( -1.3705,  1.3705, -1.0000)..( -1.5556,  1.5556, -1.0000)..controls( -1.7389,  1.3724, -1.0000)and( -1.9334,  1.0814, -1.0000)..( -2.0325,  0.8419, -1.0000)..controls( -1.7907,  0.7417, -1.0000)and( -1.5503,  0.6421, -0.8922)..( -1.3793,  0.5713, -0.7071)..controls( -1.3119,  0.7338, -0.7071)and( -1.1800,  0.9313, -0.7071)..( -1.0556,  1.0556, -0.7071)--cycle;
\path[facetback]( -1.3793,  0.5713, -0.7071)..controls( -1.5503,  0.6421, -0.8922)and( -1.7907,  0.7417, -1.0000)..( -2.0325,  0.8419, -1.0000)..controls( -2.1317,  0.6025, -1.0000)and( -2.2000,  0.2592, -1.0000)..( -2.2000,  0.0000, -1.0000)..controls( -1.9382,  0.0000, -1.0000)and( -1.6780,  0.0000, -0.8922)..( -1.4929,  0.0000, -0.7071)..controls( -1.4929,  0.1759, -0.7071)and( -1.4466,  0.4088, -0.7071)..( -1.3793,  0.5713, -0.7071)--cycle;
\path[facetback]( -1.4929,  0.0000, -0.7071)..controls( -1.6780,  0.0000, -0.8922)and( -1.9382,  0.0000, -1.0000)..( -2.2000,  0.0000, -1.0000)..controls( -2.2000, -0.2592, -1.0000)and( -2.1317, -0.6025, -1.0000)..( -2.0325, -0.8419, -1.0000)..controls( -1.7907, -0.7417, -1.0000)and( -1.5503, -0.6421, -0.8922)..( -1.3793, -0.5713, -0.7071)..controls( -1.4466, -0.4088, -0.7071)and( -1.4929, -0.1759, -0.7071)..( -1.4929,  0.0000, -0.7071)--cycle;
\path[facetback]( -1.3793, -0.5713, -0.7071)..controls( -1.5503, -0.6421, -0.8922)and( -1.7907, -0.7417, -1.0000)..( -2.0325, -0.8419, -1.0000)..controls( -1.9334, -1.0814, -1.0000)and( -1.7389, -1.3724, -1.0000)..( -1.5556, -1.5556, -1.0000)..controls( -1.3705, -1.3705, -1.0000)and( -1.1865, -1.1865, -0.8922)..( -1.0556, -1.0556, -0.7071)..controls( -1.1800, -0.9313, -0.7071)and( -1.3119, -0.7338, -0.7071)..( -1.3793, -0.5713, -0.7071)--cycle;
\path[facetback]( -1.0556, -1.0556, -0.7071)..controls( -1.1865, -1.1865, -0.8922)and( -1.3705, -1.3705, -1.0000)..( -1.5556, -1.5556, -1.0000)..controls( -1.3724, -1.7389, -1.0000)and( -1.0814, -1.9334, -1.0000)..( -0.8419, -2.0325, -1.0000)..controls( -0.7417, -1.7907, -1.0000)and( -0.6421, -1.5503, -0.8922)..( -0.5713, -1.3793, -0.7071)..controls( -0.7338, -1.3119, -0.7071)and( -0.9313, -1.1800, -0.7071)..( -1.0556, -1.0556, -0.7071)--cycle;
\path[facetback]( -0.5713, -1.3793, -0.7071)..controls( -0.6421, -1.5503, -0.8922)and( -0.7417, -1.7907, -1.0000)..( -0.8419, -2.0325, -1.0000)..controls( -0.6025, -2.1317, -1.0000)and( -0.2592, -2.2000, -1.0000)..( -0.0000, -2.2000, -1.0000)..controls( -0.0000, -1.9382, -1.0000)and( -0.0000, -1.6780, -0.8922)..( -0.0000, -1.4929, -0.7071)..controls( -0.1759, -1.4929, -0.7071)and( -0.4088, -1.4466, -0.7071)..( -0.5713, -1.3793, -0.7071)--cycle;
\path[facetback]( -0.0000, -1.4929, -0.7071)..controls( -0.0000, -1.6780, -0.8922)and( -0.0000, -1.9382, -1.0000)..( -0.0000, -2.2000, -1.0000)..controls(  0.2592, -2.2000, -1.0000)and(  0.6025, -2.1317, -1.0000)..(  0.8419, -2.0325, -1.0000)..controls(  0.7417, -1.7907, -1.0000)and(  0.6421, -1.5503, -0.8922)..(  0.5713, -1.3793, -0.7071)..controls(  0.4088, -1.4466, -0.7071)and(  0.1759, -1.4929, -0.7071)..( -0.0000, -1.4929, -0.7071)--cycle;
\path[facetback](  0.5713, -1.3793, -0.7071)..controls(  0.6421, -1.5503, -0.8922)and(  0.7417, -1.7907, -1.0000)..(  0.8419, -2.0325, -1.0000)..controls(  1.0814, -1.9334, -1.0000)and(  1.3724, -1.7389, -1.0000)..(  1.5556, -1.5556, -1.0000)..controls(  1.3705, -1.3705, -1.0000)and(  1.1865, -1.1865, -0.8922)..(  1.0556, -1.0556, -0.7071)..controls(  0.9313, -1.1800, -0.7071)and(  0.7338, -1.3119, -0.7071)..(  0.5713, -1.3793, -0.7071)--cycle;
\path[facetback](  1.0556, -1.0556, -0.7071)..controls(  1.1865, -1.1865, -0.8922)and(  1.3705, -1.3705, -1.0000)..(  1.5556, -1.5556, -1.0000)..controls(  1.7389, -1.3724, -1.0000)and(  1.9334, -1.0814, -1.0000)..(  2.0325, -0.8419, -1.0000)..controls(  1.7907, -0.7417, -1.0000)and(  1.5503, -0.6421, -0.8922)..(  1.3793, -0.5713, -0.7071)..controls(  1.3119, -0.7338, -0.7071)and(  1.1800, -0.9313, -0.7071)..(  1.0556, -1.0556, -0.7071)--cycle;
\path[facetback](  1.3793, -0.5713, -0.7071)..controls(  1.5503, -0.6421, -0.8922)and(  1.7907, -0.7417, -1.0000)..(  2.0325, -0.8419, -1.0000)..controls(  2.1317, -0.6025, -1.0000)and(  2.2000, -0.2592, -1.0000)..(  2.2000, -0.0000, -1.0000)..controls(  1.9382, -0.0000, -1.0000)and(  1.6780, -0.0000, -0.8922)..(  1.4929, -0.0000, -0.7071)..controls(  1.4929, -0.1759, -0.7071)and(  1.4466, -0.4088, -0.7071)..(  1.3793, -0.5713, -0.7071)--cycle;
\path[facetback](  2.2000,  0.0000, -1.0000)..controls(  2.4618,  0.0000, -1.0000)and(  2.7220,  0.0000, -0.8922)..(  2.9071,  0.0000, -0.7071)..controls(  2.9071,  0.3425, -0.7071)and(  2.8169,  0.7961, -0.7071)..(  2.6858,  1.1125, -0.7071)..controls(  2.5148,  1.0417, -0.8922)and(  2.2744,  0.9421, -1.0000)..(  2.0325,  0.8419, -1.0000)..controls(  2.1317,  0.6025, -1.0000)and(  2.2000,  0.2592, -1.0000)..(  2.2000,  0.0000, -1.0000)--cycle;
\path[facetback](  2.0325,  0.8419, -1.0000)..controls(  2.2744,  0.9421, -1.0000)and(  2.5148,  1.0417, -0.8922)..(  2.6858,  1.1125, -0.7071)..controls(  2.5548,  1.4289, -0.7071)and(  2.2978,  1.8135, -0.7071)..(  2.0556,  2.0556, -0.7071)..controls(  1.9247,  1.9247, -0.8922)and(  1.7408,  1.7408, -1.0000)..(  1.5556,  1.5556, -1.0000)..controls(  1.7389,  1.3724, -1.0000)and(  1.9334,  1.0814, -1.0000)..(  2.0325,  0.8419, -1.0000)--cycle;
\path[facetback](  1.5556,  1.5556, -1.0000)..controls(  1.7408,  1.7408, -1.0000)and(  1.9247,  1.9247, -0.8922)..(  2.0556,  2.0556, -0.7071)..controls(  1.8135,  2.2978, -0.7071)and(  1.4289,  2.5548, -0.7071)..(  1.1125,  2.6858, -0.7071)..controls(  1.0417,  2.5148, -0.8922)and(  0.9421,  2.2744, -1.0000)..(  0.8419,  2.0325, -1.0000)..controls(  1.0814,  1.9334, -1.0000)and(  1.3724,  1.7389, -1.0000)..(  1.5556,  1.5556, -1.0000)--cycle;
\path[facetback](  0.8419,  2.0325, -1.0000)..controls(  0.9421,  2.2744, -1.0000)and(  1.0417,  2.5148, -0.8922)..(  1.1125,  2.6858, -0.7071)..controls(  0.7961,  2.8169, -0.7071)and(  0.3425,  2.9071, -0.7071)..(  0.0000,  2.9071, -0.7071)..controls(  0.0000,  2.7220, -0.8922)and(  0.0000,  2.4618, -1.0000)..(  0.0000,  2.2000, -1.0000)..controls(  0.2592,  2.2000, -1.0000)and(  0.6025,  2.1317, -1.0000)..(  0.8419,  2.0325, -1.0000)--cycle;
\path[facetback](  0.0000,  2.2000, -1.0000)..controls(  0.0000,  2.4618, -1.0000)and(  0.0000,  2.7220, -0.8922)..(  0.0000,  2.9071, -0.7071)..controls( -0.3425,  2.9071, -0.7071)and( -0.7961,  2.8169, -0.7071)..( -1.1125,  2.6858, -0.7071)..controls( -1.0417,  2.5148, -0.8922)and( -0.9421,  2.2744, -1.0000)..( -0.8419,  2.0325, -1.0000)..controls( -0.6025,  2.1317, -1.0000)and( -0.2592,  2.2000, -1.0000)..(  0.0000,  2.2000, -1.0000)--cycle;
\path[facetback]( -0.8419,  2.0325, -1.0000)..controls( -0.9421,  2.2744, -1.0000)and( -1.0417,  2.5148, -0.8922)..( -1.1125,  2.6858, -0.7071)..controls( -1.4289,  2.5548, -0.7071)and( -1.8135,  2.2978, -0.7071)..( -2.0556,  2.0556, -0.7071)..controls( -1.9247,  1.9247, -0.8922)and( -1.7408,  1.7408, -1.0000)..( -1.5556,  1.5556, -1.0000)..controls( -1.3724,  1.7389, -1.0000)and( -1.0814,  1.9334, -1.0000)..( -0.8419,  2.0325, -1.0000)--cycle;
\path[facetback]( -1.5556,  1.5556, -1.0000)..controls( -1.7408,  1.7408, -1.0000)and( -1.9247,  1.9247, -0.8922)..( -2.0556,  2.0556, -0.7071)..controls( -2.2978,  1.8135, -0.7071)and( -2.5548,  1.4289, -0.7071)..( -2.6858,  1.1125, -0.7071)..controls( -2.5148,  1.0417, -0.8922)and( -2.2744,  0.9421, -1.0000)..( -2.0325,  0.8419, -1.0000)..controls( -1.9334,  1.0814, -1.0000)and( -1.7389,  1.3724, -1.0000)..( -1.5556,  1.5556, -1.0000)--cycle;
\path[facetback]( -2.0325,  0.8419, -1.0000)..controls( -2.2744,  0.9421, -1.0000)and( -2.5148,  1.0417, -0.8922)..( -2.6858,  1.1125, -0.7071)..controls( -2.8169,  0.7961, -0.7071)and( -2.9071,  0.3425, -0.7071)..( -2.9071,  0.0000, -0.7071)..controls( -2.7220,  0.0000, -0.8922)and( -2.4618,  0.0000, -1.0000)..( -2.2000,  0.0000, -1.0000)..controls( -2.2000,  0.2592, -1.0000)and( -2.1317,  0.6025, -1.0000)..( -2.0325,  0.8419, -1.0000)--cycle;
\path[facetback]( -2.2000,  0.0000, -1.0000)..controls( -2.4618,  0.0000, -1.0000)and( -2.7220,  0.0000, -0.8922)..( -2.9071,  0.0000, -0.7071)..controls( -2.9071, -0.3425, -0.7071)and( -2.8169, -0.7961, -0.7071)..( -2.6858, -1.1125, -0.7071)..controls( -2.5148, -1.0417, -0.8922)and( -2.2744, -0.9421, -1.0000)..( -2.0325, -0.8419, -1.0000)..controls( -2.1317, -0.6025, -1.0000)and( -2.2000, -0.2592, -1.0000)..( -2.2000,  0.0000, -1.0000)--cycle;
\path[facetback]( -2.0325, -0.8419, -1.0000)..controls( -2.2744, -0.9421, -1.0000)and( -2.5148, -1.0417, -0.8922)..( -2.6858, -1.1125, -0.7071)..controls( -2.5548, -1.4289, -0.7071)and( -2.2978, -1.8135, -0.7071)..( -2.0556, -2.0556, -0.7071)..controls( -1.9247, -1.9247, -0.8922)and( -1.7408, -1.7408, -1.0000)..( -1.5556, -1.5556, -1.0000)..controls( -1.7389, -1.3724, -1.0000)and( -1.9334, -1.0814, -1.0000)..( -2.0325, -0.8419, -1.0000)--cycle;
\path[facetback]( -1.5556, -1.5556, -1.0000)..controls( -1.7408, -1.7408, -1.0000)and( -1.9247, -1.9247, -0.8922)..( -2.0556, -2.0556, -0.7071)..controls( -1.8135, -2.2978, -0.7071)and( -1.4289, -2.5548, -0.7071)..( -1.1125, -2.6858, -0.7071)..controls( -1.0417, -2.5148, -0.8922)and( -0.9421, -2.2744, -1.0000)..( -0.8419, -2.0325, -1.0000)..controls( -1.0814, -1.9334, -1.0000)and( -1.3724, -1.7389, -1.0000)..( -1.5556, -1.5556, -1.0000)--cycle;
\path[facetback]( -0.8419, -2.0325, -1.0000)..controls( -0.9421, -2.2744, -1.0000)and( -1.0417, -2.5148, -0.8922)..( -1.1125, -2.6858, -0.7071)..controls( -0.7961, -2.8169, -0.7071)and( -0.3425, -2.9071, -0.7071)..( -0.0000, -2.9071, -0.7071)..controls( -0.0000, -2.7220, -0.8922)and( -0.0000, -2.4618, -1.0000)..( -0.0000, -2.2000, -1.0000)..controls( -0.2592, -2.2000, -1.0000)and( -0.6025, -2.1317, -1.0000)..( -0.8419, -2.0325, -1.0000)--cycle;
\path[facetback]( -0.0000, -2.2000, -1.0000)..controls( -0.0000, -2.4618, -1.0000)and( -0.0000, -2.7220, -0.8922)..( -0.0000, -2.9071, -0.7071)..controls(  0.3425, -2.9071, -0.7071)and(  0.7961, -2.8169, -0.7071)..(  1.1125, -2.6858, -0.7071)..controls(  1.0417, -2.5148, -0.8922)and(  0.9421, -2.2744, -1.0000)..(  0.8419, -2.0325, -1.0000)..controls(  0.6025, -2.1317, -1.0000)and(  0.2592, -2.2000, -1.0000)..( -0.0000, -2.2000, -1.0000)--cycle;
\path[facetback](  0.8419, -2.0325, -1.0000)..controls(  0.9421, -2.2744, -1.0000)and(  1.0417, -2.5148, -0.8922)..(  1.1125, -2.6858, -0.7071)..controls(  1.4289, -2.5548, -0.7071)and(  1.8135, -2.2978, -0.7071)..(  2.0556, -2.0556, -0.7071)..controls(  1.9247, -1.9247, -0.8922)and(  1.7408, -1.7408, -1.0000)..(  1.5556, -1.5556, -1.0000)..controls(  1.3724, -1.7389, -1.0000)and(  1.0814, -1.9334, -1.0000)..(  0.8419, -2.0325, -1.0000)--cycle;
\path[facetback](  1.5556, -1.5556, -1.0000)..controls(  1.7408, -1.7408, -1.0000)and(  1.9247, -1.9247, -0.8922)..(  2.0556, -2.0556, -0.7071)..controls(  2.2978, -1.8135, -0.7071)and(  2.5548, -1.4289, -0.7071)..(  2.6858, -1.1125, -0.7071)..controls(  2.5148, -1.0417, -0.8922)and(  2.2744, -0.9421, -1.0000)..(  2.0325, -0.8419, -1.0000)..controls(  1.9334, -1.0814, -1.0000)and(  1.7389, -1.3724, -1.0000)..(  1.5556, -1.5556, -1.0000)--cycle;
\path[facetback](  2.0325, -0.8419, -1.0000)..controls(  2.2744, -0.9421, -1.0000)and(  2.5148, -1.0417, -0.8922)..(  2.6858, -1.1125, -0.7071)..controls(  2.8169, -0.7961, -0.7071)and(  2.9071, -0.3425, -0.7071)..(  2.9071, -0.0000, -0.7071)..controls(  2.7220, -0.0000, -0.8922)and(  2.4618, -0.0000, -1.0000)..(  2.2000, -0.0000, -1.0000)..controls(  2.2000, -0.2592, -1.0000)and(  2.1317, -0.6025, -1.0000)..(  2.0325, -0.8419, -1.0000)--cycle;
\path[facetback](  2.6858,  1.1125, -0.7071)..controls(  2.8568,  1.1833, -0.5220)and(  2.9564,  1.2246, -0.2618)..(  2.9564,  1.2246, -0.0000)..controls(  2.8121,  1.5729, -0.0000)and(  2.5293,  1.9962, -0.0000)..(  2.2627,  2.2627, -0.0000)..controls(  2.2627,  2.2627, -0.2618)and(  2.1865,  2.1865, -0.5220)..(  2.0556,  2.0556, -0.7071)..controls(  2.2978,  1.8135, -0.7071)and(  2.5548,  1.4289, -0.7071)..(  2.6858,  1.1125, -0.7071)--cycle;
\path[facetback](  2.0556,  2.0556, -0.7071)..controls(  2.1865,  2.1865, -0.5220)and(  2.2627,  2.2627, -0.2618)..(  2.2627,  2.2627, -0.0000)..controls(  1.9962,  2.5293, -0.0000)and(  1.5729,  2.8121, -0.0000)..(  1.2246,  2.9564, -0.0000)..controls(  1.2246,  2.9564, -0.2618)and(  1.1833,  2.8568, -0.5220)..(  1.1125,  2.6858, -0.7071)..controls(  1.4289,  2.5548, -0.7071)and(  1.8135,  2.2978, -0.7071)..(  2.0556,  2.0556, -0.7071)--cycle;
\path[facetback](  1.1125,  2.6858, -0.7071)..controls(  1.1833,  2.8568, -0.5220)and(  1.2246,  2.9564, -0.2618)..(  1.2246,  2.9564, -0.0000)..controls(  0.8763,  3.1007, -0.0000)and(  0.3770,  3.2000, -0.0000)..(  0.0000,  3.2000, -0.0000)..controls(  0.0000,  3.2000, -0.2618)and(  0.0000,  3.0922, -0.5220)..(  0.0000,  2.9071, -0.7071)..controls(  0.3425,  2.9071, -0.7071)and(  0.7961,  2.8169, -0.7071)..(  1.1125,  2.6858, -0.7071)--cycle;
\path[facetback](  0.0000,  2.9071, -0.7071)..controls(  0.0000,  3.0922, -0.5220)and(  0.0000,  3.2000, -0.2618)..(  0.0000,  3.2000, -0.0000)..controls( -0.3770,  3.2000, -0.0000)and( -0.8763,  3.1007, -0.0000)..( -1.2246,  2.9564, -0.0000)..controls( -1.2246,  2.9564, -0.2618)and( -1.1833,  2.8568, -0.5220)..( -1.1125,  2.6858, -0.7071)..controls( -0.7961,  2.8169, -0.7071)and( -0.3425,  2.9071, -0.7071)..(  0.0000,  2.9071, -0.7071)--cycle;
\path[facetback]( -1.1125,  2.6858, -0.7071)..controls( -1.1833,  2.8568, -0.5220)and( -1.2246,  2.9564, -0.2618)..( -1.2246,  2.9564, -0.0000)..controls( -1.5729,  2.8121, -0.0000)and( -1.9962,  2.5293, -0.0000)..( -2.2627,  2.2627, -0.0000)..controls( -2.2627,  2.2627, -0.2618)and( -2.1865,  2.1865, -0.5220)..( -2.0556,  2.0556, -0.7071)..controls( -1.8135,  2.2978, -0.7071)and( -1.4289,  2.5548, -0.7071)..( -1.1125,  2.6858, -0.7071)--cycle;
\path[facetback]( -2.0556,  2.0556, -0.7071)..controls( -2.1865,  2.1865, -0.5220)and( -2.2627,  2.2627, -0.2618)..( -2.2627,  2.2627, -0.0000)..controls( -2.5293,  1.9962, -0.0000)and( -2.8121,  1.5729, -0.0000)..( -2.9564,  1.2246, -0.0000)..controls( -2.9564,  1.2246, -0.2618)and( -2.8568,  1.1833, -0.5220)..( -2.6858,  1.1125, -0.7071)..controls( -2.5548,  1.4289, -0.7071)and( -2.2978,  1.8135, -0.7071)..( -2.0556,  2.0556, -0.7071)--cycle;
\path[facetback]( -2.6858,  1.1125, -0.7071)..controls( -2.8568,  1.1833, -0.5220)and( -2.9564,  1.2246, -0.2618)..( -2.9564,  1.2246, -0.0000)..controls( -3.1007,  0.8763, -0.0000)and( -3.2000,  0.3770, -0.0000)..( -3.2000,  0.0000, -0.0000)..controls( -3.2000,  0.0000, -0.2618)and( -3.0922,  0.0000, -0.5220)..( -2.9071,  0.0000, -0.7071)..controls( -2.9071,  0.3425, -0.7071)and( -2.8169,  0.7961, -0.7071)..( -2.6858,  1.1125, -0.7071)--cycle;
\path[facetback]( -2.6858, -1.1125, -0.7071)..controls( -2.8192, -1.1678, -0.5627)and( -2.9091, -1.2050, -0.3726)..( -2.9423, -1.2187, -0.1718)..controls( -2.7546, -1.5265, -0.4765)and( -2.5390, -1.9967, -0.2145)..( -2.1881, -2.1881, -0.4452)..controls( -2.1542, -2.1542, -0.5418)and( -2.1096, -2.1096, -0.6308)..( -2.0556, -2.0556, -0.7071)..controls( -2.2978, -1.8135, -0.7071)and( -2.5548, -1.4289, -0.7071)..( -2.6858, -1.1125, -0.7071)--cycle;
\path[facetback]( -2.0556, -2.0556, -0.7071)..controls( -2.1085, -2.1085, -0.6323)and( -2.1525, -2.1525, -0.5453)..( -2.1861, -2.1861, -0.4508)..controls( -1.8614, -2.3590, -0.6606)and( -1.5363, -2.7087, -0.4770)..( -1.1478, -2.7711, -0.6003)..controls( -1.1370, -2.7451, -0.6380)and( -1.1253, -2.7166, -0.6738)..( -1.1125, -2.6858, -0.7071)..controls( -1.4289, -2.5548, -0.7071)and( -1.8135, -2.2978, -0.7071)..( -2.0556, -2.0556, -0.7071)--cycle;
\path[facetback]( -1.1125, -2.6858, -0.7071)..controls( -1.1251, -2.7163, -0.6742)and( -1.1368, -2.7444, -0.6388)..( -1.1475, -2.7702, -0.6015)..controls( -0.7891, -2.8263, -0.7140)and( -0.3705, -2.9810, -0.6256)..( -0.0000, -2.9490, -0.6625)..controls( -0.0000, -2.9355, -0.6778)and( -0.0000, -2.9215, -0.6927)..( -0.0000, -2.9071, -0.7071)..controls( -0.3425, -2.9071, -0.7071)and( -0.7961, -2.8169, -0.7071)..( -1.1125, -2.6858, -0.7071)--cycle;
\path[facetback]( -0.0000, -2.9071, -0.7071)..controls( -0.0000, -2.9215, -0.6927)and( -0.0000, -2.9355, -0.6778)..( -0.0000, -2.9490, -0.6625)..controls(  0.3445, -2.9187, -0.6962)and(  0.7986, -2.8284, -0.6962)..(  1.1285, -2.7245, -0.6625)..controls(  1.1234, -2.7121, -0.6778)and(  1.1180, -2.6992, -0.6927)..(  1.1125, -2.6858, -0.7071)..controls(  0.7961, -2.8169, -0.7071)and(  0.3425, -2.9071, -0.7071)..( -0.0000, -2.9071, -0.7071)--cycle;
\path[facetback](  1.1125, -2.6858, -0.7071)..controls(  1.1180, -2.6992, -0.6927)and(  1.1234, -2.7121, -0.6778)..(  1.1285, -2.7245, -0.6625)..controls(  1.4831, -2.6123, -0.6256)and(  1.8106, -2.3091, -0.7140)..(  2.1202, -2.1202, -0.6015)..controls(  2.1005, -2.1005, -0.6388)and(  2.0789, -2.0789, -0.6742)..(  2.0556, -2.0556, -0.7071)..controls(  1.8135, -2.2978, -0.7071)and(  1.4289, -2.5548, -0.7071)..(  1.1125, -2.6858, -0.7071)--cycle;
\path[facetback](  2.0556, -2.0556, -0.7071)..controls(  2.0792, -2.0792, -0.6738)and(  2.1010, -2.1010, -0.6380)..(  2.1209, -2.1209, -0.6003)..controls(  2.4559, -1.9146, -0.4770)and(  2.6224, -1.4671, -0.6606)..(  2.8563, -1.1831, -0.4508)..controls(  2.8123, -1.1649, -0.5453)and(  2.7549, -1.1411, -0.6323)..(  2.6858, -1.1125, -0.7071)..controls(  2.5548, -1.4289, -0.7071)and(  2.2978, -1.8135, -0.7071)..(  2.0556, -2.0556, -0.7071)--cycle;
\path[facetback](  2.6858, -1.1125, -0.7071)..controls(  2.7563, -1.1417, -0.6308)and(  2.8146, -1.1659, -0.5418)..(  2.8589, -1.1842, -0.4452)..controls(  3.1098, -0.8731, -0.2145)and(  3.1290, -0.3561, -0.4765)..(  3.1847, -0.0000, -0.1718)..controls(  3.1488, -0.0000, -0.3726)and(  3.0515, -0.0000, -0.5627)..(  2.9071, -0.0000, -0.7071)..controls(  2.9071, -0.3425, -0.7071)and(  2.8169, -0.7961, -0.7071)..(  2.6858, -1.1125, -0.7071)--cycle;
\path[facetback]( -1.1213, -0.4645,  0.1625)..controls( -1.1130, -0.4610,  0.1088)and( -1.1087, -0.4592,  0.0545)..( -1.1087, -0.4592,  0.0000)..controls( -1.0546, -0.5898,  0.0000)and( -0.9485, -0.7486,  0.0000)..( -0.8485, -0.8485,  0.0000)..controls( -0.8485, -0.8485,  0.1571)and( -0.8760, -0.8760,  0.3136)..( -0.9262, -0.9262,  0.4536)..controls( -0.9656, -0.9055,  0.4788)and( -1.1061, -0.4908,  0.1366)..( -1.1213, -0.4645,  0.1625)--cycle;
\path[facetback]( -0.9217, -0.9217,  0.4410)..controls( -0.8743, -0.8743,  0.3042)and( -0.8485, -0.8485,  0.1524)..( -0.8485, -0.8485,  0.0000)..controls( -0.7486, -0.9485,  0.0000)and( -0.5898, -1.0546,  0.0000)..( -0.4592, -1.1087,  0.0000)..controls( -0.4592, -1.1087,  0.2157)and( -0.4872, -1.1763,  0.4304)..( -0.5367, -1.2957,  0.6028)..controls( -0.6327, -1.2809,  0.6327)and( -0.8539, -0.9592,  0.3960)..( -0.9217, -0.9217,  0.4410)--cycle;
\path[facetback]( -0.5356, -1.2931,  0.5990)..controls( -0.4868, -1.1753,  0.4272)and( -0.4592, -1.1087,  0.2142)..( -0.4592, -1.1087,  0.0000)..controls( -0.3286, -1.1628,  0.0000)and( -0.1414, -1.2000,  0.0000)..( -0.0000, -1.2000,  0.0000)..controls( -0.0000, -1.2000,  0.2414)and( -0.0000, -1.2916,  0.4814)..( -0.0000, -1.4510,  0.6625)..controls( -0.1424, -1.4637,  0.6763)and( -0.4063, -1.3142,  0.5577)..( -0.5356, -1.2931,  0.5990)--cycle;
\path[facetback]( -0.0000, -1.4510,  0.6625)..controls( -0.0000, -1.2916,  0.4814)and( -0.0000, -1.2000,  0.2414)..( -0.0000, -1.2000,  0.0000)..controls(  0.1414, -1.2000,  0.0000)and(  0.3286, -1.1628,  0.0000)..(  0.4592, -1.1087,  0.0000)..controls(  0.4592, -1.1087,  0.2414)and(  0.4943, -1.1933,  0.4814)..(  0.5553, -1.3406,  0.6625)..controls(  0.3932, -1.3920,  0.6455)and(  0.1694, -1.4365,  0.6455)..( -0.0000, -1.4510,  0.6625)--cycle;
\path[facetback](  0.5553, -1.3406,  0.6625)..controls(  0.4943, -1.1933,  0.4814)and(  0.4592, -1.1087,  0.2414)..(  0.4592, -1.1087,  0.0000)..controls(  0.5898, -1.0546,  0.0000)and(  0.7486, -0.9485,  0.0000)..(  0.8485, -0.8485,  0.0000)..controls(  0.8485, -0.8485,  0.2142)and(  0.8995, -0.8995,  0.4272)..(  0.9897, -0.9897,  0.5990)..controls(  0.8783, -1.0586,  0.5577)and(  0.6917, -1.2978,  0.6763)..(  0.5553, -1.3406,  0.6625)--cycle;
\path[facetback](  0.9917, -0.9917,  0.6028)..controls(  0.9003, -0.9003,  0.4304)and(  0.8485, -0.8485,  0.2157)..(  0.8485, -0.8485,  0.0000)..controls(  0.9485, -0.7486,  0.0000)and(  1.0546, -0.5898,  0.0000)..(  1.1087, -0.4592,  0.0000)..controls(  1.1087, -0.4592,  0.1524)and(  1.1424, -0.4732,  0.3042)..(  1.2043, -0.4988,  0.4410)..controls(  1.1560, -0.5594,  0.3960)and(  1.0747, -0.9413,  0.6327)..(  0.9917, -0.9917,  0.6028)--cycle;
\path[facetback](  1.2101, -0.5012,  0.4536)..controls(  1.1445, -0.4741,  0.3136)and(  1.1087, -0.4592,  0.1571)..(  1.1087, -0.4592,  0.0000)..controls(  1.1628, -0.3286,  0.0000)and(  1.2000, -0.1414,  0.0000)..(  1.2000, -0.0000,  0.0000)..controls(  1.2000, -0.0000,  0.0545)and(  1.2047, -0.0000,  0.1088)..(  1.2137, -0.0000,  0.1625)..controls(  1.2097, -0.0301,  0.1366)and(  1.2386, -0.4671,  0.4788)..(  1.2101, -0.5012,  0.4536)--cycle;
\path[facetback](  1.1216,  0.4646, -0.1640)..controls(  1.1536,  0.4778, -0.3678)and(  1.2441,  0.5153, -0.5609)..(  1.3793,  0.5713, -0.7071)..controls(  1.3119,  0.7338, -0.7071)and(  1.1800,  0.9313, -0.7071)..(  1.0556,  1.0556, -0.7071)..controls(  1.0035,  1.0035, -0.6334)and(  0.9601,  0.9601, -0.5479)..(  0.9267,  0.9267, -0.4550)..controls(  0.9661,  0.9060, -0.4802)and(  1.1064,  0.4909, -0.1381)..(  1.1216,  0.4646, -0.1640)--cycle;
\path[facetback](  0.9222,  0.9222, -0.4424)..controls(  0.9562,  0.9562, -0.5401)and(  1.0012,  1.0012, -0.6301)..(  1.0556,  1.0556, -0.7071)..controls(  0.9313,  1.1800, -0.7071)and(  0.7338,  1.3119, -0.7071)..(  0.5713,  1.3793, -0.7071)..controls(  0.5590,  1.3495, -0.6749)and(  0.5475,  1.3219, -0.6404)..(  0.5371,  1.2966, -0.6040)..controls(  0.6330,  1.2818, -0.6340)and(  0.8544,  0.9597, -0.3974)..(  0.9222,  0.9222, -0.4424)--cycle;
\path[facetback](  0.5360,  1.2940, -0.6003)..controls(  0.5468,  1.3200, -0.6380)and(  0.5586,  1.3485, -0.6738)..(  0.5713,  1.3793, -0.7071)..controls(  0.4088,  1.4466, -0.7071)and(  0.1759,  1.4929, -0.7071)..(  0.0000,  1.4929, -0.7071)..controls(  0.0000,  1.4788, -0.6930)and(  0.0000,  1.4652, -0.6785)..(  0.0000,  1.4520, -0.6636)..controls(  0.1424,  1.4647, -0.6774)and(  0.4067,  1.3150, -0.5590)..(  0.5360,  1.2940, -0.6003)--cycle;
\path[facetback](  0.0000,  1.4520, -0.6636)..controls(  0.0000,  1.4652, -0.6785)and(  0.0000,  1.4788, -0.6930)..(  0.0000,  1.4929, -0.7071)..controls( -0.1759,  1.4929, -0.7071)and( -0.4088,  1.4466, -0.7071)..( -0.5713,  1.3793, -0.7071)..controls( -0.5659,  1.3663, -0.6930)and( -0.5607,  1.3537, -0.6785)..( -0.5557,  1.3415, -0.6636)..controls( -0.3936,  1.3929, -0.6467)and( -0.1694,  1.4375, -0.6467)..(  0.0000,  1.4520, -0.6636)--cycle;
\path[facetback]( -0.5557,  1.3415, -0.6636)..controls( -0.5607,  1.3537, -0.6785)and( -0.5659,  1.3663, -0.6930)..( -0.5713,  1.3793, -0.7071)..controls( -0.7338,  1.3119, -0.7071)and( -0.9313,  1.1800, -0.7071)..( -1.0556,  1.0556, -0.7071)..controls( -1.0321,  1.0321, -0.6738)and( -1.0103,  1.0103, -0.6380)..( -0.9904,  0.9904, -0.6003)..controls( -0.8790,  1.0593, -0.5590)and( -0.6921,  1.2987, -0.6774)..( -0.5557,  1.3415, -0.6636)--cycle;
\path[facetback]( -0.9924,  0.9924, -0.6040)..controls( -1.0117,  1.0117, -0.6404)and( -1.0329,  1.0329, -0.6749)..( -1.0556,  1.0556, -0.7071)..controls( -1.1800,  0.9313, -0.7071)and( -1.3119,  0.7338, -0.7071)..( -1.3793,  0.5713, -0.7071)..controls( -1.3081,  0.5418, -0.6301)and( -1.2493,  0.5175, -0.5401)..( -1.2049,  0.4991, -0.4424)..controls( -1.1566,  0.5597, -0.3974)and( -1.0753,  0.9420, -0.6340)..( -0.9924,  0.9924, -0.6040)--cycle;
\path[facetback]( -1.2108,  0.5015, -0.4550)..controls( -1.2544,  0.5196, -0.5479)and( -1.3112,  0.5431, -0.6334)..( -1.3793,  0.5713, -0.7071)..controls( -1.4466,  0.4088, -0.7071)and( -1.4929,  0.1759, -0.7071)..( -1.4929,  0.0000, -0.7071)..controls( -1.3466,  0.0000, -0.5609)and( -1.2487,  0.0000, -0.3678)..( -1.2140,  0.0000, -0.1640)..controls( -1.2100,  0.0301, -0.1381)and( -1.2393,  0.4674, -0.4802)..( -1.2108,  0.5015, -0.4550)--cycle;
\path[facetback]( -1.2164,  0.0000, -0.1779)..controls( -1.2532,  0.0000, -0.3765)and( -1.3500,  0.0000, -0.5642)..( -1.4929,  0.0000, -0.7071)..controls( -1.4929, -0.1759, -0.7071)and( -1.4466, -0.4088, -0.7071)..( -1.3793, -0.5713, -0.7071)..controls( -1.2082, -0.5005, -0.5220)and( -1.1087, -0.4592, -0.2618)..( -1.1087, -0.4592,  0.0000)..controls( -1.1368, -0.3913,  0.0000)and( -1.1604, -0.3081,  0.0000)..( -1.1765, -0.2239,  0.0000)..controls( -1.1776, -0.2180,  0.0055)and( -1.2175, -0.0067, -0.1837)..( -1.2164,  0.0000, -0.1779)--cycle;
\path[facetback](  2.9071,  0.0000, -0.7071)..controls(  3.0530,  0.0000, -0.5612)and(  3.1508,  0.0000, -0.3687)..(  3.1858,  0.0000, -0.1656)..controls(  3.2115,  0.1780, -0.0129)and(  3.1694,  0.4244, -0.1582)..(  3.1374,  0.5972, -0.0000)..controls(  3.0943,  0.8215, -0.0000)and(  3.0314,  1.0435, -0.0000)..(  2.9564,  1.2246, -0.0000)..controls(  2.9564,  1.2246, -0.2618)and(  2.8568,  1.1833, -0.5220)..(  2.6858,  1.1125, -0.7071)..controls(  2.8169,  0.7961, -0.7071)and(  2.9071,  0.3425, -0.7071)..(  2.9071,  0.0000, -0.7071)--cycle;
\path[facetback]( -2.9071,  0.0000, -0.7071)..controls( -3.0922,  0.0000, -0.5220)and( -3.2000,  0.0000, -0.2618)..( -3.2000,  0.0000, -0.0000)..controls( -3.2000, -0.1968, -0.0000)and( -3.1729, -0.4269, -0.0000)..( -3.1265, -0.6515, -0.0000)..controls( -3.0902, -0.8228, -0.1577)and( -3.0349, -1.0651, -0.0134)..( -2.9433, -1.2191, -0.1656)..controls( -2.9110, -1.2058, -0.3687)and( -2.8206, -1.1683, -0.5612)..( -2.6858, -1.1125, -0.7071)..controls( -2.8169, -0.7961, -0.7071)and( -2.9071, -0.3425, -0.7071)..( -2.9071,  0.0000, -0.7071)--cycle;
\path[facetback](  1.2000,  0.0000,  0.0000)..controls(  1.2000,  0.0000, -0.2618)and(  1.3078,  0.0000, -0.5220)..(  1.4929,  0.0000, -0.7071)..controls(  1.4929,  0.1759, -0.7071)and(  1.4466,  0.4088, -0.7071)..(  1.3793,  0.5713, -0.7071)..controls(  1.2472,  0.5166, -0.5642)and(  1.1578,  0.4796, -0.3765)..(  1.1239,  0.4655, -0.1779)..controls(  1.1273,  0.4599, -0.1835)and(  1.1712,  0.2500,  0.0053)..(  1.1725,  0.2443,  0.0000)..controls(  1.1899,  0.1601,  0.0000)and(  1.2000,  0.0738,  0.0000)..(  1.2000,  0.0000,  0.0000)--cycle; 
}
\providecommand{\torusfront}{
\path[facetfront]( -2.9564, -1.2246,  0.0000)..controls( -2.9564, -1.2246,  0.2618)and( -2.8568, -1.1833,  0.5220)..( -2.6858, -1.1125,  0.7071)..controls( -2.5548, -1.4289,  0.7071)and( -2.2978, -1.8135,  0.7071)..( -2.0556, -2.0556,  0.7071)..controls( -2.1865, -2.1865,  0.5220)and( -2.2627, -2.2627,  0.2618)..( -2.2627, -2.2627,  0.0000)..controls( -2.5293, -1.9962,  0.0000)and( -2.8121, -1.5729,  0.0000)..( -2.9564, -1.2246,  0.0000)--cycle;
\path[facetfront]( -2.2627, -2.2627,  0.0000)..controls( -2.2627, -2.2627,  0.2618)and( -2.1865, -2.1865,  0.5220)..( -2.0556, -2.0556,  0.7071)..controls( -1.8135, -2.2978,  0.7071)and( -1.4289, -2.5548,  0.7071)..( -1.1125, -2.6858,  0.7071)..controls( -1.1833, -2.8568,  0.5220)and( -1.2246, -2.9564,  0.2618)..( -1.2246, -2.9564,  0.0000)..controls( -1.5729, -2.8121,  0.0000)and( -1.9962, -2.5293,  0.0000)..( -2.2627, -2.2627,  0.0000)--cycle;
\path[facetfront]( -1.2246, -2.9564,  0.0000)..controls( -1.2246, -2.9564,  0.2618)and( -1.1833, -2.8568,  0.5220)..( -1.1125, -2.6858,  0.7071)..controls( -0.7961, -2.8169,  0.7071)and( -0.3425, -2.9071,  0.7071)..( -0.0000, -2.9071,  0.7071)..controls( -0.0000, -3.0922,  0.5220)and( -0.0000, -3.2000,  0.2618)..( -0.0000, -3.2000,  0.0000)..controls( -0.3770, -3.2000,  0.0000)and( -0.8763, -3.1007,  0.0000)..( -1.2246, -2.9564,  0.0000)--cycle;
\path[facetfront]( -0.0000, -3.2000,  0.0000)..controls( -0.0000, -3.2000,  0.2618)and( -0.0000, -3.0922,  0.5220)..( -0.0000, -2.9071,  0.7071)..controls(  0.3425, -2.9071,  0.7071)and(  0.7961, -2.8169,  0.7071)..(  1.1125, -2.6858,  0.7071)..controls(  1.1833, -2.8568,  0.5220)and(  1.2246, -2.9564,  0.2618)..(  1.2246, -2.9564,  0.0000)..controls(  0.8763, -3.1007,  0.0000)and(  0.3770, -3.2000,  0.0000)..( -0.0000, -3.2000,  0.0000)--cycle;
\path[facetfront](  1.2246, -2.9564,  0.0000)..controls(  1.2246, -2.9564,  0.2618)and(  1.1833, -2.8568,  0.5220)..(  1.1125, -2.6858,  0.7071)..controls(  1.4289, -2.5548,  0.7071)and(  1.8135, -2.2978,  0.7071)..(  2.0556, -2.0556,  0.7071)..controls(  2.1865, -2.1865,  0.5220)and(  2.2627, -2.2627,  0.2618)..(  2.2627, -2.2627,  0.0000)..controls(  1.9962, -2.5293,  0.0000)and(  1.5729, -2.8121,  0.0000)..(  1.2246, -2.9564,  0.0000)--cycle;
\path[facetfront](  2.2627, -2.2627,  0.0000)..controls(  2.2627, -2.2627,  0.2618)and(  2.1865, -2.1865,  0.5220)..(  2.0556, -2.0556,  0.7071)..controls(  2.2978, -1.8135,  0.7071)and(  2.5548, -1.4289,  0.7071)..(  2.6858, -1.1125,  0.7071)..controls(  2.8568, -1.1833,  0.5220)and(  2.9564, -1.2246,  0.2618)..(  2.9564, -1.2246,  0.0000)..controls(  2.8121, -1.5729,  0.0000)and(  2.5293, -1.9962,  0.0000)..(  2.2627, -2.2627,  0.0000)--cycle;
\path[facetfront](  2.9564, -1.2246,  0.0000)..controls(  2.9564, -1.2246,  0.2618)and(  2.8568, -1.1833,  0.5220)..(  2.6858, -1.1125,  0.7071)..controls(  2.8169, -0.7961,  0.7071)and(  2.9071, -0.3425,  0.7071)..(  2.9071, -0.0000,  0.7071)..controls(  3.0922, -0.0000,  0.5220)and(  3.2000, -0.0000,  0.2618)..(  3.2000, -0.0000,  0.0000)..controls(  3.2000, -0.3770,  0.0000)and(  3.1007, -0.8763,  0.0000)..(  2.9564, -1.2246,  0.0000)--cycle;
\path[facetfront](  2.9071,  0.0000,  0.7071)..controls(  2.7220,  0.0000,  0.8922)and(  2.4618,  0.0000,  1.0000)..(  2.2000,  0.0000,  1.0000)..controls(  2.2000,  0.2592,  1.0000)and(  2.1317,  0.6025,  1.0000)..(  2.0325,  0.8419,  1.0000)..controls(  2.2744,  0.9421,  1.0000)and(  2.5148,  1.0417,  0.8922)..(  2.6858,  1.1125,  0.7071)..controls(  2.8169,  0.7961,  0.7071)and(  2.9071,  0.3425,  0.7071)..(  2.9071,  0.0000,  0.7071)--cycle;
\path[facetfront](  2.6858,  1.1125,  0.7071)..controls(  2.5148,  1.0417,  0.8922)and(  2.2744,  0.9421,  1.0000)..(  2.0325,  0.8419,  1.0000)..controls(  1.9334,  1.0814,  1.0000)and(  1.7389,  1.3724,  1.0000)..(  1.5556,  1.5556,  1.0000)..controls(  1.7408,  1.7408,  1.0000)and(  1.9247,  1.9247,  0.8922)..(  2.0556,  2.0556,  0.7071)..controls(  2.2978,  1.8135,  0.7071)and(  2.5548,  1.4289,  0.7071)..(  2.6858,  1.1125,  0.7071)--cycle;
\path[facetfront](  2.0556,  2.0556,  0.7071)..controls(  1.9247,  1.9247,  0.8922)and(  1.7408,  1.7408,  1.0000)..(  1.5556,  1.5556,  1.0000)..controls(  1.3724,  1.7389,  1.0000)and(  1.0814,  1.9334,  1.0000)..(  0.8419,  2.0325,  1.0000)..controls(  0.9421,  2.2744,  1.0000)and(  1.0417,  2.5148,  0.8922)..(  1.1125,  2.6858,  0.7071)..controls(  1.4289,  2.5548,  0.7071)and(  1.8135,  2.2978,  0.7071)..(  2.0556,  2.0556,  0.7071)--cycle;
\path[facetfront](  1.1125,  2.6858,  0.7071)..controls(  1.0417,  2.5148,  0.8922)and(  0.9421,  2.2744,  1.0000)..(  0.8419,  2.0325,  1.0000)..controls(  0.6025,  2.1317,  1.0000)and(  0.2592,  2.2000,  1.0000)..(  0.0000,  2.2000,  1.0000)..controls(  0.0000,  2.4618,  1.0000)and(  0.0000,  2.7220,  0.8922)..(  0.0000,  2.9071,  0.7071)..controls(  0.3425,  2.9071,  0.7071)and(  0.7961,  2.8169,  0.7071)..(  1.1125,  2.6858,  0.7071)--cycle;
\path[facetfront](  0.0000,  2.9071,  0.7071)..controls(  0.0000,  2.7220,  0.8922)and(  0.0000,  2.4618,  1.0000)..(  0.0000,  2.2000,  1.0000)..controls( -0.2592,  2.2000,  1.0000)and( -0.6025,  2.1317,  1.0000)..( -0.8419,  2.0325,  1.0000)..controls( -0.9421,  2.2744,  1.0000)and( -1.0417,  2.5148,  0.8922)..( -1.1125,  2.6858,  0.7071)..controls( -0.7961,  2.8169,  0.7071)and( -0.3425,  2.9071,  0.7071)..(  0.0000,  2.9071,  0.7071)--cycle;
\path[facetfront]( -1.1125,  2.6858,  0.7071)..controls( -1.0417,  2.5148,  0.8922)and( -0.9421,  2.2744,  1.0000)..( -0.8419,  2.0325,  1.0000)..controls( -1.0814,  1.9334,  1.0000)and( -1.3724,  1.7389,  1.0000)..( -1.5556,  1.5556,  1.0000)..controls( -1.7408,  1.7408,  1.0000)and( -1.9247,  1.9247,  0.8922)..( -2.0556,  2.0556,  0.7071)..controls( -1.8135,  2.2978,  0.7071)and( -1.4289,  2.5548,  0.7071)..( -1.1125,  2.6858,  0.7071)--cycle;
\path[facetfront]( -2.0556,  2.0556,  0.7071)..controls( -1.9247,  1.9247,  0.8922)and( -1.7408,  1.7408,  1.0000)..( -1.5556,  1.5556,  1.0000)..controls( -1.7389,  1.3724,  1.0000)and( -1.9334,  1.0814,  1.0000)..( -2.0325,  0.8419,  1.0000)..controls( -2.2744,  0.9421,  1.0000)and( -2.5148,  1.0417,  0.8922)..( -2.6858,  1.1125,  0.7071)..controls( -2.5548,  1.4289,  0.7071)and( -2.2978,  1.8135,  0.7071)..( -2.0556,  2.0556,  0.7071)--cycle;
\path[facetfront]( -2.6858,  1.1125,  0.7071)..controls( -2.5148,  1.0417,  0.8922)and( -2.2744,  0.9421,  1.0000)..( -2.0325,  0.8419,  1.0000)..controls( -2.1317,  0.6025,  1.0000)and( -2.2000,  0.2592,  1.0000)..( -2.2000,  0.0000,  1.0000)..controls( -2.4618,  0.0000,  1.0000)and( -2.7220,  0.0000,  0.8922)..( -2.9071,  0.0000,  0.7071)..controls( -2.9071,  0.3425,  0.7071)and( -2.8169,  0.7961,  0.7071)..( -2.6858,  1.1125,  0.7071)--cycle;
\path[facetfront]( -2.9071,  0.0000,  0.7071)..controls( -2.7220,  0.0000,  0.8922)and( -2.4618,  0.0000,  1.0000)..( -2.2000,  0.0000,  1.0000)..controls( -2.2000, -0.2592,  1.0000)and( -2.1317, -0.6025,  1.0000)..( -2.0325, -0.8419,  1.0000)..controls( -2.2744, -0.9421,  1.0000)and( -2.5148, -1.0417,  0.8922)..( -2.6858, -1.1125,  0.7071)..controls( -2.8169, -0.7961,  0.7071)and( -2.9071, -0.3425,  0.7071)..( -2.9071,  0.0000,  0.7071)--cycle;
\path[facetfront]( -2.6858, -1.1125,  0.7071)..controls( -2.5148, -1.0417,  0.8922)and( -2.2744, -0.9421,  1.0000)..( -2.0325, -0.8419,  1.0000)..controls( -1.9334, -1.0814,  1.0000)and( -1.7389, -1.3724,  1.0000)..( -1.5556, -1.5556,  1.0000)..controls( -1.7408, -1.7408,  1.0000)and( -1.9247, -1.9247,  0.8922)..( -2.0556, -2.0556,  0.7071)..controls( -2.2978, -1.8135,  0.7071)and( -2.5548, -1.4289,  0.7071)..( -2.6858, -1.1125,  0.7071)--cycle;
\path[facetfront]( -2.0556, -2.0556,  0.7071)..controls( -1.9247, -1.9247,  0.8922)and( -1.7408, -1.7408,  1.0000)..( -1.5556, -1.5556,  1.0000)..controls( -1.3724, -1.7389,  1.0000)and( -1.0814, -1.9334,  1.0000)..( -0.8419, -2.0325,  1.0000)..controls( -0.9421, -2.2744,  1.0000)and( -1.0417, -2.5148,  0.8922)..( -1.1125, -2.6858,  0.7071)..controls( -1.4289, -2.5548,  0.7071)and( -1.8135, -2.2978,  0.7071)..( -2.0556, -2.0556,  0.7071)--cycle;
\path[facetfront]( -1.1125, -2.6858,  0.7071)..controls( -1.0417, -2.5148,  0.8922)and( -0.9421, -2.2744,  1.0000)..( -0.8419, -2.0325,  1.0000)..controls( -0.6025, -2.1317,  1.0000)and( -0.2592, -2.2000,  1.0000)..( -0.0000, -2.2000,  1.0000)..controls( -0.0000, -2.4618,  1.0000)and( -0.0000, -2.7220,  0.8922)..( -0.0000, -2.9071,  0.7071)..controls( -0.3425, -2.9071,  0.7071)and( -0.7961, -2.8169,  0.7071)..( -1.1125, -2.6858,  0.7071)--cycle;
\path[facetfront]( -0.0000, -2.9071,  0.7071)..controls( -0.0000, -2.7220,  0.8922)and( -0.0000, -2.4618,  1.0000)..( -0.0000, -2.2000,  1.0000)..controls(  0.2592, -2.2000,  1.0000)and(  0.6025, -2.1317,  1.0000)..(  0.8419, -2.0325,  1.0000)..controls(  0.9421, -2.2744,  1.0000)and(  1.0417, -2.5148,  0.8922)..(  1.1125, -2.6858,  0.7071)..controls(  0.7961, -2.8169,  0.7071)and(  0.3425, -2.9071,  0.7071)..( -0.0000, -2.9071,  0.7071)--cycle;
\path[facetfront](  1.1125, -2.6858,  0.7071)..controls(  1.0417, -2.5148,  0.8922)and(  0.9421, -2.2744,  1.0000)..(  0.8419, -2.0325,  1.0000)..controls(  1.0814, -1.9334,  1.0000)and(  1.3724, -1.7389,  1.0000)..(  1.5556, -1.5556,  1.0000)..controls(  1.7408, -1.7408,  1.0000)and(  1.9247, -1.9247,  0.8922)..(  2.0556, -2.0556,  0.7071)..controls(  1.8135, -2.2978,  0.7071)and(  1.4289, -2.5548,  0.7071)..(  1.1125, -2.6858,  0.7071)--cycle;
\path[facetfront](  2.0556, -2.0556,  0.7071)..controls(  1.9247, -1.9247,  0.8922)and(  1.7408, -1.7408,  1.0000)..(  1.5556, -1.5556,  1.0000)..controls(  1.7389, -1.3724,  1.0000)and(  1.9334, -1.0814,  1.0000)..(  2.0325, -0.8419,  1.0000)..controls(  2.2744, -0.9421,  1.0000)and(  2.5148, -1.0417,  0.8922)..(  2.6858, -1.1125,  0.7071)..controls(  2.5548, -1.4289,  0.7071)and(  2.2978, -1.8135,  0.7071)..(  2.0556, -2.0556,  0.7071)--cycle;
\path[facetfront](  2.6858, -1.1125,  0.7071)..controls(  2.5148, -1.0417,  0.8922)and(  2.2744, -0.9421,  1.0000)..(  2.0325, -0.8419,  1.0000)..controls(  2.1317, -0.6025,  1.0000)and(  2.2000, -0.2592,  1.0000)..(  2.2000, -0.0000,  1.0000)..controls(  2.4618, -0.0000,  1.0000)and(  2.7220, -0.0000,  0.8922)..(  2.9071, -0.0000,  0.7071)..controls(  2.9071, -0.3425,  0.7071)and(  2.8169, -0.7961,  0.7071)..(  2.6858, -1.1125,  0.7071)--cycle;
\path[facetfront](  2.2000,  0.0000,  1.0000)..controls(  1.9382,  0.0000,  1.0000)and(  1.6780,  0.0000,  0.8922)..(  1.4929,  0.0000,  0.7071)..controls(  1.4929,  0.1759,  0.7071)and(  1.4466,  0.4088,  0.7071)..(  1.3793,  0.5713,  0.7071)..controls(  1.5503,  0.6421,  0.8922)and(  1.7907,  0.7417,  1.0000)..(  2.0325,  0.8419,  1.0000)..controls(  2.1317,  0.6025,  1.0000)and(  2.2000,  0.2592,  1.0000)..(  2.2000,  0.0000,  1.0000)--cycle;
\path[facetfront](  2.0325,  0.8419,  1.0000)..controls(  1.7907,  0.7417,  1.0000)and(  1.5503,  0.6421,  0.8922)..(  1.3793,  0.5713,  0.7071)..controls(  1.3119,  0.7338,  0.7071)and(  1.1800,  0.9313,  0.7071)..(  1.0556,  1.0556,  0.7071)..controls(  1.1865,  1.1865,  0.8922)and(  1.3705,  1.3705,  1.0000)..(  1.5556,  1.5556,  1.0000)..controls(  1.7389,  1.3724,  1.0000)and(  1.9334,  1.0814,  1.0000)..(  2.0325,  0.8419,  1.0000)--cycle;
\path[facetfront](  1.5556,  1.5556,  1.0000)..controls(  1.3705,  1.3705,  1.0000)and(  1.1865,  1.1865,  0.8922)..(  1.0556,  1.0556,  0.7071)..controls(  0.9313,  1.1800,  0.7071)and(  0.7338,  1.3119,  0.7071)..(  0.5713,  1.3793,  0.7071)..controls(  0.6421,  1.5503,  0.8922)and(  0.7417,  1.7907,  1.0000)..(  0.8419,  2.0325,  1.0000)..controls(  1.0814,  1.9334,  1.0000)and(  1.3724,  1.7389,  1.0000)..(  1.5556,  1.5556,  1.0000)--cycle;
\path[facetfront](  0.8419,  2.0325,  1.0000)..controls(  0.7417,  1.7907,  1.0000)and(  0.6421,  1.5503,  0.8922)..(  0.5713,  1.3793,  0.7071)..controls(  0.4088,  1.4466,  0.7071)and(  0.1759,  1.4929,  0.7071)..(  0.0000,  1.4929,  0.7071)..controls(  0.0000,  1.6780,  0.8922)and(  0.0000,  1.9382,  1.0000)..(  0.0000,  2.2000,  1.0000)..controls(  0.2592,  2.2000,  1.0000)and(  0.6025,  2.1317,  1.0000)..(  0.8419,  2.0325,  1.0000)--cycle;
\path[facetfront](  0.0000,  2.2000,  1.0000)..controls(  0.0000,  1.9382,  1.0000)and(  0.0000,  1.6780,  0.8922)..(  0.0000,  1.4929,  0.7071)..controls( -0.1759,  1.4929,  0.7071)and( -0.4088,  1.4466,  0.7071)..( -0.5713,  1.3793,  0.7071)..controls( -0.6421,  1.5503,  0.8922)and( -0.7417,  1.7907,  1.0000)..( -0.8419,  2.0325,  1.0000)..controls( -0.6025,  2.1317,  1.0000)and( -0.2592,  2.2000,  1.0000)..(  0.0000,  2.2000,  1.0000)--cycle;
\path[facetfront]( -0.8419,  2.0325,  1.0000)..controls( -0.7417,  1.7907,  1.0000)and( -0.6421,  1.5503,  0.8922)..( -0.5713,  1.3793,  0.7071)..controls( -0.7338,  1.3119,  0.7071)and( -0.9313,  1.1800,  0.7071)..( -1.0556,  1.0556,  0.7071)..controls( -1.1865,  1.1865,  0.8922)and( -1.3705,  1.3705,  1.0000)..( -1.5556,  1.5556,  1.0000)..controls( -1.3724,  1.7389,  1.0000)and( -1.0814,  1.9334,  1.0000)..( -0.8419,  2.0325,  1.0000)--cycle;
\path[facetfront]( -1.5556,  1.5556,  1.0000)..controls( -1.3705,  1.3705,  1.0000)and( -1.1865,  1.1865,  0.8922)..( -1.0556,  1.0556,  0.7071)..controls( -1.1800,  0.9313,  0.7071)and( -1.3119,  0.7338,  0.7071)..( -1.3793,  0.5713,  0.7071)..controls( -1.5503,  0.6421,  0.8922)and( -1.7907,  0.7417,  1.0000)..( -2.0325,  0.8419,  1.0000)..controls( -1.9334,  1.0814,  1.0000)and( -1.7389,  1.3724,  1.0000)..( -1.5556,  1.5556,  1.0000)--cycle;
\path[facetfront]( -2.0325,  0.8419,  1.0000)..controls( -1.7907,  0.7417,  1.0000)and( -1.5503,  0.6421,  0.8922)..( -1.3793,  0.5713,  0.7071)..controls( -1.4466,  0.4088,  0.7071)and( -1.4929,  0.1759,  0.7071)..( -1.4929,  0.0000,  0.7071)..controls( -1.6780,  0.0000,  0.8922)and( -1.9382,  0.0000,  1.0000)..( -2.2000,  0.0000,  1.0000)..controls( -2.2000,  0.2592,  1.0000)and( -2.1317,  0.6025,  1.0000)..( -2.0325,  0.8419,  1.0000)--cycle;
\path[facetfront]( -2.2000,  0.0000,  1.0000)..controls( -1.9382,  0.0000,  1.0000)and( -1.6780,  0.0000,  0.8922)..( -1.4929,  0.0000,  0.7071)..controls( -1.4929, -0.1759,  0.7071)and( -1.4466, -0.4088,  0.7071)..( -1.3793, -0.5713,  0.7071)..controls( -1.5503, -0.6421,  0.8922)and( -1.7907, -0.7417,  1.0000)..( -2.0325, -0.8419,  1.0000)..controls( -2.1317, -0.6025,  1.0000)and( -2.2000, -0.2592,  1.0000)..( -2.2000,  0.0000,  1.0000)--cycle;
\path[facetfront]( -2.0325, -0.8419,  1.0000)..controls( -1.7907, -0.7417,  1.0000)and( -1.5503, -0.6421,  0.8922)..( -1.3793, -0.5713,  0.7071)..controls( -1.3119, -0.7338,  0.7071)and( -1.1800, -0.9313,  0.7071)..( -1.0556, -1.0556,  0.7071)..controls( -1.1865, -1.1865,  0.8922)and( -1.3705, -1.3705,  1.0000)..( -1.5556, -1.5556,  1.0000)..controls( -1.7389, -1.3724,  1.0000)and( -1.9334, -1.0814,  1.0000)..( -2.0325, -0.8419,  1.0000)--cycle;
\path[facetfront]( -1.5556, -1.5556,  1.0000)..controls( -1.3705, -1.3705,  1.0000)and( -1.1865, -1.1865,  0.8922)..( -1.0556, -1.0556,  0.7071)..controls( -0.9313, -1.1800,  0.7071)and( -0.7338, -1.3119,  0.7071)..( -0.5713, -1.3793,  0.7071)..controls( -0.6421, -1.5503,  0.8922)and( -0.7417, -1.7907,  1.0000)..( -0.8419, -2.0325,  1.0000)..controls( -1.0814, -1.9334,  1.0000)and( -1.3724, -1.7389,  1.0000)..( -1.5556, -1.5556,  1.0000)--cycle;
\path[facetfront]( -0.8419, -2.0325,  1.0000)..controls( -0.7417, -1.7907,  1.0000)and( -0.6421, -1.5503,  0.8922)..( -0.5713, -1.3793,  0.7071)..controls( -0.4088, -1.4466,  0.7071)and( -0.1759, -1.4929,  0.7071)..( -0.0000, -1.4929,  0.7071)..controls( -0.0000, -1.6780,  0.8922)and( -0.0000, -1.9382,  1.0000)..( -0.0000, -2.2000,  1.0000)..controls( -0.2592, -2.2000,  1.0000)and( -0.6025, -2.1317,  1.0000)..( -0.8419, -2.0325,  1.0000)--cycle;
\path[facetfront]( -0.0000, -2.2000,  1.0000)..controls( -0.0000, -1.9382,  1.0000)and( -0.0000, -1.6780,  0.8922)..( -0.0000, -1.4929,  0.7071)..controls(  0.1759, -1.4929,  0.7071)and(  0.4088, -1.4466,  0.7071)..(  0.5713, -1.3793,  0.7071)..controls(  0.6421, -1.5503,  0.8922)and(  0.7417, -1.7907,  1.0000)..(  0.8419, -2.0325,  1.0000)..controls(  0.6025, -2.1317,  1.0000)and(  0.2592, -2.2000,  1.0000)..( -0.0000, -2.2000,  1.0000)--cycle;
\path[facetfront](  0.8419, -2.0325,  1.0000)..controls(  0.7417, -1.7907,  1.0000)and(  0.6421, -1.5503,  0.8922)..(  0.5713, -1.3793,  0.7071)..controls(  0.7338, -1.3119,  0.7071)and(  0.9313, -1.1800,  0.7071)..(  1.0556, -1.0556,  0.7071)..controls(  1.1865, -1.1865,  0.8922)and(  1.3705, -1.3705,  1.0000)..(  1.5556, -1.5556,  1.0000)..controls(  1.3724, -1.7389,  1.0000)and(  1.0814, -1.9334,  1.0000)..(  0.8419, -2.0325,  1.0000)--cycle;
\path[facetfront](  1.5556, -1.5556,  1.0000)..controls(  1.3705, -1.3705,  1.0000)and(  1.1865, -1.1865,  0.8922)..(  1.0556, -1.0556,  0.7071)..controls(  1.1800, -0.9313,  0.7071)and(  1.3119, -0.7338,  0.7071)..(  1.3793, -0.5713,  0.7071)..controls(  1.5503, -0.6421,  0.8922)and(  1.7907, -0.7417,  1.0000)..(  2.0325, -0.8419,  1.0000)..controls(  1.9334, -1.0814,  1.0000)and(  1.7389, -1.3724,  1.0000)..(  1.5556, -1.5556,  1.0000)--cycle;
\path[facetfront](  2.0325, -0.8419,  1.0000)..controls(  1.7907, -0.7417,  1.0000)and(  1.5503, -0.6421,  0.8922)..(  1.3793, -0.5713,  0.7071)..controls(  1.4466, -0.4088,  0.7071)and(  1.4929, -0.1759,  0.7071)..(  1.4929, -0.0000,  0.7071)..controls(  1.6780, -0.0000,  0.8922)and(  1.9382, -0.0000,  1.0000)..(  2.2000, -0.0000,  1.0000)..controls(  2.2000, -0.2592,  1.0000)and(  2.1317, -0.6025,  1.0000)..(  2.0325, -0.8419,  1.0000)--cycle;
\path[facetfront](  1.3793,  0.5713,  0.7071)..controls(  1.2082,  0.5005,  0.5220)and(  1.1087,  0.4592,  0.2618)..(  1.1087,  0.4592,  0.0000)..controls(  1.0546,  0.5898,  0.0000)and(  0.9485,  0.7486,  0.0000)..(  0.8485,  0.8485,  0.0000)..controls(  0.8485,  0.8485,  0.2618)and(  0.9247,  0.9247,  0.5220)..(  1.0556,  1.0556,  0.7071)..controls(  1.1800,  0.9313,  0.7071)and(  1.3119,  0.7338,  0.7071)..(  1.3793,  0.5713,  0.7071)--cycle;
\path[facetfront](  1.0556,  1.0556,  0.7071)..controls(  0.9247,  0.9247,  0.5220)and(  0.8485,  0.8485,  0.2618)..(  0.8485,  0.8485,  0.0000)..controls(  0.7486,  0.9485,  0.0000)and(  0.5898,  1.0546,  0.0000)..(  0.4592,  1.1087,  0.0000)..controls(  0.4592,  1.1087,  0.2618)and(  0.5005,  1.2082,  0.5220)..(  0.5713,  1.3793,  0.7071)..controls(  0.7338,  1.3119,  0.7071)and(  0.9313,  1.1800,  0.7071)..(  1.0556,  1.0556,  0.7071)--cycle;
\path[facetfront](  0.5713,  1.3793,  0.7071)..controls(  0.5005,  1.2082,  0.5220)and(  0.4592,  1.1087,  0.2618)..(  0.4592,  1.1087,  0.0000)..controls(  0.3286,  1.1628,  0.0000)and(  0.1414,  1.2000,  0.0000)..(  0.0000,  1.2000,  0.0000)..controls(  0.0000,  1.2000,  0.2618)and(  0.0000,  1.3078,  0.5220)..(  0.0000,  1.4929,  0.7071)..controls(  0.1759,  1.4929,  0.7071)and(  0.4088,  1.4466,  0.7071)..(  0.5713,  1.3793,  0.7071)--cycle;
\path[facetfront](  0.0000,  1.4929,  0.7071)..controls(  0.0000,  1.3078,  0.5220)and(  0.0000,  1.2000,  0.2618)..(  0.0000,  1.2000,  0.0000)..controls( -0.1414,  1.2000,  0.0000)and( -0.3286,  1.1628,  0.0000)..( -0.4592,  1.1087,  0.0000)..controls( -0.4592,  1.1087,  0.2618)and( -0.5005,  1.2082,  0.5220)..( -0.5713,  1.3793,  0.7071)..controls( -0.4088,  1.4466,  0.7071)and( -0.1759,  1.4929,  0.7071)..(  0.0000,  1.4929,  0.7071)--cycle;
\path[facetfront]( -0.5713,  1.3793,  0.7071)..controls( -0.5005,  1.2082,  0.5220)and( -0.4592,  1.1087,  0.2618)..( -0.4592,  1.1087,  0.0000)..controls( -0.5898,  1.0546,  0.0000)and( -0.7486,  0.9485,  0.0000)..( -0.8485,  0.8485,  0.0000)..controls( -0.8485,  0.8485,  0.2618)and( -0.9247,  0.9247,  0.5220)..( -1.0556,  1.0556,  0.7071)..controls( -0.9313,  1.1800,  0.7071)and( -0.7338,  1.3119,  0.7071)..( -0.5713,  1.3793,  0.7071)--cycle;
\path[facetfront]( -1.0556,  1.0556,  0.7071)..controls( -0.9247,  0.9247,  0.5220)and( -0.8485,  0.8485,  0.2618)..( -0.8485,  0.8485,  0.0000)..controls( -0.9485,  0.7486,  0.0000)and( -1.0546,  0.5898,  0.0000)..( -1.1087,  0.4592,  0.0000)..controls( -1.1087,  0.4592,  0.2618)and( -1.2082,  0.5005,  0.5220)..( -1.3793,  0.5713,  0.7071)..controls( -1.3119,  0.7338,  0.7071)and( -1.1800,  0.9313,  0.7071)..( -1.0556,  1.0556,  0.7071)--cycle;
\path[facetfront]( -1.3793,  0.5713,  0.7071)..controls( -1.2082,  0.5005,  0.5220)and( -1.1087,  0.4592,  0.2618)..( -1.1087,  0.4592,  0.0000)..controls( -1.1628,  0.3286,  0.0000)and( -1.2000,  0.1414,  0.0000)..( -1.2000,  0.0000,  0.0000)..controls( -1.2000,  0.0000,  0.2618)and( -1.3078,  0.0000,  0.5220)..( -1.4929,  0.0000,  0.7071)..controls( -1.4929,  0.1759,  0.7071)and( -1.4466,  0.4088,  0.7071)..( -1.3793,  0.5713,  0.7071)--cycle;
\path[facetfront]( -1.3793, -0.5713,  0.7071)..controls( -1.2438, -0.5152,  0.5605)and( -1.1532, -0.4777,  0.3668)..( -1.1213, -0.4645,  0.1625)..controls( -1.1061, -0.4908,  0.1366)and( -0.9656, -0.9055,  0.4788)..( -0.9262, -0.9262,  0.4536)..controls( -0.9597, -0.9597,  0.5470)and( -1.0033, -1.0033,  0.6331)..( -1.0556, -1.0556,  0.7071)..controls( -1.1800, -0.9313,  0.7071)and( -1.3119, -0.7338,  0.7071)..( -1.3793, -0.5713,  0.7071)--cycle;
\path[facetfront]( -1.0556, -1.0556,  0.7071)..controls( -1.0009, -1.0009,  0.6297)and( -0.9558, -0.9558,  0.5392)..( -0.9217, -0.9217,  0.4410)..controls( -0.8539, -0.9592,  0.3960)and( -0.6327, -1.2809,  0.6327)..( -0.5367, -1.2957,  0.6028)..controls( -0.5473, -1.3213,  0.6396)and( -0.5588, -1.3492,  0.6745)..( -0.5713, -1.3793,  0.7071)..controls( -0.7338, -1.3119,  0.7071)and( -0.9313, -1.1800,  0.7071)..( -1.0556, -1.0556,  0.7071)--cycle;
\path[facetfront]( -0.5713, -1.3793,  0.7071)..controls( -0.5584, -1.3481,  0.6734)and( -0.5465, -1.3194,  0.6372)..( -0.5356, -1.2931,  0.5990)..controls( -0.4063, -1.3142,  0.5577)and( -0.1424, -1.4637,  0.6763)..( -0.0000, -1.4510,  0.6625)..controls( -0.0000, -1.4645,  0.6778)and( -0.0000, -1.4785,  0.6927)..( -0.0000, -1.4929,  0.7071)..controls( -0.1759, -1.4929,  0.7071)and( -0.4088, -1.4466,  0.7071)..( -0.5713, -1.3793,  0.7071)--cycle;
\path[facetfront]( -0.0000, -1.4929,  0.7071)..controls( -0.0000, -1.4785,  0.6927)and( -0.0000, -1.4645,  0.6778)..( -0.0000, -1.4510,  0.6625)..controls(  0.1694, -1.4365,  0.6455)and(  0.3932, -1.3920,  0.6455)..(  0.5553, -1.3406,  0.6625)..controls(  0.5604, -1.3530,  0.6778)and(  0.5658, -1.3659,  0.6927)..(  0.5713, -1.3793,  0.7071)..controls(  0.4088, -1.4466,  0.7071)and(  0.1759, -1.4929,  0.7071)..( -0.0000, -1.4929,  0.7071)--cycle;
\path[facetfront](  0.5713, -1.3793,  0.7071)..controls(  0.5658, -1.3659,  0.6927)and(  0.5604, -1.3530,  0.6778)..(  0.5553, -1.3406,  0.6625)..controls(  0.6917, -1.2978,  0.6763)and(  0.8783, -1.0586,  0.5577)..(  0.9897, -0.9897,  0.5990)..controls(  1.0098, -1.0098,  0.6372)and(  1.0318, -1.0318,  0.6734)..(  1.0556, -1.0556,  0.7071)..controls(  0.9313, -1.1800,  0.7071)and(  0.7338, -1.3119,  0.7071)..(  0.5713, -1.3793,  0.7071)--cycle;
\path[facetfront](  1.0556, -1.0556,  0.7071)..controls(  1.0326, -1.0326,  0.6745)and(  1.0113, -1.0113,  0.6396)..(  0.9917, -0.9917,  0.6028)..controls(  1.0747, -0.9413,  0.6327)and(  1.1560, -0.5594,  0.3960)..(  1.2043, -0.4988,  0.4410)..controls(  1.2488, -0.5173,  0.5392)and(  1.3078, -0.5417,  0.6297)..(  1.3793, -0.5713,  0.7071)..controls(  1.3119, -0.7338,  0.7071)and(  1.1800, -0.9313,  0.7071)..(  1.0556, -1.0556,  0.7071)--cycle;
\path[facetfront](  1.3793, -0.5713,  0.7071)..controls(  1.3108, -0.5430,  0.6331)and(  1.2539, -0.5194,  0.5470)..(  1.2101, -0.5012,  0.4536)..controls(  1.2386, -0.4671,  0.4788)and(  1.2097, -0.0301,  0.1366)..(  1.2137, -0.0000,  0.1625)..controls(  1.2482, -0.0000,  0.3668)and(  1.3463, -0.0000,  0.5605)..(  1.4929, -0.0000,  0.7071)..controls(  1.4929, -0.1759,  0.7071)and(  1.4466, -0.4088,  0.7071)..(  1.3793, -0.5713,  0.7071)--cycle;
\path[facetfront](  1.1725,  0.2443,  0.0000)..controls(  1.1712,  0.2500,  0.0053)and(  1.1273,  0.4599, -0.1835)..(  1.1239,  0.4655, -0.1779)..controls(  1.1138,  0.4614, -0.1193)and(  1.1087,  0.4592, -0.0597)..(  1.1087,  0.4592,  0.0000)..controls(  1.1345,  0.3968,  0.0000)and(  1.1565,  0.3214,  0.0000)..(  1.1725,  0.2443,  0.0000)--cycle;
\path[facetfront](  1.1087,  0.4592,  0.0000)..controls(  1.1087,  0.4592, -0.0550)and(  1.1130,  0.4610, -0.1099)..(  1.1216,  0.4646, -0.1640)..controls(  1.1064,  0.4909, -0.1381)and(  0.9661,  0.9060, -0.4802)..(  0.9267,  0.9267, -0.4550)..controls(  0.8761,  0.8761, -0.3146)and(  0.8485,  0.8485, -0.1576)..(  0.8485,  0.8485,  0.0000)..controls(  0.9485,  0.7486,  0.0000)and(  1.0546,  0.5898,  0.0000)..(  1.1087,  0.4592,  0.0000)--cycle;
\path[facetfront](  0.8485,  0.8485,  0.0000)..controls(  0.8485,  0.8485, -0.1529)and(  0.8745,  0.8745, -0.3052)..(  0.9222,  0.9222, -0.4424)..controls(  0.8544,  0.9597, -0.3974)and(  0.6330,  1.2818, -0.6340)..(  0.5371,  1.2966, -0.6040)..controls(  0.4874,  1.1766, -0.4314)and(  0.4592,  1.1087, -0.2162)..(  0.4592,  1.1087,  0.0000)..controls(  0.5898,  1.0546,  0.0000)and(  0.7486,  0.9485,  0.0000)..(  0.8485,  0.8485,  0.0000)--cycle;
\path[facetfront](  0.4592,  1.1087,  0.0000)..controls(  0.4592,  1.1087, -0.2147)and(  0.4870,  1.1756, -0.4283)..(  0.5360,  1.2940, -0.6003)..controls(  0.4067,  1.3150, -0.5590)and(  0.1424,  1.4647, -0.6774)..(  0.0000,  1.4520, -0.6636)..controls(  0.0000,  1.2920, -0.4824)and(  0.0000,  1.2000, -0.2419)..(  0.0000,  1.2000,  0.0000)..controls(  0.1414,  1.2000,  0.0000)and(  0.3286,  1.1628,  0.0000)..(  0.4592,  1.1087,  0.0000)--cycle;
\path[facetfront](  0.0000,  1.2000,  0.0000)..controls(  0.0000,  1.2000, -0.2419)and(  0.0000,  1.2920, -0.4824)..(  0.0000,  1.4520, -0.6636)..controls( -0.1694,  1.4375, -0.6467)and( -0.3936,  1.3929, -0.6467)..( -0.5557,  1.3415, -0.6636)..controls( -0.4944,  1.1937, -0.4824)and( -0.4592,  1.1087, -0.2419)..( -0.4592,  1.1087,  0.0000)..controls( -0.3286,  1.1628,  0.0000)and( -0.1414,  1.2000,  0.0000)..(  0.0000,  1.2000,  0.0000)--cycle;
\path[facetfront]( -0.4592,  1.1087,  0.0000)..controls( -0.4592,  1.1087, -0.2419)and( -0.4944,  1.1937, -0.4824)..( -0.5557,  1.3415, -0.6636)..controls( -0.6921,  1.2987, -0.6774)and( -0.8790,  1.0593, -0.5590)..( -0.9904,  0.9904, -0.6003)..controls( -0.8998,  0.8998, -0.4283)and( -0.8485,  0.8485, -0.2147)..( -0.8485,  0.8485,  0.0000)..controls( -0.7486,  0.9485,  0.0000)and( -0.5898,  1.0546,  0.0000)..( -0.4592,  1.1087,  0.0000)--cycle;
\path[facetfront]( -0.8485,  0.8485,  0.0000)..controls( -0.8485,  0.8485, -0.2162)and( -0.9005,  0.9005, -0.4314)..( -0.9924,  0.9924, -0.6040)..controls( -1.0753,  0.9420, -0.6340)and( -1.1566,  0.5597, -0.3974)..( -1.2049,  0.4991, -0.4424)..controls( -1.1426,  0.4733, -0.3052)and( -1.1087,  0.4592, -0.1529)..( -1.1087,  0.4592,  0.0000)..controls( -1.0546,  0.5898,  0.0000)and( -0.9485,  0.7486,  0.0000)..( -0.8485,  0.8485,  0.0000)--cycle;
\path[facetfront]( -1.1087,  0.4592,  0.0000)..controls( -1.1087,  0.4592, -0.1576)and( -1.1447,  0.4742, -0.3146)..( -1.2108,  0.5015, -0.4550)..controls( -1.2393,  0.4674, -0.4802)and( -1.2100,  0.0301, -0.1381)..( -1.2140,  0.0000, -0.1640)..controls( -1.2048,  0.0000, -0.1099)and( -1.2000,  0.0000, -0.0550)..( -1.2000,  0.0000,  0.0000)..controls( -1.2000,  0.1414,  0.0000)and( -1.1628,  0.3286,  0.0000)..( -1.1087,  0.4592,  0.0000)--cycle;
\path[facetfront]( -1.2000,  0.0000,  0.0000)..controls( -1.2000,  0.0000, -0.0597)and( -1.2056,  0.0000, -0.1193)..( -1.2164,  0.0000, -0.1779)..controls( -1.2175, -0.0067, -0.1837)and( -1.1776, -0.2180,  0.0055)..( -1.1765, -0.2239,  0.0000)..controls( -1.1914, -0.1463,  0.0000)and( -1.2000, -0.0679,  0.0000)..( -1.2000,  0.0000,  0.0000)--cycle;
\path[facetfront](  3.1858,  0.0000, -0.1656)..controls(  3.1952,  0.0000, -0.1109)and(  3.2000,  0.0000, -0.0555)..(  3.2000,  0.0000, -0.0000)..controls(  3.2000,  0.1810, -0.0000)and(  3.1771,  0.3901, -0.0000)..(  3.1374,  0.5972, -0.0000)..controls(  3.1694,  0.4244, -0.1582)and(  3.2115,  0.1780, -0.0129)..(  3.1858,  0.0000, -0.1656)--cycle;
\path[facetfront]( -3.1265, -0.6515, -0.0000)..controls( -3.0841, -0.8571, -0.0000)and( -3.0254, -1.0581, -0.0000)..( -2.9564, -1.2246, -0.0000)..controls( -2.9564, -1.2246, -0.0555)and( -2.9519, -1.2227, -0.1109)..( -2.9433, -1.2191, -0.1656)..controls( -3.0349, -1.0651, -0.0134)and( -3.0902, -0.8228, -0.1577)..( -3.1265, -0.6515, -0.0000)--cycle;
\path[facetfront](  2.9420,  1.2186,  0.1733)..controls(  2.9087,  1.2048,  0.3736)and(  2.8189,  1.1676,  0.5631)..(  2.6858,  1.1125,  0.7071)..controls(  2.5548,  1.4289,  0.7071)and(  2.2978,  1.8135,  0.7071)..(  2.0556,  2.0556,  0.7071)..controls(  2.1093,  2.1093,  0.6312)and(  2.1538,  2.1538,  0.5427)..(  2.1876,  2.1876,  0.4466)..controls(  2.5385,  1.9962,  0.2160)and(  2.7543,  1.5264,  0.4780)..(  2.9420,  1.2186,  0.1733)--cycle;
\path[facetfront](  2.1856,  2.1856,  0.4522)..controls(  2.1520,  2.1520,  0.5461)and(  2.1083,  2.1083,  0.6327)..(  2.0556,  2.0556,  0.7071)..controls(  1.8135,  2.2978,  0.7071)and(  1.4289,  2.5548,  0.7071)..(  1.1125,  2.6858,  0.7071)..controls(  1.1251,  2.7163,  0.6742)and(  1.1368,  2.7444,  0.6388)..(  1.1475,  2.7702,  0.6015)..controls(  1.5359,  2.7078,  0.4782)and(  1.8609,  2.3585,  0.6620)..(  2.1856,  2.1856,  0.4522)--cycle;
\path[facetfront](  1.1471,  2.7693,  0.6028)..controls(  1.1365,  2.7438,  0.6396)and(  1.1250,  2.7159,  0.6745)..(  1.1125,  2.6858,  0.7071)..controls(  0.7961,  2.8169,  0.7071)and(  0.3425,  2.9071,  0.7071)..(  0.0000,  2.9071,  0.7071)..controls(  0.0000,  2.9212,  0.6930)and(  0.0000,  2.9348,  0.6785)..(  0.0000,  2.9480,  0.6636)..controls(  0.3705,  2.9800,  0.6268)and(  0.7888,  2.8254,  0.7153)..(  1.1471,  2.7693,  0.6028)--cycle;
\path[facetfront](  0.0000,  2.9480,  0.6636)..controls(  0.0000,  2.9348,  0.6785)and(  0.0000,  2.9212,  0.6930)..(  0.0000,  2.9071,  0.7071)..controls( -0.3425,  2.9071,  0.7071)and( -0.7961,  2.8169,  0.7071)..( -1.1125,  2.6858,  0.7071)..controls( -1.1179,  2.6988,  0.6930)and( -1.1231,  2.7114,  0.6785)..( -1.1281,  2.7236,  0.6636)..controls( -0.7983,  2.8274,  0.6974)and( -0.3445,  2.9177,  0.6974)..(  0.0000,  2.9480,  0.6636)--cycle;
\path[facetfront]( -1.1281,  2.7236,  0.6636)..controls( -1.1231,  2.7114,  0.6785)and( -1.1179,  2.6988,  0.6930)..( -1.1125,  2.6858,  0.7071)..controls( -1.4289,  2.5548,  0.7071)and( -1.8135,  2.2978,  0.7071)..( -2.0556,  2.0556,  0.7071)..controls( -2.0787,  2.0787,  0.6745)and( -2.1000,  2.1000,  0.6396)..( -2.1196,  2.1196,  0.6028)..controls( -1.8100,  2.3085,  0.7153)and( -1.4827,  2.6114,  0.6268)..( -1.1281,  2.7236,  0.6636)--cycle;
\path[facetfront]( -2.1202,  2.1202,  0.6015)..controls( -2.1005,  2.1005,  0.6388)and( -2.0789,  2.0789,  0.6742)..( -2.0556,  2.0556,  0.7071)..controls( -2.2978,  1.8135,  0.7071)and( -2.5548,  1.4289,  0.7071)..( -2.6858,  1.1125,  0.7071)..controls( -2.7546,  1.1410,  0.6327)and( -2.8118,  1.1647,  0.5461)..( -2.8556,  1.1828,  0.4522)..controls( -2.6218,  1.4668,  0.6620)and( -2.4552,  1.9140,  0.4782)..( -2.1202,  2.1202,  0.6015)--cycle;
\path[facetfront]( -2.8582,  1.1839,  0.4466)..controls( -2.8140,  1.1656,  0.5427)and( -2.7559,  1.1415,  0.6312)..( -2.6858,  1.1125,  0.7071)..controls( -2.8169,  0.7961,  0.7071)and( -2.9071,  0.3425,  0.7071)..( -2.9071,  0.0000,  0.7071)..controls( -3.0511,  0.0000,  0.5631)and( -3.1483,  0.0000,  0.3736)..( -3.1844,  0.0000,  0.1733)..controls( -3.1288,  0.3561,  0.4780)and( -3.1092,  0.8728,  0.2160)..( -2.8582,  1.1839,  0.4466)--cycle;
\path[facetfront]( -2.9423, -1.2187, -0.1718)..controls( -2.9516, -1.2226, -0.1151)and( -2.9564, -1.2246, -0.0576)..( -2.9564, -1.2246, -0.0000)..controls( -2.8121, -1.5729, -0.0000)and( -2.5293, -1.9962, -0.0000)..( -2.2627, -2.2627, -0.0000)..controls( -2.2627, -2.2627, -0.1539)and( -2.2364, -2.2364, -0.3073)..( -2.1881, -2.1881, -0.4452)..controls( -2.5390, -1.9967, -0.2145)and( -2.7546, -1.5265, -0.4765)..( -2.9423, -1.2187, -0.1718)--cycle;
\path[facetfront]( -2.1861, -2.1861, -0.4508)..controls( -2.2357, -2.2357, -0.3115)and( -2.2627, -2.2627, -0.1560)..( -2.2627, -2.2627, -0.0000)..controls( -1.9962, -2.5293, -0.0000)and( -1.5729, -2.8121, -0.0000)..( -1.2246, -2.9564, -0.0000)..controls( -1.2246, -2.9564, -0.2147)and( -1.1969, -2.8895, -0.4283)..( -1.1478, -2.7711, -0.6003)..controls( -1.5363, -2.7087, -0.4770)and( -1.8614, -2.3590, -0.6606)..( -2.1861, -2.1861, -0.4508)--cycle;
\path[facetfront]( -1.1475, -2.7702, -0.6015)..controls( -1.1967, -2.8891, -0.4293)and( -1.2246, -2.9564, -0.2152)..( -1.2246, -2.9564, -0.0000)..controls( -0.8763, -3.1007, -0.0000)and( -0.3770, -3.2000, -0.0000)..( -0.0000, -3.2000, -0.0000)..controls( -0.0000, -3.2000, -0.2414)and( -0.0000, -3.1084, -0.4814)..( -0.0000, -2.9490, -0.6625)..controls( -0.3705, -2.9810, -0.6256)and( -0.7891, -2.8263, -0.7140)..( -1.1475, -2.7702, -0.6015)--cycle;
\path[facetfront]( -0.0000, -2.9490, -0.6625)..controls( -0.0000, -3.1084, -0.4814)and( -0.0000, -3.2000, -0.2414)..( -0.0000, -3.2000, -0.0000)..controls(  0.3770, -3.2000, -0.0000)and(  0.8763, -3.1007, -0.0000)..(  1.2246, -2.9564, -0.0000)..controls(  1.2246, -2.9564, -0.2414)and(  1.1895, -2.8718, -0.4814)..(  1.1285, -2.7245, -0.6625)..controls(  0.7986, -2.8284, -0.6962)and(  0.3445, -2.9187, -0.6962)..( -0.0000, -2.9490, -0.6625)--cycle;
\path[facetfront](  1.1285, -2.7245, -0.6625)..controls(  1.1895, -2.8718, -0.4814)and(  1.2246, -2.9564, -0.2414)..(  1.2246, -2.9564, -0.0000)..controls(  1.5729, -2.8121, -0.0000)and(  1.9962, -2.5293, -0.0000)..(  2.2627, -2.2627, -0.0000)..controls(  2.2627, -2.2627, -0.2152)and(  2.2112, -2.2112, -0.4293)..(  2.1202, -2.1202, -0.6015)..controls(  1.8106, -2.3091, -0.7140)and(  1.4831, -2.6123, -0.6256)..(  1.1285, -2.7245, -0.6625)--cycle;
\path[facetfront](  2.1209, -2.1209, -0.6003)..controls(  2.2115, -2.2115, -0.4283)and(  2.2627, -2.2627, -0.2147)..(  2.2627, -2.2627, -0.0000)..controls(  2.5293, -1.9962, -0.0000)and(  2.8121, -1.5729, -0.0000)..(  2.9564, -1.2246, -0.0000)..controls(  2.9564, -1.2246, -0.1560)and(  2.9210, -1.2099, -0.3115)..(  2.8563, -1.1831, -0.4508)..controls(  2.6224, -1.4671, -0.6606)and(  2.4559, -1.9146, -0.4770)..(  2.1209, -2.1209, -0.6003)--cycle;
\path[facetfront](  2.8589, -1.1842, -0.4452)..controls(  2.9220, -1.2103, -0.3073)and(  2.9564, -1.2246, -0.1539)..(  2.9564, -1.2246, -0.0000)..controls(  3.1007, -0.8763, -0.0000)and(  3.2000, -0.3770, -0.0000)..(  3.2000, -0.0000, -0.0000)..controls(  3.2000, -0.0000, -0.0576)and(  3.1948, -0.0000, -0.1151)..(  3.1847, -0.0000, -0.1718)..controls(  3.1290, -0.3561, -0.4765)and(  3.1098, -0.8731, -0.2145)..(  2.8589, -1.1842, -0.4452)--cycle;
\path[facetfront]( -3.1855,  0.0000,  0.1671)..controls( -3.1503,  0.0000,  0.3697)and( -3.0526,  0.0000,  0.5616)..( -2.9071,  0.0000,  0.7071)..controls( -2.9071, -0.3425,  0.7071)and( -2.8169, -0.7961,  0.7071)..( -2.6858, -1.1125,  0.7071)..controls( -2.8568, -1.1833,  0.5220)and( -2.9564, -1.2246,  0.2618)..( -2.9564, -1.2246,  0.0000)..controls( -3.0314, -1.0435,  0.0000)and( -3.0943, -0.8215,  0.0000)..( -3.1374, -0.5972,  0.0000)..controls( -3.1695, -0.4240,  0.1586)and( -3.2113, -0.1784,  0.0142)..( -3.1855,  0.0000,  0.1671)--cycle;
\path[facetfront](  1.4929,  0.0000,  0.7071)..controls(  1.3496,  0.0000,  0.5638)and(  1.2527,  0.0000,  0.3756)..(  1.2162,  0.0000,  0.1764)..controls(  1.2173,  0.0071,  0.1824)and(  1.1777,  0.2176, -0.0058)..(  1.1765,  0.2239,  0.0000)..controls(  1.1604,  0.3081,  0.0000)and(  1.1368,  0.3913,  0.0000)..(  1.1087,  0.4592,  0.0000)..controls(  1.1087,  0.4592,  0.2618)and(  1.2082,  0.5005,  0.5220)..(  1.3793,  0.5713,  0.7071)..controls(  1.4466,  0.4088,  0.7071)and(  1.4929,  0.1759,  0.7071)..(  1.4929,  0.0000,  0.7071)--cycle;
\path[facetfront]( -1.4929,  0.0000,  0.7071)..controls( -1.3078,  0.0000,  0.5220)and( -1.2000,  0.0000,  0.2618)..( -1.2000,  0.0000,  0.0000)..controls( -1.2000, -0.0738,  0.0000)and( -1.1899, -0.1601,  0.0000)..( -1.1725, -0.2443,  0.0000)..controls( -1.1712, -0.2504, -0.0056)and( -1.1273, -0.4595,  0.1823)..( -1.1236, -0.4654,  0.1764)..controls( -1.1573, -0.4794,  0.3756)and( -1.2469, -0.5165,  0.5638)..( -1.3793, -0.5713,  0.7071)..controls( -1.4466, -0.4088,  0.7071)and( -1.4929, -0.1759,  0.7071)..( -1.4929,  0.0000,  0.7071)--cycle;
\path[facetfront](  3.2000,  0.0000,  0.0000)..controls(  3.2000,  0.0000,  0.2618)and(  3.0922,  0.0000,  0.5220)..(  2.9071,  0.0000,  0.7071)..controls(  2.9071,  0.3425,  0.7071)and(  2.8169,  0.7961,  0.7071)..(  2.6858,  1.1125,  0.7071)..controls(  2.8202,  1.1682,  0.5616)and(  2.9105,  1.2056,  0.3697)..(  2.9430,  1.2190,  0.1671)..controls(  3.0348,  1.0646,  0.0146)and(  3.0901,  0.8232,  0.1581)..(  3.1265,  0.6515,  0.0000)..controls(  3.1729,  0.4269,  0.0000)and(  3.2000,  0.1968,  0.0000)..(  3.2000,  0.0000,  0.0000)--cycle;
}
\tikzset{ 
/torus/.cd,  
theta/.initial=42, 
phi/.initial=11.25, 
scale/.initial=\unitscale,  
}
\newenvironment{torus}[1][]{%
\tikzset{/torus/.cd,#1}%
\pgfmathsetmacro{\thetapar}{\pgfkeysvalueof{/torus/theta}}%
\pgfmathsetmacro{\phipar}{\pgfkeysvalueof{/torus/phi}}% 
\tdplotsetmaincoords{\thetapar}{\phipar}%
\begin{tikzpicture}[tdplot_main_coords,line cap=round,line join=round,scale=\pgfkeysvalueof{/torus/scale},thin,baseline={(0,0,0)},
facetback/.style={fill=cyan!5!black!15,draw=black!30},
facetfront/.style={draw=black},
]
\torusback
}{
\pgfresetboundingbox
\torusfront 
\end{tikzpicture}%
}

\usepackage[initials]{amsrefs}
 
\usepackage[
pdftitle={},
pdfauthor={Mario B. Schulz}, 
pdfborder={0 0 0},
]{hyperref}

\theoremstyle{plain}
\newtheorem{theorem}{Theorem}[section]
\newtheorem*{theorem*}{Theorem}
\newtheorem{lemma}[theorem]{Lemma}
\newtheorem{corollary}[theorem]{Corollary}

\theoremstyle{definition}

\theoremstyle{remark}
\newtheorem{remark}[theorem]{Remark}

\title{Free boundary minimal Möbius bands in~toroids}
\author{Mario B. Schulz}
\date{\vspace*{-3.75ex}} 
 
\newcommand\printaddress{{
\setlength{\parindent}{17pt}
\smallskip
\hfill October 2024
\par
{\scshape Mario B. Schulz}
\newline 
Università di Trento, 
Dipartimento di Matematica, 
via Sommarive 14, 
38123 Povo, 
Italy
\newline
\textit{E-mail address:} 
\texttt{mario.schulz@unitn.it}
}} 
 
\begin{document}

\maketitle

\begin{abstract}
We prove that strictly mean convex toroids contain infinitely many (geometrically distinct) embedded free boundary minimal Möbius bands as well as infinitely many embedded free boundary minimal annuli. 
The surfaces in both families are constructed by means of equivariant variational methods and their areas grow linearly with the order of their symmetry groups.  
\end{abstract}

\section{Introduction}
Given a compact, three-dimensional Riemannian manifold $M$ with nonempty boundary $\partial M$, 
can a compact surface of a given topological type be realised as free boundary minimal surface embedded in the ambient manifold $M$? 
A solution $\Sigma\subset M$ of this problem is a critical point for the area functional among all surfaces in $M$ with boundary constrained to $\partial M$. 
Equivalently, $\Sigma$ has vanishing mean curvature and meets the ambient boundary $\partial M$ orthogonally along its own boundary. 
The existence question is open in general and particularly challenging when aiming for unstable critical points. 
In such cases, min-max methods offer promising avenues for existence results. 
The min-max theory for the area functional was pioneered by Almgren--Pitts \cite{Almgren1965, Pitts1981} and later significantly advanced by Marques--Neves \cite{MarquesNeves2014, MarquesNeves2017}, leading to substantial progress in minimal surface theory. 
For the construction of free boundary minimal surfaces with prescribed topology, it is advantageous to employ the min-max theory by Simon--Smith \cite{Smith1982} and Colding--De~Lellis \cite{ColdingDeLellis2003} because this variant imposes stricter criteria on the regularity of the sweepouts in question. 
Min-max methods have indeed been successful in constructing embedded free boundary minimal surfaces 
of various topological types \cite{GruterJost1986,Ketover2016FBMS,Li2015,CarlottoFranzSchulz2022,FranzSchulz2023,SchulzEllipsoids}. 

Controlling the topology of the resulting limit surface is the main difficulty in the min-max approach.   
In the setting of Simon--Smith min-max theory, general genus bounds have been obtained in \cite{DeLellisPellandini2010,Li2015,Ketover2019}. 
The author's joint work with Franz \cite{FranzSchulz2023} also provides a general result providing control on the number of boundary components of a free boundary minimal surface obtained through min-max methods.   
The main result \cite[Theorem 1.8]{FranzSchulz2023} is applicable in ambient manifolds with strictly mean convex boundary and
establishes lower semicontinuity of the first Betti number along min-max sequences.   
Notably, the theorem does not require the assumption of orientability.
Nonorientable surfaces can not be properly embedded in simply connected ambient $3$-manifolds such as the Euclidean unit ball for topological reasons. 
The goal of this article is to demonstrate that certain three-dimensional ambient manifolds with mean convex boundary and nontrivial topology do contain infinitely many embedded free boundary minimal Möbius bands. 
The inspiration for this work dates back over 80 years to the historical origins of the field. 
Indeed, Courant \cite{Courant1940}*{II.\,\S\,3.2} wrote that ``the most interesting problems with free boundaries are those in which the entire boundary is free on a given closed surface \emph{not} of genus zero, e.\,g. on a \emph{torus}.''

We consider the compact toroid $\ambient$ in $\R^3$ bounded by a rotationally symmetric torus with outer radius $\rho>2$ and unit circular vertical cross-sections as ambient manifold.  
More precisely, 
\begin{align}
\label{eqn:ambient} 
\ambient=\Bigl\{(x_1,x_2,x_3)\in\R^3\st\bigl(\sqrt{x_1^2+x_2^2}-\rho\bigr)^2+x_3^2\leq1\Bigr\}. 
\end{align}
The assumption $\rho>2$ implies that the boundary $\partial\ambient$ is smooth and strictly \emph{mean convex}. 
Indeed, the sum of the principal curvatures $\kappa_1=1$ and $\kappa_2\geq-1/(\rho-1)$ is positive if $\rho>2$. 
Mean convexity is an important condition, ensuring that embedded free boundary minimal surfaces $\Sigma$ in $\ambient$ are necessarily properly embedded, i.\,e.~embedded and satisfying $\partial\Sigma=\Sigma\cap\partial \ambient$. 

Given $n\in\N$, we understand the \emph{dihedral group} $\dih_n$ of order $2n$ as the subgroup of Euclidean isometries acting on $\ambient$ generated by the rotation $\rotation_{e_1}^{\pi}$ of angle $\pi$ around the $x_1$-axis and by the rotation $\rotation_{e_3}^{2\pi/n}$ of angle $2\pi/n$ around the $x_3$-axis.  
Being generated by rotations, the dihedral group is orientation-preserving and contains the cyclic group $\Z_n$ as a subgroup. 
The dihedral group has already proven extremely useful for equivariant min-max constructions in the unit ball \cite{CarlottoFranzSchulz2022,FranzSchulz2023} and in ellipsoids \cite{SchulzEllipsoids}. 
For each $k\in\{0,\ldots,2n-1\}$ we define the horizontal segment 
\begin{align}\label{eqn:axes}
\xi_k\vcentcolon=\Bigl\{\Bigl(r\cos\Bigl(\frac{\pi k}{n}\Bigr),r\sin\Bigl(\frac{\pi k}{n}\Bigr),0\Bigr)\st r\in[\rho-1,\rho+1]\Bigr\}
\end{align}
as visualised in Figure \ref{fig:main} (left image). 
The union $\xi_0\cup\ldots\cup\xi_{2n-1}$ is fixed under the action of the dihedral group $\dih_n$ and $\dih_n$ contains the rotation of angle $\pi$ around $\xi_k$ for any $k\in\{0,\ldots,2n-1\}$. 
We prove the following existence result.

\begin{figure}%
\tdplotsetmaincoords{42}{11.25}%
\begin{tikzpicture}[tdplot_main_coords,line cap=round,line join=round,scale=\unitscale,thin,baseline={(0,0,0)},
facetback/.style={fill=cyan!5!black!15,draw=black!30},
facetfront/.style={draw=black},
]
\torusback
\tdplotsetrotatedcoords{0}{0}{-90}
\begin{scope}[tdplot_rotated_coords,black!75,thick]
\draw(0,0,0)--(2.2-1,0,0)(0,0,0)--(0,2.2-1,0);
\end{scope}
\FBMS{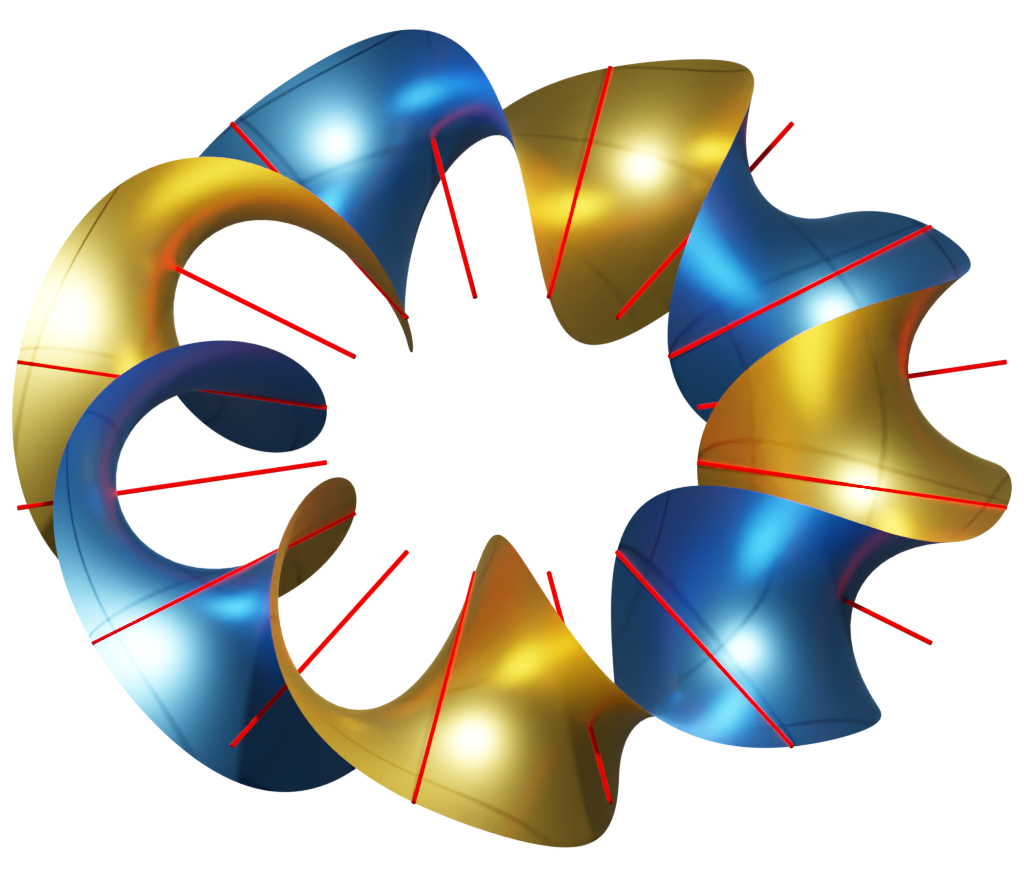}
\draw[latex-latex,very thick](0,0,0)--++(2.2,0,0)node[pos=0.33,above,inner sep=1.5pt]{$\rho$}; 
\draw[latex-latex,very thick](3.2,0,0)--(2.2,0,0)node[pos=0.66,above]{$1$};
\torusfront
\tdplotsetrotatedcoords{0}{0}{-90}
\begin{scope}[tdplot_rotated_coords,black!75,thick]
\draw[->](0,0,0)--(0,0,4.4)node[inner sep=0,below right]{$~x_3$};
\draw[->](2.2+1,0,0)--(4,0,0)node[inner sep=0,above left]{$x_1~$};
\draw[->](0,2.2+1,0)--(0,4,0)node[below left]{$x_2$};
\begin{scope}[red]
\draw(3.2,0,0)node[above right]{$\xi_0$};
\draw({3.2*cos(180/8)},{3.2*sin(180/8)},0)node[above right]{$\xi_1$};
\draw({3.2*cos(2*180/8)},{3.2*sin(2*180/8)},0)node[anchor=-100]{$\xi_2$};
\draw({3.2*cos(3*180/8)},{3.2*sin(3*180/8)},0)node[above]{$\xi_3$};
\end{scope}
\end{scope}
\end{tikzpicture}%
\hfill
\begin{torus}
\FBMS{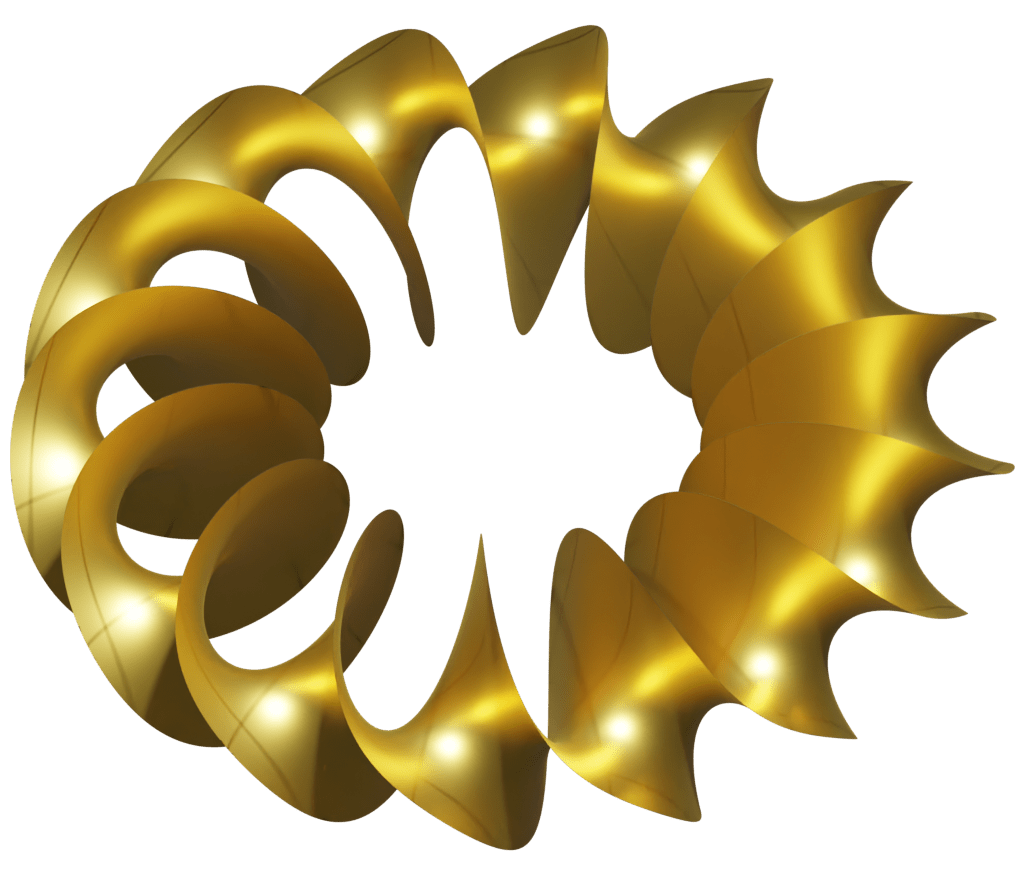}
\end{torus}
\caption{$\dih_n$-equivariant free boundary minimal surfaces $\sol_n$ in $\ambient$ for $\rho=2.2$ and $n\in\{8,15\}$.}%
\label{fig:main}%
\end{figure} 

\begin{theorem}\label{thm:main} 
For every $\rho>2$ and every integer $n\geq5(\rho+1)/2$ the toroid $\ambient$ contains an embedded, $\dih_n$-equivariant free boundary minimal surface $\sol_n$ with the following properties. 
\begin{enumerate}[label={\normalfont(\roman*)},nosep]
\item\label{thm:main-i} $\sol_n$ contains the segment $\xi_k$ for all even $k$ and intersects it orthogonally for all odd $k$. 
\item\label{thm:main-ii} $\sol_n$ is a topological Möbius band if $n$ is odd respectively an annulus if $n$ is even. 
\item\label{thm:main-iii} The area of $\sol_n$ is strictly between $n\pi$ and $2n\pi$.  
\end{enumerate} 
\end{theorem}

For comparison, the Euclidean unit ball $\B^3$ is not known to contain infinite families of pairwise noncongruent, embedded free boundary minimal surfaces all of which have the same topology, even though $k$-tuples of surfaces with these properties have been discovered in \cite{CSWstackings,KarpukhinKusnerMcGrathStern} for every $k\in\N$. 
Moreover, Cardona \cite{Cardona} proved that even in the class of branched minimal immersions, the Euclidean unit ball does not contain any free boundary minimal Möbius bands.  
Therefore, Theorem \ref{thm:main} points at striking differences between free boundary minimal surfaces in toroids when compared to those in the unit ball. 

The free boundary problem solved in Theorem~\ref{thm:main} can also be compared with the Dirichlet problem in toroids: Area minimising Möbius bands with prescribed boundary on a mean convex torus have been found in \cite[Corollary 6.3]{Jost1986I}.  
Complete minimal Möbius bands immersed in $\R^3$ have been studied in \cite{Weber2007,Mira2006,Oliveira1986,Meeks1981}. 
Allowing for codimension $2$, Fraser and Schoen \cite{FraserSchoen2016} characterised the \emph{critical Möbius band}, a free boundary minimal surface in $\B^4$. 

We prove Theorem~\ref{thm:main} in Section~\ref{sec:proof} using an equivariant, $1$-parameter min-max scheme. 
Previous constructions of this type rely on area estimates which are uniform with respect to the order of the symmetry group \cite{CarlottoFranzSchulz2022,Ketover2016FBMS,FranzSchulz2023}. 
In this context, property \ref{thm:main-iii} stating that the area of $\sol_n$ grows linearly with $n$ is a remarkable novelty. 
Area estimates are crucial for the min-max construction:  
The strict lower area bound is related with the width estimate established in Section \ref{sec:width}.   
The strict upper area bound relies on the assumption $n\geq5(\rho+1)/2$ and ensures that $\sol_n$ is not just a $\dih_n$-equivariant union of $2n$ planar discs, as detailed in Sections~\ref{sec:sweepout} and~\ref{sec:proof}. 
We prove these estimates by employing a combination of new ideas tailored to the toroidal setting. 
Another key ingredient is our characterization of all free boundary minimal discs in $\ambient$, which is presented in Section~\ref{sec:uniqueness}.

\paragraph{Acknowledgements.}  
This project has received funding from the European Research Council (ERC) under the European Union's Horizon 2020 research and innovation programme (grant agreement No.~947923).

\section{Construction of helicoidal sweepouts}
\label{sec:sweepout}

The goal of this section is to construct an effective $1$-parameter $\dih_n$-sweepout of $\ambient$, 
that is a family $\{\Sigma_t\}_{t\in[0,1]}$ of $\dih_n$-equivariant subsets of $\ambient$ such that 
$\Sigma_t$ is a smooth, properly embedded surface in $\ambient$ for every $t\in\interval{0,1}$, 
varying smoothly in $t\in\interval{0,1}$ and continuously, in the sense of varifolds, for $t\in[0,1]$ 
(cf. \cite[Definition 1.1]{FranzSchulz2023}). 
The novel idea is to use helicoids to construct suitable sweepout slices in cylinders and then map them to $\ambient$.

\paragraph{Cylinders and translations.} 

We denote the Cartesian coordinates on $\R^3$  by $x_1,x_2,x_3$ and the corresponding standard orthonormal basis by $e_1,e_2,e_3$.  
Let $\B^2=\{(x_1,x_2)\in\R^2\st x_1^2+x_2^2\leq1\}$ denote the closed Euclidean unit disc.  
For the construction of surfaces in $\ambient$ it is convenient 
to consider the Euclidean unit cylinder $Z\vcentcolon=\B^2\times\R\subset\R^3$ and the covering map $\Phi_\height\colon Z\to \ambient$ defined for $\height>0$ by 
\begin{align}\label{eqn:covering}
\Phi_\height(x_1,x_2,x_3)&=\begin{pmatrix}
(\rho+x_1)\cos(2x_3/\height)\\
(\rho+x_1)\sin(2x_3/\height) \\
x_2
\end{pmatrix}.
\end{align}
For any $\alpha\in[0,1]$ and any $\height>0$ we introduce the notation  
\begin{align}\label{eqn:sector}
Z_{\height\alpha}&\vcentcolon=\B^2\times\IntervaL{0,\height\pi\alpha}, &
\ambient_\alpha\vcentcolon=\Phi_\height(Z_{\height\alpha}). 
\end{align}
$\ambient_0$ is a vertical slice through $\ambient$ at toroidal angle $0$. 
For $0<\alpha<1$ the set $\ambient_\alpha$ is a torus sector with toroidal angle $2\pi\alpha$ and the restriction of $\Phi_\height$ to $Z_{\height\alpha}$ is a diffeomorphism onto its image. 
For $\alpha=1$ the restriction of $\Phi_\height$ to $Z_\height=\B^2\times[0,\height\pi]$ is surjective and we simply have $\ambient_1=\ambient$. 
The following lemma facilitates area estimates for surfaces in $\ambient$. 
Here and in the following, $\hsd^2$ denotes the $2$-dimensional Hausdorff measure such that $\hsd^2(\Sigma)$ is the area of a surface $\Sigma\subset\R^3$.

\begin{lemma}
\label{lem:biLipschitz}
Let $\Phi_\height\colon Z_\height\to \ambient$ be as in \eqref{eqn:covering} and 
$\Sigma\subset Z_\height$ any properly embedded surface. Then 
\[ 
\frac{2(\rho-1)}{\height}\hsd^2(\Sigma)
\leq\hsd^2\bigl(\Phi_\height(\Sigma)\bigr)
\leq\frac{2(\rho+1)}{\height}\hsd^2(\Sigma).
\] 
\end{lemma}

\begin{proof}
Let $D\Phi_\height$ be the Jacobian matrix of $\Phi_\height$. 
Then $(D\Phi_\height)^{\intercal}\cdot(D\Phi_\height)$ is a diagonal matrix with nonzero entries $1,1,a^2$ for $a=2(\rho+x_1)/\height$. 
In particular, $\abs{\det(D\Phi_\height)}=\abs{a}$. 
Since $\abs{x_1}\leq1$ the claim follows. 
\end{proof}

Let $\translation_{y}(x)=x+y$ denote the translation by $y$ in $\R^3$. 
Given $n\in\N$ let $G_n$ be the subgroup of Euclidean isometries acting on $Z=\B^2\times\R$ generated by the rotation $\rotation_{e_1}^{\pi}$ and the vertical translation $\translation_{(\height\pi/n)e_3}$.  
Let $\sk{\translation_{\height\pi e_3}}$ denote the subgroup generated by $\translation_{\height\pi e_3}$. 
The quotient 
\begin{align}
\label{eqn:quotient}
\tilde{Z}_{\height}&\vcentcolon=Z/\sk{\translation_{\height\pi e_3}}
\end{align}
is a solid torus with mean convex boundary. 
In analogy to \eqref{eqn:axes} we define for each $k\in\Z$
\begin{align}\label{eqn:axes2}
\zeta_k\vcentcolon=\Bigl\{\Bigl(r,0,\frac{\height\pi k}{2n}\Bigr)\st r\in[-1,1]\Bigr\}\subset Z.
\end{align}
We also introduce the notation $\tilde{\zeta}_k\subset\tilde{Z}_{\height}$ for the image of $\zeta_k$ under the quotient map $\varpi\colon Z\to\tilde{Z}_\height$.  
Definitions \eqref{eqn:covering}, \eqref{eqn:axes} and \eqref{eqn:axes2} are compatible in the following sense. 

\begin{lemma}\label{lem:symmetry_compatibility}
The map $\Phi_\height\colon Z\to \ambient$ defined in \eqref{eqn:covering} descends to the quotient $\tilde{Z}_{\height}$, where it becomes a $\dih_n$-equivariant diffeomorphism for any $n\in\N$. 
Moreover, $\Phi_\height(\zeta_k)=\xi_k$ for every $k$. 
\end{lemma}	
	
\begin{proof}
Since sine and cosine are $2\pi$-periodic functions, it is clear that $\Phi_\height\circ\translation_{\height\pi e_3}=\Phi_\height$. 
Recalling $\rotation_{e_1}^{\pi}(x_1,x_2,x_3)=(x_1,-x_2,-x_3)$ it follows by definition that $\Phi_\height\circ \rotation_{e_1}^{\pi}=\rotation_{e_1}^{\pi}\circ\Phi_\height$. 
Moreover, 
$\Phi_\height\circ\translation_{(\height\pi/n) e_3}
=\rotation_{e_3}^{2\pi/n}\circ\Phi_\height$
 it is easy to verify.  
We conclude that the group action of the factor group $G_n/\sk{\translation_{\height\pi e_3}}$ on $\tilde{Z}_{\height}$ is isomorphic to the action of $\dih_n$ on $\ambient$. 
The statement $\Phi_\height(\zeta_k)=\xi_k$ follows directly from \eqref{eqn:axes} and \eqref{eqn:axes2}. 
\end{proof}

\paragraph{Free boundary helicoids.}
The helicoid $\heli\vcentcolon=\{(x_1,x_2,x_3)\in\R^3\st x_1\tan x_3=x_2\}$ 
is a complete, embedded, singly-periodic, simply connected minimal surface 
(cf. \cite[§\,1.2.2]{ColdingMinicozzi2011}).  
Its minimality was shown by Jean Baptiste Meusnier \cite{meusnier1785} in 1776.  
The helicoid $\heli$ is is a ruled surface parametrised by $\R^2\ni(u,v)\mapsto(v\cos u,v\sin u,u)$. 
In particular, any rescaling $\height\heli$ for $\height>0$ of the helicoid intersects the boundary of the Euclidean unit cylinder $Z=\B^2\times\R$ orthogonally; 
thus $Z\cap\height\heli$ is a free boundary minimal surface in $Z$. 
Recalling \eqref{eqn:axes2}, we notice that $\zeta_k$ intersects
$Z\cap\tfrac{\height}{n}\heli$ orthogonally for every odd $k$ and that 
$Z\cap\tfrac{\height}{n}\heli\cap\{x_2=0\}$ coincides with the union of all $\zeta_k$ with even $k$.  
Moreover, $Z\cap\tfrac{\height}{n}\heli$ is $G_n$-equivariant and descends to a 
 $\dih_n$-equivariant free boundary minimal surface in the quotient $\tilde{Z}_{\height}$ defined in \eqref{eqn:quotient}. 
The following lemma indicates that the area of a helicoid's period inside a cylinder is decreasing under horizontal translation. 

\begin{lemma}
\label{lem:translation}
Given $\height>0$ let $Z_\height=\B^2\times\IntervaL{0,\height\pi}$ and let 
$\heli_{\height,s}\vcentcolon=(Z_\height+(s,0,0))\cap\height\heli$ for $s\in\R$. 
Then $\hsd^2(\heli_{\height,s})=\hsd^2(\heli_{\height,-s})$ is strictly decreasing in $s>0$.
\end{lemma}	
	
\begin{proof}
Given $s\in\R$ let $\Omega_s=\{x=(x_1,x_2)\in\R^2\st (x_1-s)^2+x_2^2\leq1,~x_2>0\}$ be the laterally translated upper half-disc (see Figure \ref{fig:translation}). 
Consider the function $w\colon\Omega_s\to\interval{0,\height\pi}$ given by 
\begin{align}\label{eqn:graph}
w(x)&=\begin{cases}
\height\arctan(x_2/x_1) & \text{ if $x_1>0$, } \\
\height\pi/2 & \text{ if $x_1=0$, } \\
\height\arctan(x_2/x_1)+\height\pi  & \text{ if $x_1<0$, }
\end{cases}	
\end{align}
where $\arctan\colon\R\to\interval{-\frac{\pi}{2},\frac{\pi}{2}}$ is the inverse of $\tan\colon\interval{-\frac{\pi}{2},\frac{\pi}{2}}\to\R$. 
By definition, the graph of $w$ coincides with the surface 
$\heli_{\height,s}^{+}\vcentcolon=\heli_{\height,s}\cap\{x_2>0\}$ and by symmetry of $\heli$ we have 
\begin{align}\label{eqn:20240303-1}
\hsd^2(\heli_{\height,s})
&=2\hsd^2(\heli_{\height,s}^+)
=2\int_{\Omega_s}\sqrt{1+\abs{\nabla w}^2}\,dx
=2\int_{\Omega_s}\sqrt{1+\height^2\abs{x}^{-2}}\,dx 
\end{align}
for every $s\in\R$. 
Let $\reflection_0\colon(x_1,x_2)\mapsto(-x_1,x_2)$ and let $f(x)\vcentcolon=\sqrt{1+\height^2\abs{x}^{-2}}$. 
Then $f\circ\reflection_0=f$ and $\reflection_0(\Omega_s)=\Omega_{-s}$. 
From \eqref{eqn:20240303-1} we then obtain $\hsd^2(\heli_{\height,s})=\hsd^2(\heli_{\height,-s})$ as claimed.  

Let us now assume $s>0$ and consider the reflection $\reflection_s\colon(x_1,x_2)\mapsto(s-x_1,x_2)$ in $\R^2$ across $\{x_1=s/2\}$. 
Then $\Omega_s\setminus\Omega_0=\reflection_s(\Omega_0\setminus\Omega_s)$ 
as viusalised in Figure \ref{fig:translation}. 
Since $f\in L_{\text{loc}}^1(\R^2)$, we have 
\begin{align}\label{eqn:20240303-2}
 \int_{\Omega_s}f\,dx
-\int_{\Omega_0}f\,dx	
&=\int_{\Omega_s\setminus\Omega_0}f\,dx
-\int_{\Omega_0\setminus\Omega_s}f\,dx	
=\int_{\Omega_0\setminus\Omega_s}(f\circ\reflection_s)-f\,dx.
\end{align}
Notice that $x_1<s/2$ and thus $(s-x_1)^2>x_1^2$ for every $x=(x_1,x_2)\in\Omega_0\setminus\Omega_s$. 
Hence, we directly obtain $f\bigl(\reflection_s(x)\bigr)<f(x)$ for every $x\in\Omega_0\setminus\Omega_s$. 
Moreover, for every fixed $x\in\Omega_0\setminus\Omega_s$, 
\begin{align*} 
\frac{\partial}{\partial s}f\bigl(\reflection_s(x)\bigr)
&=\frac{\partial}{\partial s}\sqrt{1+\frac{\height^2}{(s-x_1)^2+x_2^2}}
=\frac{-(s-x_1)\height^2}{\bigl((s-x_1)^2+x_2^2\bigr)^2f\bigl(\reflection_s(x)\bigr)}<0.
\end{align*}
Since additionally, $\Omega_0\setminus\Omega_{s_1}\subset\Omega_0\setminus\Omega_{s_2}$ for every $0<s_1<s_2$ we conclude that quantity \eqref{eqn:20240303-2} is strictly decreasing in $s>0$ and the claim follows. 
\end{proof}	

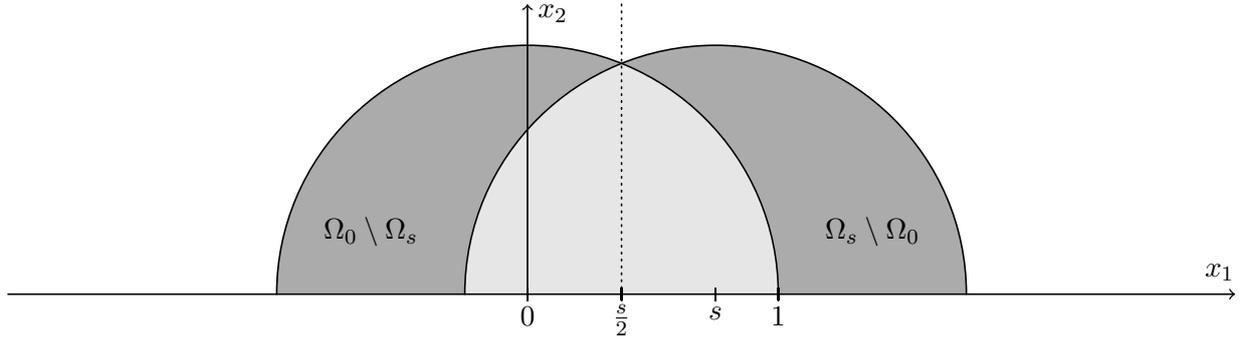
\begin{figure}\centering
\pgfmathsetmacro{\scalepar}{3.3}
\begin{tikzpicture}[line cap=round,line join=round,scale=\scalepar,semithick]
\pgfmathsetmacro{\spar}{0.75}
\pgfmathsetmacro{\xmax}{\textwidth/\scalepar/2cm-0.001}
\pgfmathsetmacro{\ymax}{1.165}
\pgfmathsetmacro{\winkel}{acos(\spar/2)}
\fill[black!33](1,0)arc(0:\winkel:1)arc(180-\winkel:0:1)--cycle;
\fill[black!33,shift={(\spar,0)},xscale=-1](1,0)arc(0:\winkel:1)arc(180-\winkel:0:1)--cycle;
\fill[black!10](1,0)arc(0:\winkel:1)arc(180-\winkel:180:1)--cycle;
\draw plot[vdash](\spar/2,0)node[below]{$\frac{s}{2}$};
\pgfresetboundingbox
\draw plot[vdash](0,0)node[below=1pt]{$0$};
\draw plot[vdash](1,0)node[below=1pt]{$1$};
\draw plot[vdash](\spar,0)node[below=1pt]{$s$};
\draw( 1+\spar/2,0)node[above=3ex]{$\Omega_s\setminus\Omega_0$};
\draw(-1+\spar/2,0)node[above=3ex]{$\Omega_0\setminus\Omega_s$};
\draw[->](\spar/2-\xmax,0)--(\spar/2+\xmax,0)node[above left=1ex and 0,inner sep=0]{$x_1$};
\draw[->](0,0)--(0,\ymax)node[below right,inner sep=0]{~$x_2$};
\draw[dotted](\spar/2,\ymax)--(\spar/2,0);
\draw(1,0)arc(0:180:1);
\draw(1+\spar,0)arc(0:180:1);	
\end{tikzpicture}
\caption{The upper half-disc $\Omega_0$ and its translation $\Omega_s=\Omega_0+(s,0)$.}%
\label{fig:translation}%
\end{figure}

As a direct consequence of Lemma~\ref{lem:translation}, we obtain a proof of the following Morse index estimate. 
While not required for the remainder of the article, this result is of independent interest being an example of index growth not imputable to topology (cf.~\cite{CSWIndexgrowth} for higher-dimensional scenarios).

\pagebreak[2]

\begin{corollary}
\label{cor:index}
Let $\tilde{Z}_{\height}$ be as in \eqref{eqn:quotient} and let 
$\tilde{\heli}_n$ be the $\dih_n$-equivariant free boundary minimal surface in $\tilde{Z}_{\height}$ which lifts to $Z\cap\frac{\height}{n}\heli$ in $Z$. 
Then the Morse index of $\tilde{\heli}_n$ is at least $n$. 
\end{corollary}

\begin{proof}
The surface $\tilde{\heli}_n$ can be divided into $n$ pairwise isometric domains $\Omega_1,\ldots,\Omega_n$ by removing the segments $\tilde\zeta_k$ for all even $k$, recalling \eqref{eqn:axes2}.  
By \cite[Corollary~3.2\,(i)]{CSWSpectral} the Morse index of $\tilde{\heli}_n$ is bounded from below by $n$ times the index of the Jacobi operator $J$ on $\Omega_1$ subject to the Dirichlet boundary condition on $\tilde\zeta_0\cup\tilde\zeta_2$.   
We apply Lemma \ref{lem:translation} with $\height/n$ in place of $\height$ to obtain a variation $\Omega_1(s)$ of $\Omega_1$ defined for some $-\varepsilon<s<\varepsilon$, with (strictly) maximal area at $s=0$.  
The corresponding variation vector field at $s=0$ vanishes on the segments $\tilde{\zeta}_0$ and $\tilde{\zeta}_2$ and therefore induces a negative direction for $J$ with the aforementioned boundary condition. 
\end{proof}

\begin{remark}
Let $M$ be any compact, three-dimensional ambient manifold with nonnegative Ricci curvature and \emph{convex} boundary.  
Lima \cite[Theorem~4]{Lima2022} proved that the Morse index of any free boundary minimal surface $\Sigma$ in $M$ is bounded by its topological complexity. 
Corollary~\ref{cor:index} shows that strict convexity of $\partial M$ is a necessary assumption for this result because for any $n\in\N$ the free boundary minimal surface $\tilde{\heli}_{2n}$ in $\tilde{Z}_{\height}$ 
has the topology of an annulus.  
\end{remark}

\begin{lemma}[Area of a piece of a helicoid]
\label{lem:helicoid-area}
Let $\heli_{\height,0}$ be as in Lemma \ref{lem:translation}. 
Then 
\begin{align*}
\hsd^2(\heli_{\height,0})
&=\pi\sqrt{1+\height^2}+\pi\height^2\log\Bigl(\height^{-1}+\sqrt{\height^{-2}+1}\Bigr).
\end{align*}
Moreover, $\hsd^2(\heli_{\height,0})<2\pi$ if $\height\leq\frac{4}{5}$. 
\end{lemma}

\begin{proof}
We recall that $\heli_{\height,0}$ is parametrised by $\phi(u,v)=(v\cos u,v\sin u,\height u)$ defined on $\Interval{0,\pi}\times[-1,1]$. 
In particular, $\heli_{\height,0}$ intersects the positive level sets of the function 
\[
\varrho(x_1,x_2,x_3)\vcentcolon=\sqrt{x_1^2+x_2^2}
\]
orthogonally. 
Hence, $\abs{\nabla^{\heli_{\height,0}}\varrho}=\abs{\nabla\varrho}=1$, where $\nabla^{\heli_{\height,0}}\varrho$ denotes the tangential gradient of the function $\varrho$ restricted to $\heli_{\height,0}$. 
For every $0<r\leq1$ the set $\heli_{\height,0}\cap\varrho^{-1}(r)$ is a double helix
parametrised by $\Interval{0,\pi}\ni u\mapsto\phi(u,\pm r)$. 
Its total length is equal to 
\begin{align*}
\hsd^1\bigl(\heli_{\height,0}\cap\varrho^{-1}(r)\bigr)
&=2\pi\sqrt{r^2+\height^2}.
\end{align*}
The function $r\mapsto2\sqrt{r^2+\height^2}$ has a primitive given by 
$r\sqrt{r^2+\height^2}+\height^2\log\bigl(r+\sqrt{r^2+\height^2}\bigr)$. 
The coarea formula then implies 
\begin{align}
\hsd^2(\heli_{\height,0})
=\int^{1}_{0}\biggl(\int_{\heli_{\height,0}\cap\varrho^{-1}(r)}\frac{1}{\abs{\nabla^{\heli_{\height,0}}\varrho}}d\hsd^1\biggr)\,dr	
&=\pi\int^{1}_{0}2\sqrt{r^2+\height^2}\,dr
\label{eqn:20240304}
\\\notag
&=\pi\sqrt{1+\height^2}+\pi\height^2\log\Bigl(\height^{-1}+\sqrt{\height^{-2}+1}\Bigr). 
\end{align}
The area \eqref{eqn:20240304} is clearly increasing in $\height$. 
For $\height=\frac{4}{5}$ we evaluate $\hsd^2(\heli_{\frac{4}{5},0})\approx1.95\pi<2\pi$. 
\end{proof}

\begin{remark}
An alternative approach is to compute the area of $\heli_{\height,0}$ as in equation \eqref{eqn:20240303-1}. 
Indeed, integrating $\sqrt{1+\height^2\abs{x}^{-2}}$ in polar coordinates over the unit disc yields the same integral as in equation \eqref{eqn:20240304}. 
Numerically, we can solve the equation $\hsd^2(\heli_{\height,0})=2\pi$ for $\height\approx0.829$.
\end{remark}

\begin{lemma}[Optimal sweepout for the helicoid]
\label{lem:sweepout}
Given any $\height>0$, let $Z_\height$ and $\heli_{\height,s}$ be as in Lemma \ref{lem:translation} and 
let $\zeta_0,\zeta_1,\zeta_2\subset Z_\height$ be as in \eqref{eqn:axes2} for $n=1$.
There exists a sweepout $\{\Gamma_t\}_{t\in[0,1]}$ of $Z_{\height}$ with the following properties.
\begin{enumerate}[label={\normalfont(\roman*)}] 
\item\label{lem:sweepout-area} $\hsd^2(\Gamma_0)=\hsd^2(\Gamma_1)=\pi$ and $\hsd^2(\Gamma_t)\leq\hsd^2(\heli_{\height,0})$ for every $0<t<1$.  
\item\label{lem:sweepout-sym} $\Gamma_t$ contains the segments $\zeta_{0}$ and $\zeta_{2}$, intersects $\zeta_1$ orthogonally, and is equivariant with respect to rotation by angle $\pi$ around $\zeta_1$ for every $0<t<1$. 
\item\label{lem:sweepout-topo} 
$\Gamma_t$ descends to a smooth, properly embedded, $\dih_1$-equivariant Möbius band $\tilde\Gamma_t$ in the quotient $\tilde{Z}_{\height}$ defined in~\eqref{eqn:quotient} for every $0<t<1$. 
\item\label{lem:sweepout-perim}  For every $0<t<1$ there exists a set $F_t\subset Z_\height$ of finite perimeter such that 
\begin{itemize}[nosep]
\item 
$\Gamma_t$ is the relative perimeter of $F_t$ in the sense that $\Gamma_t\setminus\partial Z_\height=\partial F_t\setminus\partial Z_\height$;
\item $\hsd^3(F_t)\to0$ as $t\to0$ and $\hsd^3(F_t)\to\hsd^3(Z_\height)$ as $t\to1$; 
\item $t\to t_0$ implies $\hsd^3(F_t\symdiff F_{t_0})\to0$, 
where $F_t\symdiff F_{t_0}\vcentcolon=(F_{t}\setminus F_{t_0})\cup(F_{t_0}\setminus F_{t})$. 
\end{itemize}
\end{enumerate}
\end{lemma}

\begin{figure}
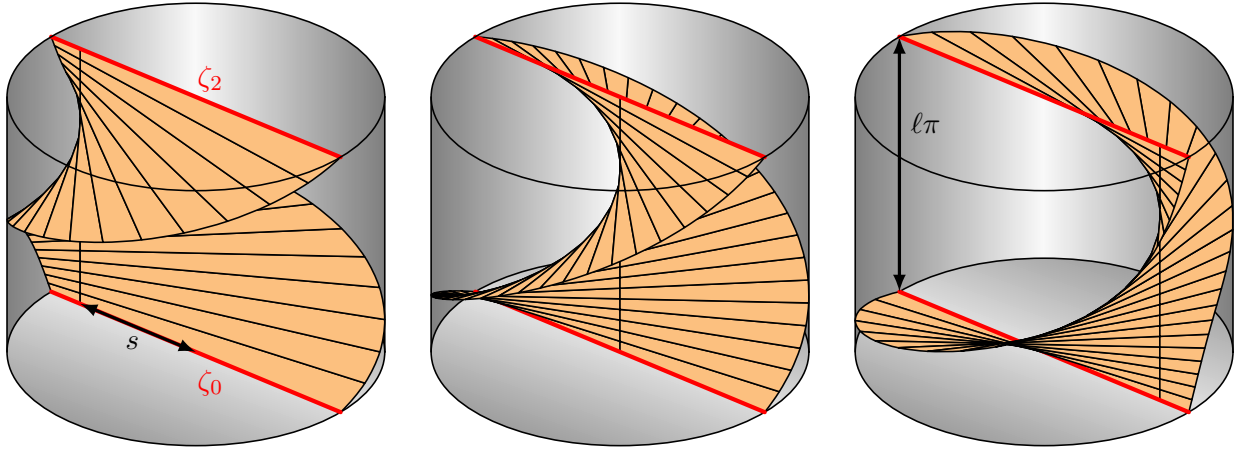
%
\pgfmathsetmacro{\phiO}{40} 
\pgfmathsetmacro{\thetaO}{60}
\tdplotsetmaincoords{\thetaO}{\phiO}
\cylinder{% Surface L
\path[facet]( -1.0000, -0.0000,  0.0000)..controls( -1.0000, -0.0081,  0.0201)and( -0.9999, -0.0181,  0.0453)..( -0.9997, -0.0263,  0.0654)--( -0.8000,  0.0000,  0.0654)..controls( -0.8000,  0.0000,  0.0453)and( -0.8000,  0.0000,  0.0201)..( -0.8000,  0.0000,  0.0000)--cycle;%[1]
\path[facet]( -0.8000,  0.0000,  0.0000)..controls( -0.8000,  0.0000,  0.0201)and( -0.8000,  0.0000,  0.0453)..( -0.8000,  0.0000,  0.0654)--(  0.9724,  0.2333,  0.0654)..controls(  0.9893,  0.1629,  0.0453)and(  1.0000,  0.0725,  0.0201)..(  1.0000,  0.0000,  0.0000)--cycle;%[2]
\draw[axis](-1,0,0)--++(2,0,0) node[pos=0.55,below]{$\zeta_0$};
\draw[latex-latex,very thick](0,0,0)--++(-0.8,0,0)node[pos=0.55,below]{$s$};
\path[facet]( -0.9997, -0.0263,  0.0654)..controls( -0.9994, -0.0344,  0.0856)and( -0.9990, -0.0448,  0.1108)..( -0.9986, -0.0532,  0.1309)--( -0.8000,  0.0000,  0.1309)..controls( -0.8000,  0.0000,  0.1108)and( -0.8000,  0.0000,  0.0856)..( -0.8000,  0.0000,  0.0654)--cycle;%[3]
\path[facet]( -0.8000,  0.0000,  0.0654)..controls( -0.8000,  0.0000,  0.0856)and( -0.8000,  0.0000,  0.1108)..( -0.8000,  0.0000,  0.1309)--(  0.8914,  0.4532,  0.1309)..controls(  0.9241,  0.3889,  0.1108)and(  0.9555,  0.3037,  0.0856)..(  0.9724,  0.2333,  0.0654)--cycle;%[4]
\path[facet]( -0.9986, -0.0532,  0.1309)..controls( -0.9981, -0.0617,  0.1510)and( -0.9974, -0.0725,  0.1762)..( -0.9967, -0.0815,  0.1963)--( -0.8000,  0.0000,  0.1963)..controls( -0.8000,  0.0000,  0.1762)and( -0.8000,  0.0000,  0.1510)..( -0.8000,  0.0000,  0.1309)--cycle;%[5]
\path[facet]( -0.8000,  0.0000,  0.1309)..controls( -0.8000,  0.0000,  0.1510)and( -0.8000,  0.0000,  0.1762)..( -0.8000,  0.0000,  0.1963)--(  0.7624,  0.6472,  0.1963)..controls(  0.8087,  0.5926,  0.1762)and(  0.8587,  0.5175,  0.1510)..(  0.8914,  0.4532,  0.1309)--cycle;%[6]
\path[facet]( -0.9967, -0.0815,  0.1963)..controls( -0.9959, -0.0904,  0.2165)and( -0.9948, -0.1021,  0.2417)..( -0.9937, -0.1118,  0.2618)--( -0.8000,  0.0000,  0.2618)..controls( -0.8000,  0.0000,  0.2417)and( -0.8000,  0.0000,  0.2165)..( -0.8000,  0.0000,  0.1963)--cycle;%[7]
\path[facet]( -0.8000,  0.0000,  0.1963)..controls( -0.8000,  0.0000,  0.2165)and( -0.8000,  0.0000,  0.2417)..( -0.8000,  0.0000,  0.2618)--(  0.5937,  0.8047,  0.2618)..controls(  0.6506,  0.7627,  0.2417)and(  0.7161,  0.7017,  0.2165)..(  0.7624,  0.6472,  0.1963)--cycle;%[8]
\path[facet]( -0.9937, -0.1118,  0.2618)..controls( -0.9926, -0.1216,  0.2819)and( -0.9910, -0.1344,  0.3071)..( -0.9894, -0.1453,  0.3272)--( -0.8000,  0.0000,  0.3272)..controls( -0.8000,  0.0000,  0.3071)and( -0.8000,  0.0000,  0.2819)..( -0.8000,  0.0000,  0.2618)--cycle;%[9]
\path[facet]( -0.8000,  0.0000,  0.2618)..controls( -0.8000,  0.0000,  0.2819)and( -0.8000,  0.0000,  0.3071)..( -0.8000,  0.0000,  0.3272)--(  0.3964,  0.9181,  0.3272)..controls(  0.4603,  0.8905,  0.3071)and(  0.5368,  0.8467,  0.2819)..(  0.5937,  0.8047,  0.2618)--cycle;%[10]
\path[facet]( -0.9894, -0.1453,  0.3272)..controls( -0.9878, -0.1562,  0.3474)and( -0.9854, -0.1707,  0.3726)..( -0.9831, -0.1831,  0.3927)--( -0.8000,  0.0000,  0.3927)..controls( -0.8000,  0.0000,  0.3726)and( -0.8000,  0.0000,  0.3474)..( -0.8000,  0.0000,  0.3272)--cycle;%[11]
\path[facet]( -0.8000,  0.0000,  0.3272)..controls( -0.8000,  0.0000,  0.3474)and( -0.8000,  0.0000,  0.3726)..( -0.8000,  0.0000,  0.3927)--(  0.1831,  0.9831,  0.3927)..controls(  0.2499,  0.9707,  0.3726)and(  0.3326,  0.9456,  0.3474)..(  0.3964,  0.9181,  0.3272)--cycle;%[12]
\path[facet]( -0.9831, -0.1831,  0.3927)..controls( -0.9808, -0.1955,  0.4128)and( -0.9773, -0.2122,  0.4380)..( -0.9740, -0.2267,  0.4581)--( -0.8000,  0.0000,  0.4581)..controls( -0.8000,  0.0000,  0.4380)and( -0.8000,  0.0000,  0.4128)..( -0.8000,  0.0000,  0.3927)--cycle;%[13]
\path[facet]( -0.8000,  0.0000,  0.3927)..controls( -0.8000,  0.0000,  0.4128)and( -0.8000,  0.0000,  0.4380)..( -0.8000,  0.0000,  0.4581)--( -0.0331,  0.9995,  0.4581)..controls(  0.0325,  1.0016,  0.4380)and(  0.1163,  0.9955,  0.4128)..(  0.1831,  0.9831,  0.3927)--cycle;%[14]
\path[facet]( -0.9740, -0.2267,  0.4581)..controls( -0.9706, -0.2412,  0.4783)and( -0.9655, -0.2609,  0.5035)..( -0.9606, -0.2781,  0.5236)--( -0.8000,  0.0000,  0.5236)..controls( -0.8000,  0.0000,  0.5035)and( -0.8000,  0.0000,  0.4783)..( -0.8000,  0.0000,  0.4581)--cycle;%[15]
\path[facet]( -0.8000,  0.0000,  0.4581)..controls( -0.8000,  0.0000,  0.4783)and( -0.8000,  0.0000,  0.5035)..( -0.8000,  0.0000,  0.5236)--( -0.2394,  0.9709,  0.5236)..controls( -0.1786,  0.9859,  0.5035)and( -0.0987,  0.9973,  0.4783)..( -0.0331,  0.9995,  0.4581)--cycle;%[16]
\path[facet]( -0.9606, -0.2781,  0.5236)..controls( -0.9556, -0.2953,  0.5437)and( -0.9481, -0.3188,  0.5689)..( -0.9406, -0.3395,  0.5890)--( -0.8000,  0.0000,  0.5890)..controls( -0.8000,  0.0000,  0.5689)and( -0.8000,  0.0000,  0.5437)..( -0.8000,  0.0000,  0.5236)--cycle;%[17]
\path[facet]( -0.8000,  0.0000,  0.5236)..controls( -0.8000,  0.0000,  0.5437)and( -0.8000,  0.0000,  0.5689)..( -0.8000,  0.0000,  0.5890)--( -0.4251,  0.9052,  0.5890)..controls( -0.3720,  0.9301,  0.5689)and( -0.3002,  0.9559,  0.5437)..( -0.2394,  0.9709,  0.5236)--cycle;%[18]
\path[facet]( -0.9406, -0.3395,  0.5890)..controls( -0.9332, -0.3601,  0.6092)and( -0.9219, -0.3884,  0.6344)..( -0.9107, -0.4131,  0.6545)--( -0.8000,  0.0000,  0.6545)..controls( -0.8000,  0.0000,  0.6344)and( -0.8000,  0.0000,  0.6092)..( -0.8000,  0.0000,  0.5890)--cycle;%[19]
\path[facet]( -0.8000,  0.0000,  0.5890)..controls( -0.8000,  0.0000,  0.6092)and( -0.8000,  0.0000,  0.6344)..( -0.8000,  0.0000,  0.6545)--( -0.5821,  0.8131,  0.6545)..controls( -0.5387,  0.8442,  0.6344)and( -0.4781,  0.8803,  0.6092)..( -0.4251,  0.9052,  0.5890)--cycle;%[20]
\path[facet]( -0.9107, -0.4131,  0.6545)..controls( -0.8995, -0.4378,  0.6746)and( -0.8826, -0.4714,  0.6998)..( -0.8659, -0.5003,  0.7199)--( -0.8000,  0.0000,  0.7199)..controls( -0.8000,  0.0000,  0.6998)and( -0.8000,  0.0000,  0.6746)..( -0.8000,  0.0000,  0.6545)--cycle;%[21]
\path[facet]( -0.8000,  0.0000,  0.6545)..controls( -0.8000,  0.0000,  0.6746)and( -0.8000,  0.0000,  0.6998)..( -0.8000,  0.0000,  0.7199)--( -0.7069,  0.7073,  0.7199)..controls( -0.6735,  0.7407,  0.6998)and( -0.6256,  0.7820,  0.6746)..( -0.5821,  0.8131,  0.6545)--cycle;%[22]
\path[facet]( -0.8659, -0.5003,  0.7199)..controls( -0.8492, -0.5292,  0.7401)and( -0.8000, -0.6000,  0.7854)..( -0.8000, -0.6000,  0.7854)--( -0.8000,  0.0000,  0.7854)..controls( -0.8000,  0.0000,  0.7854)and( -0.8000,  0.0000,  0.7401)..( -0.8000,  0.0000,  0.7199)--cycle;%[23]
\path[facet]( -0.8000,  0.0000,  0.7199)..controls( -0.8000,  0.0000,  0.7401)and( -0.8000,  0.0000,  0.7854)..( -0.8000,  0.0000,  0.7854)--( -0.8000,  0.6000,  0.7854)..controls( -0.8000,  0.6000,  0.7854)and( -0.7403,  0.6740,  0.7401)..( -0.7069,  0.7073,  0.7199)--cycle;%[24]
\path[facet]( -0.8000, -0.6000,  0.7854)..controls( -0.8000, -0.6000,  0.7854)and( -0.8000, -0.6000,  0.7854)..( -0.8000, -0.6000,  0.7854)--( -0.8000,  0.0000,  0.7854)..controls( -0.8000,  0.0000,  0.7854)and( -0.8000,  0.0000,  0.7854)..( -0.8000,  0.0000,  0.7854)--cycle;%[25]
\path[facet]( -0.8000,  0.0000,  0.7854)..controls( -0.8000,  0.0000,  0.7854)and( -0.8000,  0.0000,  0.7854)..( -0.8000,  0.0000,  0.7854)--( -0.8000,  0.6000,  0.7854)..controls( -0.8000,  0.6000,  0.7854)and( -0.8000,  0.6000,  0.7854)..( -0.8000,  0.6000,  0.7854)--cycle;%[26]
\path[facet]( -0.8000, -0.6000,  0.7854)..controls( -0.8000, -0.6000,  0.7854)and( -0.7403, -0.6740,  0.8307)..( -0.7069, -0.7073,  0.8508)--( -0.8000,  0.0000,  0.8508)..controls( -0.8000,  0.0000,  0.8307)and( -0.8000,  0.0000,  0.7854)..( -0.8000,  0.0000,  0.7854)--cycle;%[27]
\path[facet]( -0.8000,  0.0000,  0.7854)..controls( -0.8000,  0.0000,  0.7854)and( -0.8000,  0.0000,  0.8307)..( -0.8000,  0.0000,  0.8508)--( -0.8659,  0.5003,  0.8508)..controls( -0.8492,  0.5292,  0.8307)and( -0.8000,  0.6000,  0.7854)..( -0.8000,  0.6000,  0.7854)--cycle;%[28]
\path[facet]( -0.7069, -0.7073,  0.8508)..controls( -0.6735, -0.7407,  0.8710)and( -0.6256, -0.7820,  0.8962)..( -0.5821, -0.8131,  0.9163)--( -0.8000,  0.0000,  0.9163)..controls( -0.8000,  0.0000,  0.8962)and( -0.8000,  0.0000,  0.8710)..( -0.8000,  0.0000,  0.8508)--cycle;%[29]
\path[facet]( -0.8000,  0.0000,  0.8508)..controls( -0.8000,  0.0000,  0.8710)and( -0.8000,  0.0000,  0.8962)..( -0.8000,  0.0000,  0.9163)--( -0.9107,  0.4131,  0.9163)..controls( -0.8995,  0.4378,  0.8962)and( -0.8826,  0.4714,  0.8710)..( -0.8659,  0.5003,  0.8508)--cycle;%[30]
\path[facet]( -0.5821, -0.8131,  0.9163)..controls( -0.5387, -0.8442,  0.9364)and( -0.4781, -0.8803,  0.9616)..( -0.4251, -0.9052,  0.9817)--( -0.8000,  0.0000,  0.9817)..controls( -0.8000,  0.0000,  0.9616)and( -0.8000,  0.0000,  0.9364)..( -0.8000,  0.0000,  0.9163)--cycle;%[31]
\path[facet]( -0.8000,  0.0000,  0.9163)..controls( -0.8000,  0.0000,  0.9364)and( -0.8000,  0.0000,  0.9616)..( -0.8000,  0.0000,  0.9817)--( -0.9406,  0.3395,  0.9817)..controls( -0.9332,  0.3601,  0.9616)and( -0.9219,  0.3884,  0.9364)..( -0.9107,  0.4131,  0.9163)--cycle;%[32]
\path[facet]( -0.4251, -0.9052,  0.9817)..controls( -0.3720, -0.9301,  1.0019)and( -0.3002, -0.9559,  1.0271)..( -0.2394, -0.9709,  1.0472)--( -0.8000,  0.0000,  1.0472)..controls( -0.8000,  0.0000,  1.0271)and( -0.8000,  0.0000,  1.0019)..( -0.8000,  0.0000,  0.9817)--cycle;%[33]
\path[facet]( -0.8000,  0.0000,  0.9817)..controls( -0.8000,  0.0000,  1.0019)and( -0.8000,  0.0000,  1.0271)..( -0.8000,  0.0000,  1.0472)--( -0.9606,  0.2781,  1.0472)..controls( -0.9556,  0.2953,  1.0271)and( -0.9481,  0.3188,  1.0019)..( -0.9406,  0.3395,  0.9817)--cycle;%[34]
\path[facet]( -0.2394, -0.9709,  1.0472)..controls( -0.1786, -0.9859,  1.0673)and( -0.0987, -0.9973,  1.0925)..( -0.0331, -0.9995,  1.1126)--( -0.8000,  0.0000,  1.1126)..controls( -0.8000,  0.0000,  1.0925)and( -0.8000,  0.0000,  1.0673)..( -0.8000,  0.0000,  1.0472)--cycle;%[35]
\path[facet]( -0.8000,  0.0000,  1.0472)..controls( -0.8000,  0.0000,  1.0673)and( -0.8000,  0.0000,  1.0925)..( -0.8000,  0.0000,  1.1126)--( -0.9740,  0.2267,  1.1126)..controls( -0.9706,  0.2412,  1.0925)and( -0.9655,  0.2609,  1.0673)..( -0.9606,  0.2781,  1.0472)--cycle;%[36]
\path[facet]( -0.0331, -0.9995,  1.1126)..controls(  0.0325, -1.0016,  1.1328)and(  0.1163, -0.9955,  1.1580)..(  0.1831, -0.9831,  1.1781)--( -0.8000,  0.0000,  1.1781)..controls( -0.8000,  0.0000,  1.1580)and( -0.8000,  0.0000,  1.1328)..( -0.8000,  0.0000,  1.1126)--cycle;%[37]
\path[facet]( -0.8000,  0.0000,  1.1126)..controls( -0.8000,  0.0000,  1.1328)and( -0.8000,  0.0000,  1.1580)..( -0.8000,  0.0000,  1.1781)--( -0.9831,  0.1831,  1.1781)..controls( -0.9808,  0.1955,  1.1580)and( -0.9773,  0.2122,  1.1328)..( -0.9740,  0.2267,  1.1126)--cycle;%[38]
\path[facet](  0.1831, -0.9831,  1.1781)..controls(  0.2499, -0.9707,  1.1982)and(  0.3326, -0.9456,  1.2234)..(  0.3964, -0.9181,  1.2435)--( -0.8000,  0.0000,  1.2435)..controls( -0.8000,  0.0000,  1.2234)and( -0.8000,  0.0000,  1.1982)..( -0.8000,  0.0000,  1.1781)--cycle;%[39]
\path[facet]( -0.8000,  0.0000,  1.1781)..controls( -0.8000,  0.0000,  1.1982)and( -0.8000,  0.0000,  1.2234)..( -0.8000,  0.0000,  1.2435)--( -0.9894,  0.1453,  1.2435)..controls( -0.9878,  0.1562,  1.2234)and( -0.9854,  0.1707,  1.1982)..( -0.9831,  0.1831,  1.1781)--cycle;%[40]
\path[facet](  0.3964, -0.9181,  1.2435)..controls(  0.4603, -0.8905,  1.2637)and(  0.5368, -0.8467,  1.2889)..(  0.5937, -0.8047,  1.3090)--( -0.8000,  0.0000,  1.3090)..controls( -0.8000,  0.0000,  1.2889)and( -0.8000,  0.0000,  1.2637)..( -0.8000,  0.0000,  1.2435)--cycle;%[41]
\path[facet]( -0.8000,  0.0000,  1.2435)..controls( -0.8000,  0.0000,  1.2637)and( -0.8000,  0.0000,  1.2889)..( -0.8000,  0.0000,  1.3090)--( -0.9937,  0.1118,  1.3090)..controls( -0.9926,  0.1216,  1.2889)and( -0.9910,  0.1344,  1.2637)..( -0.9894,  0.1453,  1.2435)--cycle;%[42]
\path[facet](  0.5937, -0.8047,  1.3090)..controls(  0.6506, -0.7627,  1.3291)and(  0.7161, -0.7017,  1.3543)..(  0.7624, -0.6472,  1.3744)--( -0.8000,  0.0000,  1.3744)..controls( -0.8000,  0.0000,  1.3543)and( -0.8000,  0.0000,  1.3291)..( -0.8000,  0.0000,  1.3090)--cycle;%[43]
\path[facet]( -0.8000,  0.0000,  1.3090)..controls( -0.8000,  0.0000,  1.3291)and( -0.8000,  0.0000,  1.3543)..( -0.8000,  0.0000,  1.3744)--( -0.9967,  0.0815,  1.3744)..controls( -0.9959,  0.0904,  1.3543)and( -0.9948,  0.1021,  1.3291)..( -0.9937,  0.1118,  1.3090)--cycle;%[44]
\path[facet](  0.7624, -0.6472,  1.3744)..controls(  0.8087, -0.5926,  1.3946)and(  0.8587, -0.5175,  1.4198)..(  0.8914, -0.4532,  1.4399)--( -0.8000,  0.0000,  1.4399)..controls( -0.8000,  0.0000,  1.4198)and( -0.8000,  0.0000,  1.3946)..( -0.8000,  0.0000,  1.3744)--cycle;%[45]
\path[facet]( -0.8000,  0.0000,  1.3744)..controls( -0.8000,  0.0000,  1.3946)and( -0.8000,  0.0000,  1.4198)..( -0.8000,  0.0000,  1.4399)--( -0.9986,  0.0532,  1.4399)..controls( -0.9981,  0.0617,  1.4198)and( -0.9974,  0.0725,  1.3946)..( -0.9967,  0.0815,  1.3744)--cycle;%[46]
\path[facet](  0.8914, -0.4532,  1.4399)..controls(  0.9241, -0.3889,  1.4600)and(  0.9555, -0.3037,  1.4852)..(  0.9724, -0.2333,  1.5053)--( -0.8000,  0.0000,  1.5053)..controls( -0.8000,  0.0000,  1.4852)and( -0.8000,  0.0000,  1.4600)..( -0.8000,  0.0000,  1.4399)--cycle;%[47]
\path[facet]( -0.8000,  0.0000,  1.4399)..controls( -0.8000,  0.0000,  1.4600)and( -0.8000,  0.0000,  1.4852)..( -0.8000,  0.0000,  1.5053)--( -0.9997,  0.0263,  1.5053)..controls( -0.9994,  0.0344,  1.4852)and( -0.9990,  0.0448,  1.4600)..( -0.9986,  0.0532,  1.4399)--cycle;%[48]
\path[facet](  0.9724, -0.2333,  1.5053)..controls(  0.9893, -0.1629,  1.5255)and(  1.0000, -0.0725,  1.5507)..(  1.0000, -0.0000,  1.5708)--( -0.8000,  0.0000,  1.5708)..controls( -0.8000,  0.0000,  1.5507)and( -0.8000,  0.0000,  1.5255)..( -0.8000,  0.0000,  1.5053)--cycle;%[49]
\path[facet]( -0.8000,  0.0000,  1.5053)..controls( -0.8000,  0.0000,  1.5255)and( -0.8000,  0.0000,  1.5507)..( -0.8000,  0.0000,  1.5708)--( -1.0000,  0.0000,  1.5708)..controls( -1.0000,  0.0081,  1.5507)and( -0.9999,  0.0181,  1.5255)..( -0.9997,  0.0263,  1.5053)--cycle;%[50]
\draw[axis](-1,0,\hoehe)--++(2,0,0) node[pos=0.55,above]{$\zeta_2$};
}\hfill
\cylinder{% Surface M
\path[facet]( -1.0000, -0.0000,  0.0000)..controls( -1.0000, -0.0403,  0.0201)and( -0.9967, -0.0906,  0.0453)..( -0.9914, -0.1305,  0.0654)--(  0.0000,  0.0000,  0.0654)..controls(  0.0000,  0.0000,  0.0453)and(  0.0000,  0.0000,  0.0201)..(  0.0000,  0.0000,  0.0000)--cycle;%[1]
\path[facet](  0.0000,  0.0000,  0.0000)..controls(  0.0000,  0.0000,  0.0201)and(  0.0000,  0.0000,  0.0453)..(  0.0000,  0.0000,  0.0654)--(  0.9914,  0.1305,  0.0654)..controls(  0.9967,  0.0906,  0.0453)and(  1.0000,  0.0403,  0.0201)..(  1.0000,  0.0000,  0.0000)--cycle;%[2]
\draw[axis](-1,0,0)--++(2,0,0);
\path[facet]( -0.9914, -0.1305,  0.0654)..controls( -0.9862, -0.1705,  0.0856)and( -0.9764, -0.2199,  0.1108)..( -0.9659, -0.2588,  0.1309)--(  0.0000,  0.0000,  0.1309)..controls(  0.0000,  0.0000,  0.1108)and(  0.0000,  0.0000,  0.0856)..(  0.0000,  0.0000,  0.0654)--cycle;%[3]
\path[facet](  0.0000,  0.0000,  0.0654)..controls(  0.0000,  0.0000,  0.0856)and(  0.0000,  0.0000,  0.1108)..(  0.0000,  0.0000,  0.1309)--(  0.9659,  0.2588,  0.1309)..controls(  0.9764,  0.2199,  0.1108)and(  0.9862,  0.1705,  0.0856)..(  0.9914,  0.1305,  0.0654)--cycle;%[4]
\path[facet]( -0.9659, -0.2588,  0.1309)..controls( -0.9555, -0.2977,  0.1510)and( -0.9393, -0.3455,  0.1762)..( -0.9239, -0.3827,  0.1963)--(  0.0000,  0.0000,  0.1963)..controls(  0.0000,  0.0000,  0.1762)and(  0.0000,  0.0000,  0.1510)..(  0.0000,  0.0000,  0.1309)--cycle;%[5]
\path[facet](  0.0000,  0.0000,  0.1309)..controls(  0.0000,  0.0000,  0.1510)and(  0.0000,  0.0000,  0.1762)..(  0.0000,  0.0000,  0.1963)--(  0.9239,  0.3827,  0.1963)..controls(  0.9393,  0.3455,  0.1762)and(  0.9555,  0.2977,  0.1510)..(  0.9659,  0.2588,  0.1309)--cycle;%[6]
\path[facet]( -0.9239, -0.3827,  0.1963)..controls( -0.9085, -0.4199,  0.2165)and( -0.8862, -0.4651,  0.2417)..( -0.8660, -0.5000,  0.2618)--(  0.0000,  0.0000,  0.2618)..controls(  0.0000,  0.0000,  0.2417)and(  0.0000,  0.0000,  0.2165)..(  0.0000,  0.0000,  0.1963)--cycle;%[7]
\path[facet](  0.0000,  0.0000,  0.1963)..controls(  0.0000,  0.0000,  0.2165)and(  0.0000,  0.0000,  0.2417)..(  0.0000,  0.0000,  0.2618)--(  0.8660,  0.5000,  0.2618)..controls(  0.8862,  0.4651,  0.2417)and(  0.9085,  0.4199,  0.2165)..(  0.9239,  0.3827,  0.1963)--cycle;%[8]
\path[facet]( -0.8660, -0.5000,  0.2618)..controls( -0.8459, -0.5349,  0.2819)and( -0.8179, -0.5768,  0.3071)..( -0.7934, -0.6088,  0.3272)--(  0.0000,  0.0000,  0.3272)..controls(  0.0000,  0.0000,  0.3071)and(  0.0000,  0.0000,  0.2819)..(  0.0000,  0.0000,  0.2618)--cycle;%[9]
\path[facet](  0.0000,  0.0000,  0.2618)..controls(  0.0000,  0.0000,  0.2819)and(  0.0000,  0.0000,  0.3071)..(  0.0000,  0.0000,  0.3272)--(  0.7934,  0.6088,  0.3272)..controls(  0.8179,  0.5768,  0.3071)and(  0.8459,  0.5349,  0.2819)..(  0.8660,  0.5000,  0.2618)--cycle;%[10]
\path[facet]( -0.7934, -0.6088,  0.3272)..controls( -0.7688, -0.6407,  0.3474)and( -0.7356, -0.6786,  0.3726)..( -0.7071, -0.7071,  0.3927)--(  0.0000,  0.0000,  0.3927)..controls(  0.0000,  0.0000,  0.3726)and(  0.0000,  0.0000,  0.3474)..(  0.0000,  0.0000,  0.3272)--cycle;%[11]
\path[facet](  0.0000,  0.0000,  0.3272)..controls(  0.0000,  0.0000,  0.3474)and(  0.0000,  0.0000,  0.3726)..(  0.0000,  0.0000,  0.3927)--(  0.7071,  0.7071,  0.3927)..controls(  0.7356,  0.6786,  0.3726)and(  0.7688,  0.6407,  0.3474)..(  0.7934,  0.6088,  0.3272)--cycle;%[12]
\path[facet]( -0.7071, -0.7071,  0.3927)..controls( -0.6786, -0.7356,  0.4128)and( -0.6407, -0.7688,  0.4380)..( -0.6088, -0.7934,  0.4581)--(  0.0000,  0.0000,  0.4581)..controls(  0.0000,  0.0000,  0.4380)and(  0.0000,  0.0000,  0.4128)..(  0.0000,  0.0000,  0.3927)--cycle;%[13]
\path[facet](  0.0000,  0.0000,  0.3927)..controls(  0.0000,  0.0000,  0.4128)and(  0.0000,  0.0000,  0.4380)..(  0.0000,  0.0000,  0.4581)--(  0.6088,  0.7934,  0.4581)..controls(  0.6407,  0.7688,  0.4380)and(  0.6786,  0.7356,  0.4128)..(  0.7071,  0.7071,  0.3927)--cycle;%[14]
\path[facet]( -0.6088, -0.7934,  0.4581)..controls( -0.5768, -0.8179,  0.4783)and( -0.5349, -0.8459,  0.5035)..( -0.5000, -0.8660,  0.5236)--(  0.0000,  0.0000,  0.5236)..controls(  0.0000,  0.0000,  0.5035)and(  0.0000,  0.0000,  0.4783)..(  0.0000,  0.0000,  0.4581)--cycle;%[15]
\path[facet](  0.0000,  0.0000,  0.4581)..controls(  0.0000,  0.0000,  0.4783)and(  0.0000,  0.0000,  0.5035)..(  0.0000,  0.0000,  0.5236)--(  0.5000,  0.8660,  0.5236)..controls(  0.5349,  0.8459,  0.5035)and(  0.5768,  0.8179,  0.4783)..(  0.6088,  0.7934,  0.4581)--cycle;%[16]
\path[facet]( -0.5000, -0.8660,  0.5236)..controls( -0.4651, -0.8862,  0.5437)and( -0.4199, -0.9085,  0.5689)..( -0.3827, -0.9239,  0.5890)--(  0.0000,  0.0000,  0.5890)..controls(  0.0000,  0.0000,  0.5689)and(  0.0000,  0.0000,  0.5437)..(  0.0000,  0.0000,  0.5236)--cycle;%[17]
\path[facet](  0.0000,  0.0000,  0.5236)..controls(  0.0000,  0.0000,  0.5437)and(  0.0000,  0.0000,  0.5689)..(  0.0000,  0.0000,  0.5890)--(  0.3827,  0.9239,  0.5890)..controls(  0.4199,  0.9085,  0.5689)and(  0.4651,  0.8862,  0.5437)..(  0.5000,  0.8660,  0.5236)--cycle;%[18]
\path[facet]( -0.3827, -0.9239,  0.5890)..controls( -0.3455, -0.9393,  0.6092)and( -0.2977, -0.9555,  0.6344)..( -0.2588, -0.9659,  0.6545)--(  0.0000,  0.0000,  0.6545)..controls(  0.0000,  0.0000,  0.6344)and(  0.0000,  0.0000,  0.6092)..(  0.0000,  0.0000,  0.5890)--cycle;%[19]
\path[facet](  0.0000,  0.0000,  0.5890)..controls(  0.0000,  0.0000,  0.6092)and(  0.0000,  0.0000,  0.6344)..(  0.0000,  0.0000,  0.6545)--(  0.2588,  0.9659,  0.6545)..controls(  0.2977,  0.9555,  0.6344)and(  0.3455,  0.9393,  0.6092)..(  0.3827,  0.9239,  0.5890)--cycle;%[20]
\path[facet]( -0.2588, -0.9659,  0.6545)..controls( -0.2199, -0.9764,  0.6746)and( -0.1705, -0.9862,  0.6998)..( -0.1305, -0.9914,  0.7199)--(  0.0000,  0.0000,  0.7199)..controls(  0.0000,  0.0000,  0.6998)and(  0.0000,  0.0000,  0.6746)..(  0.0000,  0.0000,  0.6545)--cycle;%[21]
\path[facet](  0.0000,  0.0000,  0.6545)..controls(  0.0000,  0.0000,  0.6746)and(  0.0000,  0.0000,  0.6998)..(  0.0000,  0.0000,  0.7199)--(  0.1305,  0.9914,  0.7199)..controls(  0.1705,  0.9862,  0.6998)and(  0.2199,  0.9764,  0.6746)..(  0.2588,  0.9659,  0.6545)--cycle;%[22]
\path[facet]( -0.1305, -0.9914,  0.7199)..controls( -0.0906, -0.9967,  0.7401)and( -0.0000, -1.0000,  0.7854)..( -0.0000, -1.0000,  0.7854)--(  0.0000,  0.0000,  0.7854)..controls(  0.0000,  0.0000,  0.7854)and(  0.0000,  0.0000,  0.7401)..(  0.0000,  0.0000,  0.7199)--cycle;%[23]
\path[facet](  0.0000,  0.0000,  0.7199)..controls(  0.0000,  0.0000,  0.7401)and(  0.0000,  0.0000,  0.7854)..(  0.0000,  0.0000,  0.7854)--(  0.0000,  1.0000,  0.7854)..controls(  0.0000,  1.0000,  0.7854)and(  0.0906,  0.9967,  0.7401)..(  0.1305,  0.9914,  0.7199)--cycle;%[24]
\path[facet]( -0.0000, -1.0000,  0.7854)..controls( -0.0000, -1.0000,  0.7854)and(  0.0000, -1.0000,  0.7854)..(  0.0000, -1.0000,  0.7854)--(  0.0000,  0.0000,  0.7854)..controls(  0.0000,  0.0000,  0.7854)and(  0.0000,  0.0000,  0.7854)..(  0.0000,  0.0000,  0.7854)--cycle;%[25]
\path[facet](  0.0000,  0.0000,  0.7854)..controls(  0.0000,  0.0000,  0.7854)and(  0.0000,  0.0000,  0.7854)..(  0.0000,  0.0000,  0.7854)--( -0.0000,  1.0000,  0.7854)..controls( -0.0000,  1.0000,  0.7854)and(  0.0000,  1.0000,  0.7854)..(  0.0000,  1.0000,  0.7854)--cycle;%[26]
\path[facet](  0.0000, -1.0000,  0.7854)..controls(  0.0000, -1.0000,  0.7854)and(  0.0906, -0.9967,  0.8307)..(  0.1305, -0.9914,  0.8508)--(  0.0000,  0.0000,  0.8508)..controls(  0.0000,  0.0000,  0.8307)and(  0.0000,  0.0000,  0.7854)..(  0.0000,  0.0000,  0.7854)--cycle;%[27]
\path[facet](  0.0000,  0.0000,  0.7854)..controls(  0.0000,  0.0000,  0.7854)and(  0.0000,  0.0000,  0.8307)..(  0.0000,  0.0000,  0.8508)--( -0.1305,  0.9914,  0.8508)..controls( -0.0906,  0.9967,  0.8307)and( -0.0000,  1.0000,  0.7854)..( -0.0000,  1.0000,  0.7854)--cycle;%[28]
\path[facet](  0.1305, -0.9914,  0.8508)..controls(  0.1705, -0.9862,  0.8710)and(  0.2199, -0.9764,  0.8962)..(  0.2588, -0.9659,  0.9163)--(  0.0000,  0.0000,  0.9163)..controls(  0.0000,  0.0000,  0.8962)and(  0.0000,  0.0000,  0.8710)..(  0.0000,  0.0000,  0.8508)--cycle;%[29]
\path[facet](  0.0000,  0.0000,  0.8508)..controls(  0.0000,  0.0000,  0.8710)and(  0.0000,  0.0000,  0.8962)..(  0.0000,  0.0000,  0.9163)--( -0.2588,  0.9659,  0.9163)..controls( -0.2199,  0.9764,  0.8962)and( -0.1705,  0.9862,  0.8710)..( -0.1305,  0.9914,  0.8508)--cycle;%[30]
\path[facet](  0.2588, -0.9659,  0.9163)..controls(  0.2977, -0.9555,  0.9364)and(  0.3455, -0.9393,  0.9616)..(  0.3827, -0.9239,  0.9817)--(  0.0000,  0.0000,  0.9817)..controls(  0.0000,  0.0000,  0.9616)and(  0.0000,  0.0000,  0.9364)..(  0.0000,  0.0000,  0.9163)--cycle;%[31]
\path[facet](  0.0000,  0.0000,  0.9163)..controls(  0.0000,  0.0000,  0.9364)and(  0.0000,  0.0000,  0.9616)..(  0.0000,  0.0000,  0.9817)--( -0.3827,  0.9239,  0.9817)..controls( -0.3455,  0.9393,  0.9616)and( -0.2977,  0.9555,  0.9364)..( -0.2588,  0.9659,  0.9163)--cycle;%[32]
\path[facet](  0.3827, -0.9239,  0.9817)..controls(  0.4199, -0.9085,  1.0019)and(  0.4651, -0.8862,  1.0271)..(  0.5000, -0.8660,  1.0472)--(  0.0000,  0.0000,  1.0472)..controls(  0.0000,  0.0000,  1.0271)and(  0.0000,  0.0000,  1.0019)..(  0.0000,  0.0000,  0.9817)--cycle;%[33]
\path[facet](  0.0000,  0.0000,  0.9817)..controls(  0.0000,  0.0000,  1.0019)and(  0.0000,  0.0000,  1.0271)..(  0.0000,  0.0000,  1.0472)--( -0.5000,  0.8660,  1.0472)..controls( -0.4651,  0.8862,  1.0271)and( -0.4199,  0.9085,  1.0019)..( -0.3827,  0.9239,  0.9817)--cycle;%[34]
\path[facet](  0.5000, -0.8660,  1.0472)..controls(  0.5349, -0.8459,  1.0673)and(  0.5768, -0.8179,  1.0925)..(  0.6088, -0.7934,  1.1126)--(  0.0000,  0.0000,  1.1126)..controls(  0.0000,  0.0000,  1.0925)and(  0.0000,  0.0000,  1.0673)..(  0.0000,  0.0000,  1.0472)--cycle;%[35]
\path[facet](  0.0000,  0.0000,  1.0472)..controls(  0.0000,  0.0000,  1.0673)and(  0.0000,  0.0000,  1.0925)..(  0.0000,  0.0000,  1.1126)--( -0.6088,  0.7934,  1.1126)..controls( -0.5768,  0.8179,  1.0925)and( -0.5349,  0.8459,  1.0673)..( -0.5000,  0.8660,  1.0472)--cycle;%[36]
\path[facet](  0.6088, -0.7934,  1.1126)..controls(  0.6407, -0.7688,  1.1328)and(  0.6786, -0.7356,  1.1580)..(  0.7071, -0.7071,  1.1781)--(  0.0000,  0.0000,  1.1781)..controls(  0.0000,  0.0000,  1.1580)and(  0.0000,  0.0000,  1.1328)..(  0.0000,  0.0000,  1.1126)--cycle;%[37]
\path[facet](  0.0000,  0.0000,  1.1126)..controls(  0.0000,  0.0000,  1.1328)and(  0.0000,  0.0000,  1.1580)..(  0.0000,  0.0000,  1.1781)--( -0.7071,  0.7071,  1.1781)..controls( -0.6786,  0.7356,  1.1580)and( -0.6407,  0.7688,  1.1328)..( -0.6088,  0.7934,  1.1126)--cycle;%[38]
\path[facet](  0.7071, -0.7071,  1.1781)..controls(  0.7356, -0.6786,  1.1982)and(  0.7688, -0.6407,  1.2234)..(  0.7934, -0.6088,  1.2435)--(  0.0000,  0.0000,  1.2435)..controls(  0.0000,  0.0000,  1.2234)and(  0.0000,  0.0000,  1.1982)..(  0.0000,  0.0000,  1.1781)--cycle;%[39]
\path[facet](  0.0000,  0.0000,  1.1781)..controls(  0.0000,  0.0000,  1.1982)and(  0.0000,  0.0000,  1.2234)..(  0.0000,  0.0000,  1.2435)--( -0.7934,  0.6088,  1.2435)..controls( -0.7688,  0.6407,  1.2234)and( -0.7356,  0.6786,  1.1982)..( -0.7071,  0.7071,  1.1781)--cycle;%[40]
\path[facet](  0.7934, -0.6088,  1.2435)..controls(  0.8179, -0.5768,  1.2637)and(  0.8459, -0.5349,  1.2889)..(  0.8660, -0.5000,  1.3090)--(  0.0000,  0.0000,  1.3090)..controls(  0.0000,  0.0000,  1.2889)and(  0.0000,  0.0000,  1.2637)..(  0.0000,  0.0000,  1.2435)--cycle;%[41]
\path[facet](  0.0000,  0.0000,  1.2435)..controls(  0.0000,  0.0000,  1.2637)and(  0.0000,  0.0000,  1.2889)..(  0.0000,  0.0000,  1.3090)--( -0.8660,  0.5000,  1.3090)..controls( -0.8459,  0.5349,  1.2889)and( -0.8179,  0.5768,  1.2637)..( -0.7934,  0.6088,  1.2435)--cycle;%[42]
\path[facet](  0.8660, -0.5000,  1.3090)..controls(  0.8862, -0.4651,  1.3291)and(  0.9085, -0.4199,  1.3543)..(  0.9239, -0.3827,  1.3744)--(  0.0000,  0.0000,  1.3744)..controls(  0.0000,  0.0000,  1.3543)and(  0.0000,  0.0000,  1.3291)..(  0.0000,  0.0000,  1.3090)--cycle;%[43]
\path[facet](  0.0000,  0.0000,  1.3090)..controls(  0.0000,  0.0000,  1.3291)and(  0.0000,  0.0000,  1.3543)..(  0.0000,  0.0000,  1.3744)--( -0.9239,  0.3827,  1.3744)..controls( -0.9085,  0.4199,  1.3543)and( -0.8862,  0.4651,  1.3291)..( -0.8660,  0.5000,  1.3090)--cycle;%[44]
\path[facet](  0.9239, -0.3827,  1.3744)..controls(  0.9393, -0.3455,  1.3946)and(  0.9555, -0.2977,  1.4198)..(  0.9659, -0.2588,  1.4399)--(  0.0000,  0.0000,  1.4399)..controls(  0.0000,  0.0000,  1.4198)and(  0.0000,  0.0000,  1.3946)..(  0.0000,  0.0000,  1.3744)--cycle;%[45]
\path[facet](  0.0000,  0.0000,  1.3744)..controls(  0.0000,  0.0000,  1.3946)and(  0.0000,  0.0000,  1.4198)..(  0.0000,  0.0000,  1.4399)--( -0.9659,  0.2588,  1.4399)..controls( -0.9555,  0.2977,  1.4198)and( -0.9393,  0.3455,  1.3946)..( -0.9239,  0.3827,  1.3744)--cycle;%[46]
\path[facet](  0.9659, -0.2588,  1.4399)..controls(  0.9764, -0.2199,  1.4600)and(  0.9862, -0.1705,  1.4852)..(  0.9914, -0.1305,  1.5053)--(  0.0000,  0.0000,  1.5053)..controls(  0.0000,  0.0000,  1.4852)and(  0.0000,  0.0000,  1.4600)..(  0.0000,  0.0000,  1.4399)--cycle;%[47]
\path[facet](  0.0000,  0.0000,  1.4399)..controls(  0.0000,  0.0000,  1.4600)and(  0.0000,  0.0000,  1.4852)..(  0.0000,  0.0000,  1.5053)--( -0.9914,  0.1305,  1.5053)..controls( -0.9862,  0.1705,  1.4852)and( -0.9764,  0.2199,  1.4600)..( -0.9659,  0.2588,  1.4399)--cycle;%[48]
\path[facet](  0.9914, -0.1305,  1.5053)..controls(  0.9967, -0.0906,  1.5255)and(  1.0000, -0.0403,  1.5507)..(  1.0000, -0.0000,  1.5708)--(  0.0000,  0.0000,  1.5708)..controls(  0.0000,  0.0000,  1.5507)and(  0.0000,  0.0000,  1.5255)..(  0.0000,  0.0000,  1.5053)--cycle;%[49]
\path[facet](  0.0000,  0.0000,  1.5053)..controls(  0.0000,  0.0000,  1.5255)and(  0.0000,  0.0000,  1.5507)..(  0.0000,  0.0000,  1.5708)--( -1.0000,  0.0000,  1.5708)..controls( -1.0000,  0.0403,  1.5507)and( -0.9967,  0.0906,  1.5255)..( -0.9914,  0.1305,  1.5053)--cycle;%[50]
\draw[axis](-1,0,\hoehe)--++(2,0,0);
}\hfill
\cylinder{% Surface R
\path[facet]( -1.0000, -0.0000,  0.0000)..controls( -1.0000, -0.0725,  0.0201)and( -0.9893, -0.1629,  0.0453)..( -0.9724, -0.2333,  0.0654)--(  0.8000,  0.0000,  0.0654)..controls(  0.8000,  0.0000,  0.0453)and(  0.8000,  0.0000,  0.0201)..(  0.8000,  0.0000,  0.0000)--cycle;%[1]
\path[facet](  0.8000,  0.0000,  0.0000)..controls(  0.8000,  0.0000,  0.0201)and(  0.8000,  0.0000,  0.0453)..(  0.8000,  0.0000,  0.0654)--(  0.9997,  0.0263,  0.0654)..controls(  0.9999,  0.0181,  0.0453)and(  1.0000,  0.0081,  0.0201)..(  1.0000,  0.0000,  0.0000)--cycle;%[2]
\draw[axis](-1,0,0)--++(2,0,0);
\path[facet]( -0.9724, -0.2333,  0.0654)..controls( -0.9555, -0.3037,  0.0856)and( -0.9241, -0.3889,  0.1108)..( -0.8914, -0.4532,  0.1309)--(  0.8000,  0.0000,  0.1309)..controls(  0.8000,  0.0000,  0.1108)and(  0.8000,  0.0000,  0.0856)..(  0.8000,  0.0000,  0.0654)--cycle;%[3]
\path[facet](  0.8000,  0.0000,  0.0654)..controls(  0.8000,  0.0000,  0.0856)and(  0.8000,  0.0000,  0.1108)..(  0.8000,  0.0000,  0.1309)--(  0.9986,  0.0532,  0.1309)..controls(  0.9990,  0.0448,  0.1108)and(  0.9994,  0.0344,  0.0856)..(  0.9997,  0.0263,  0.0654)--cycle;%[4]
\path[facet]( -0.8914, -0.4532,  0.1309)..controls( -0.8587, -0.5175,  0.1510)and( -0.8087, -0.5926,  0.1762)..( -0.7624, -0.6472,  0.1963)--(  0.8000,  0.0000,  0.1963)..controls(  0.8000,  0.0000,  0.1762)and(  0.8000,  0.0000,  0.1510)..(  0.8000,  0.0000,  0.1309)--cycle;%[5]
\path[facet](  0.8000,  0.0000,  0.1309)..controls(  0.8000,  0.0000,  0.1510)and(  0.8000,  0.0000,  0.1762)..(  0.8000,  0.0000,  0.1963)--(  0.9967,  0.0815,  0.1963)..controls(  0.9974,  0.0725,  0.1762)and(  0.9981,  0.0617,  0.1510)..(  0.9986,  0.0532,  0.1309)--cycle;%[6]
\path[facet]( -0.7624, -0.6472,  0.1963)..controls( -0.7161, -0.7017,  0.2165)and( -0.6506, -0.7627,  0.2417)..( -0.5937, -0.8047,  0.2618)--(  0.8000,  0.0000,  0.2618)..controls(  0.8000,  0.0000,  0.2417)and(  0.8000,  0.0000,  0.2165)..(  0.8000,  0.0000,  0.1963)--cycle;%[7]
\path[facet](  0.8000,  0.0000,  0.1963)..controls(  0.8000,  0.0000,  0.2165)and(  0.8000,  0.0000,  0.2417)..(  0.8000,  0.0000,  0.2618)--(  0.9937,  0.1118,  0.2618)..controls(  0.9948,  0.1021,  0.2417)and(  0.9959,  0.0904,  0.2165)..(  0.9967,  0.0815,  0.1963)--cycle;%[8]
\path[facet]( -0.5937, -0.8047,  0.2618)..controls( -0.5368, -0.8467,  0.2819)and( -0.4603, -0.8905,  0.3071)..( -0.3964, -0.9181,  0.3272)--(  0.8000,  0.0000,  0.3272)..controls(  0.8000,  0.0000,  0.3071)and(  0.8000,  0.0000,  0.2819)..(  0.8000,  0.0000,  0.2618)--cycle;%[9]
\path[facet](  0.8000,  0.0000,  0.2618)..controls(  0.8000,  0.0000,  0.2819)and(  0.8000,  0.0000,  0.3071)..(  0.8000,  0.0000,  0.3272)--(  0.9894,  0.1453,  0.3272)..controls(  0.9910,  0.1344,  0.3071)and(  0.9926,  0.1216,  0.2819)..(  0.9937,  0.1118,  0.2618)--cycle;%[10]
\path[facet]( -0.3964, -0.9181,  0.3272)..controls( -0.3326, -0.9456,  0.3474)and( -0.2499, -0.9707,  0.3726)..( -0.1831, -0.9831,  0.3927)--(  0.8000,  0.0000,  0.3927)..controls(  0.8000,  0.0000,  0.3726)and(  0.8000,  0.0000,  0.3474)..(  0.8000,  0.0000,  0.3272)--cycle;%[11]
\path[facet](  0.8000,  0.0000,  0.3272)..controls(  0.8000,  0.0000,  0.3474)and(  0.8000,  0.0000,  0.3726)..(  0.8000,  0.0000,  0.3927)--(  0.9831,  0.1831,  0.3927)..controls(  0.9854,  0.1707,  0.3726)and(  0.9878,  0.1562,  0.3474)..(  0.9894,  0.1453,  0.3272)--cycle;%[12]
\path[facet]( -0.1831, -0.9831,  0.3927)..controls( -0.1163, -0.9955,  0.4128)and( -0.0325, -1.0016,  0.4380)..(  0.0331, -0.9995,  0.4581)--(  0.8000,  0.0000,  0.4581)..controls(  0.8000,  0.0000,  0.4380)and(  0.8000,  0.0000,  0.4128)..(  0.8000,  0.0000,  0.3927)--cycle;%[13]
\path[facet](  0.8000,  0.0000,  0.3927)..controls(  0.8000,  0.0000,  0.4128)and(  0.8000,  0.0000,  0.4380)..(  0.8000,  0.0000,  0.4581)--(  0.9740,  0.2267,  0.4581)..controls(  0.9773,  0.2122,  0.4380)and(  0.9808,  0.1955,  0.4128)..(  0.9831,  0.1831,  0.3927)--cycle;%[14]
\path[facet](  0.0331, -0.9995,  0.4581)..controls(  0.0987, -0.9973,  0.4783)and(  0.1786, -0.9859,  0.5035)..(  0.2394, -0.9709,  0.5236)--(  0.8000,  0.0000,  0.5236)..controls(  0.8000,  0.0000,  0.5035)and(  0.8000,  0.0000,  0.4783)..(  0.8000,  0.0000,  0.4581)--cycle;%[15]
\path[facet](  0.8000,  0.0000,  0.4581)..controls(  0.8000,  0.0000,  0.4783)and(  0.8000,  0.0000,  0.5035)..(  0.8000,  0.0000,  0.5236)--(  0.9606,  0.2781,  0.5236)..controls(  0.9655,  0.2609,  0.5035)and(  0.9706,  0.2412,  0.4783)..(  0.9740,  0.2267,  0.4581)--cycle;%[16]
\path[facet](  0.2394, -0.9709,  0.5236)..controls(  0.3002, -0.9559,  0.5437)and(  0.3720, -0.9301,  0.5689)..(  0.4251, -0.9052,  0.5890)--(  0.8000,  0.0000,  0.5890)..controls(  0.8000,  0.0000,  0.5689)and(  0.8000,  0.0000,  0.5437)..(  0.8000,  0.0000,  0.5236)--cycle;%[17]
\path[facet](  0.8000,  0.0000,  0.5236)..controls(  0.8000,  0.0000,  0.5437)and(  0.8000,  0.0000,  0.5689)..(  0.8000,  0.0000,  0.5890)--(  0.9406,  0.3395,  0.5890)..controls(  0.9481,  0.3188,  0.5689)and(  0.9556,  0.2953,  0.5437)..(  0.9606,  0.2781,  0.5236)--cycle;%[18]
\path[facet](  0.4251, -0.9052,  0.5890)..controls(  0.4781, -0.8803,  0.6092)and(  0.5387, -0.8442,  0.6344)..(  0.5821, -0.8131,  0.6545)--(  0.8000,  0.0000,  0.6545)..controls(  0.8000,  0.0000,  0.6344)and(  0.8000,  0.0000,  0.6092)..(  0.8000,  0.0000,  0.5890)--cycle;%[19]
\path[facet](  0.8000,  0.0000,  0.5890)..controls(  0.8000,  0.0000,  0.6092)and(  0.8000,  0.0000,  0.6344)..(  0.8000,  0.0000,  0.6545)--(  0.9107,  0.4131,  0.6545)..controls(  0.9219,  0.3884,  0.6344)and(  0.9332,  0.3601,  0.6092)..(  0.9406,  0.3395,  0.5890)--cycle;%[20]
\path[facet](  0.5821, -0.8131,  0.6545)..controls(  0.6256, -0.7820,  0.6746)and(  0.6735, -0.7407,  0.6998)..(  0.7069, -0.7073,  0.7199)--(  0.8000,  0.0000,  0.7199)..controls(  0.8000,  0.0000,  0.6998)and(  0.8000,  0.0000,  0.6746)..(  0.8000,  0.0000,  0.6545)--cycle;%[21]
\path[facet](  0.8000,  0.0000,  0.6545)..controls(  0.8000,  0.0000,  0.6746)and(  0.8000,  0.0000,  0.6998)..(  0.8000,  0.0000,  0.7199)--(  0.8659,  0.5003,  0.7199)..controls(  0.8826,  0.4714,  0.6998)and(  0.8995,  0.4378,  0.6746)..(  0.9107,  0.4131,  0.6545)--cycle;%[22]
\path[facet](  0.7069, -0.7073,  0.7199)..controls(  0.7403, -0.6740,  0.7401)and(  0.8000, -0.6000,  0.7854)..(  0.8000, -0.6000,  0.7854)--(  0.8000,  0.0000,  0.7854)..controls(  0.8000,  0.0000,  0.7854)and(  0.8000,  0.0000,  0.7401)..(  0.8000,  0.0000,  0.7199)--cycle;%[23]
\path[facet](  0.8000,  0.0000,  0.7199)..controls(  0.8000,  0.0000,  0.7401)and(  0.8000,  0.0000,  0.7854)..(  0.8000,  0.0000,  0.7854)--(  0.8000,  0.6000,  0.7854)..controls(  0.8000,  0.6000,  0.7854)and(  0.8492,  0.5292,  0.7401)..(  0.8659,  0.5003,  0.7199)--cycle;%[24]
\path[facet](  0.8000, -0.6000,  0.7854)..controls(  0.8000, -0.6000,  0.7854)and(  0.8000, -0.6000,  0.7854)..(  0.8000, -0.6000,  0.7854)--(  0.8000,  0.0000,  0.7854)..controls(  0.8000,  0.0000,  0.7854)and(  0.8000,  0.0000,  0.7854)..(  0.8000,  0.0000,  0.7854)--cycle;%[25]
\path[facet](  0.8000,  0.0000,  0.7854)..controls(  0.8000,  0.0000,  0.7854)and(  0.8000,  0.0000,  0.7854)..(  0.8000,  0.0000,  0.7854)--(  0.8000,  0.6000,  0.7854)..controls(  0.8000,  0.6000,  0.7854)and(  0.8000,  0.6000,  0.7854)..(  0.8000,  0.6000,  0.7854)--cycle;%[26]
\path[facet](  0.8000, -0.6000,  0.7854)..controls(  0.8000, -0.6000,  0.7854)and(  0.8492, -0.5292,  0.8307)..(  0.8659, -0.5003,  0.8508)--(  0.8000,  0.0000,  0.8508)..controls(  0.8000,  0.0000,  0.8307)and(  0.8000,  0.0000,  0.7854)..(  0.8000,  0.0000,  0.7854)--cycle;%[27]
\path[facet](  0.8000,  0.0000,  0.7854)..controls(  0.8000,  0.0000,  0.7854)and(  0.8000,  0.0000,  0.8307)..(  0.8000,  0.0000,  0.8508)--(  0.7069,  0.7073,  0.8508)..controls(  0.7403,  0.6740,  0.8307)and(  0.8000,  0.6000,  0.7854)..(  0.8000,  0.6000,  0.7854)--cycle;%[28]
\path[facet](  0.8659, -0.5003,  0.8508)..controls(  0.8826, -0.4714,  0.8710)and(  0.8995, -0.4378,  0.8962)..(  0.9107, -0.4131,  0.9163)--(  0.8000,  0.0000,  0.9163)..controls(  0.8000,  0.0000,  0.8962)and(  0.8000,  0.0000,  0.8710)..(  0.8000,  0.0000,  0.8508)--cycle;%[29]
\path[facet](  0.8000,  0.0000,  0.8508)..controls(  0.8000,  0.0000,  0.8710)and(  0.8000,  0.0000,  0.8962)..(  0.8000,  0.0000,  0.9163)--(  0.5821,  0.8131,  0.9163)..controls(  0.6256,  0.7820,  0.8962)and(  0.6735,  0.7407,  0.8710)..(  0.7069,  0.7073,  0.8508)--cycle;%[30]
\path[facet](  0.9107, -0.4131,  0.9163)..controls(  0.9219, -0.3884,  0.9364)and(  0.9332, -0.3601,  0.9616)..(  0.9406, -0.3395,  0.9817)--(  0.8000,  0.0000,  0.9817)..controls(  0.8000,  0.0000,  0.9616)and(  0.8000,  0.0000,  0.9364)..(  0.8000,  0.0000,  0.9163)--cycle;%[31]
\path[facet](  0.8000,  0.0000,  0.9163)..controls(  0.8000,  0.0000,  0.9364)and(  0.8000,  0.0000,  0.9616)..(  0.8000,  0.0000,  0.9817)--(  0.4251,  0.9052,  0.9817)..controls(  0.4781,  0.8803,  0.9616)and(  0.5387,  0.8442,  0.9364)..(  0.5821,  0.8131,  0.9163)--cycle;%[32]
\path[facet](  0.9406, -0.3395,  0.9817)..controls(  0.9481, -0.3188,  1.0019)and(  0.9556, -0.2953,  1.0271)..(  0.9606, -0.2781,  1.0472)--(  0.8000,  0.0000,  1.0472)..controls(  0.8000,  0.0000,  1.0271)and(  0.8000,  0.0000,  1.0019)..(  0.8000,  0.0000,  0.9817)--cycle;%[33]
\path[facet](  0.8000,  0.0000,  0.9817)..controls(  0.8000,  0.0000,  1.0019)and(  0.8000,  0.0000,  1.0271)..(  0.8000,  0.0000,  1.0472)--(  0.2394,  0.9709,  1.0472)..controls(  0.3002,  0.9559,  1.0271)and(  0.3720,  0.9301,  1.0019)..(  0.4251,  0.9052,  0.9817)--cycle;%[34]
\path[facet](  0.9606, -0.2781,  1.0472)..controls(  0.9655, -0.2609,  1.0673)and(  0.9706, -0.2412,  1.0925)..(  0.9740, -0.2267,  1.1126)--(  0.8000,  0.0000,  1.1126)..controls(  0.8000,  0.0000,  1.0925)and(  0.8000,  0.0000,  1.0673)..(  0.8000,  0.0000,  1.0472)--cycle;%[35]
\path[facet](  0.8000,  0.0000,  1.0472)..controls(  0.8000,  0.0000,  1.0673)and(  0.8000,  0.0000,  1.0925)..(  0.8000,  0.0000,  1.1126)--(  0.0331,  0.9995,  1.1126)..controls(  0.0987,  0.9973,  1.0925)and(  0.1786,  0.9859,  1.0673)..(  0.2394,  0.9709,  1.0472)--cycle;%[36]
\path[facet](  0.9740, -0.2267,  1.1126)..controls(  0.9773, -0.2122,  1.1328)and(  0.9808, -0.1955,  1.1580)..(  0.9831, -0.1831,  1.1781)--(  0.8000,  0.0000,  1.1781)..controls(  0.8000,  0.0000,  1.1580)and(  0.8000,  0.0000,  1.1328)..(  0.8000,  0.0000,  1.1126)--cycle;%[37]
\path[facet](  0.8000,  0.0000,  1.1126)..controls(  0.8000,  0.0000,  1.1328)and(  0.8000,  0.0000,  1.1580)..(  0.8000,  0.0000,  1.1781)--( -0.1831,  0.9831,  1.1781)..controls( -0.1163,  0.9955,  1.1580)and( -0.0325,  1.0016,  1.1328)..(  0.0331,  0.9995,  1.1126)--cycle;%[38]
\path[facet](  0.9831, -0.1831,  1.1781)..controls(  0.9854, -0.1707,  1.1982)and(  0.9878, -0.1562,  1.2234)..(  0.9894, -0.1453,  1.2435)--(  0.8000,  0.0000,  1.2435)..controls(  0.8000,  0.0000,  1.2234)and(  0.8000,  0.0000,  1.1982)..(  0.8000,  0.0000,  1.1781)--cycle;%[39]
\path[facet](  0.8000,  0.0000,  1.1781)..controls(  0.8000,  0.0000,  1.1982)and(  0.8000,  0.0000,  1.2234)..(  0.8000,  0.0000,  1.2435)--( -0.3964,  0.9181,  1.2435)..controls( -0.3326,  0.9456,  1.2234)and( -0.2499,  0.9707,  1.1982)..( -0.1831,  0.9831,  1.1781)--cycle;%[40]
\path[facet](  0.9894, -0.1453,  1.2435)..controls(  0.9910, -0.1344,  1.2637)and(  0.9926, -0.1216,  1.2889)..(  0.9937, -0.1118,  1.3090)--(  0.8000,  0.0000,  1.3090)..controls(  0.8000,  0.0000,  1.2889)and(  0.8000,  0.0000,  1.2637)..(  0.8000,  0.0000,  1.2435)--cycle;%[41]
\path[facet](  0.8000,  0.0000,  1.2435)..controls(  0.8000,  0.0000,  1.2637)and(  0.8000,  0.0000,  1.2889)..(  0.8000,  0.0000,  1.3090)--( -0.5937,  0.8047,  1.3090)..controls( -0.5368,  0.8467,  1.2889)and( -0.4603,  0.8905,  1.2637)..( -0.3964,  0.9181,  1.2435)--cycle;%[42]
\path[facet](  0.9937, -0.1118,  1.3090)..controls(  0.9948, -0.1021,  1.3291)and(  0.9959, -0.0904,  1.3543)..(  0.9967, -0.0815,  1.3744)--(  0.8000,  0.0000,  1.3744)..controls(  0.8000,  0.0000,  1.3543)and(  0.8000,  0.0000,  1.3291)..(  0.8000,  0.0000,  1.3090)--cycle;%[43]
\path[facet](  0.8000,  0.0000,  1.3090)..controls(  0.8000,  0.0000,  1.3291)and(  0.8000,  0.0000,  1.3543)..(  0.8000,  0.0000,  1.3744)--( -0.7624,  0.6472,  1.3744)..controls( -0.7161,  0.7017,  1.3543)and( -0.6506,  0.7627,  1.3291)..( -0.5937,  0.8047,  1.3090)--cycle;%[44]
\path[facet](  0.9967, -0.0815,  1.3744)..controls(  0.9974, -0.0725,  1.3946)and(  0.9981, -0.0617,  1.4198)..(  0.9986, -0.0532,  1.4399)--(  0.8000,  0.0000,  1.4399)..controls(  0.8000,  0.0000,  1.4198)and(  0.8000,  0.0000,  1.3946)..(  0.8000,  0.0000,  1.3744)--cycle;%[45]
\path[facet](  0.8000,  0.0000,  1.3744)..controls(  0.8000,  0.0000,  1.3946)and(  0.8000,  0.0000,  1.4198)..(  0.8000,  0.0000,  1.4399)--( -0.8914,  0.4532,  1.4399)..controls( -0.8587,  0.5175,  1.4198)and( -0.8087,  0.5926,  1.3946)..( -0.7624,  0.6472,  1.3744)--cycle;%[46]
\path[facet](  0.9986, -0.0532,  1.4399)..controls(  0.9990, -0.0448,  1.4600)and(  0.9994, -0.0344,  1.4852)..(  0.9997, -0.0263,  1.5053)--(  0.8000,  0.0000,  1.5053)..controls(  0.8000,  0.0000,  1.4852)and(  0.8000,  0.0000,  1.4600)..(  0.8000,  0.0000,  1.4399)--cycle;%[47]
\path[facet](  0.8000,  0.0000,  1.4399)..controls(  0.8000,  0.0000,  1.4600)and(  0.8000,  0.0000,  1.4852)..(  0.8000,  0.0000,  1.5053)--( -0.9724,  0.2333,  1.5053)..controls( -0.9555,  0.3037,  1.4852)and( -0.9241,  0.3889,  1.4600)..( -0.8914,  0.4532,  1.4399)--cycle;%[48]
\path[facet](  0.9997, -0.0263,  1.5053)..controls(  0.9999, -0.0181,  1.5255)and(  1.0000, -0.0081,  1.5507)..(  1.0000, -0.0000,  1.5708)--(  0.8000,  0.0000,  1.5708)..controls(  0.8000,  0.0000,  1.5507)and(  0.8000,  0.0000,  1.5255)..(  0.8000,  0.0000,  1.5053)--cycle;%[49]
\path[facet](  0.8000,  0.0000,  1.5053)..controls(  0.8000,  0.0000,  1.5255)and(  0.8000,  0.0000,  1.5507)..(  0.8000,  0.0000,  1.5708)--( -1.0000,  0.0000,  1.5708)..controls( -1.0000,  0.0725,  1.5507)and( -0.9893,  0.1629,  1.5255)..( -0.9724,  0.2333,  1.5053)--cycle;%[50]
\draw[latex-latex,very thick](-1,0,0)--++(0,0,pi/2)node[pos=0.66,right]{$\height\pi$};
\draw[axis](-1,0,\hoehe)--++(2,0,0);
}
\caption{Horizontally translated helicoids $\heli_{\height,s}$ inside the cylinder $Z_\height$.}%
\label{fig:helicoid_translation}%
\end{figure}

\begin{proof} 
Given $-1<s<1$, let $\heli_{\height,s}\subset Z_\height$ be the horizontally translated helicoid introduced in Lemma \ref{lem:translation} (see Figure \ref{fig:helicoid_translation}). 
By construction, $\heli_{\height,s}$ contains the segment $\zeta_0$. 
Moreover, the set $c_{\height,s}\vcentcolon=\heli_{\height,s}\cap\{x_3=\height\pi/2\}$ is a straight segment of length $\delta(s)$ which intersects $\zeta_1$ orthogonally (see Figure \ref{fig:vertical_strip}, left image).
It divides $\heli_{\height,s}$ into two connected components which are congruent via rotation $\rotation_{\zeta_1}^{\pi}$ by angle $\pi$ around $\zeta_1$. 
Let $\Lambda_s$ be the connected component of $\heli_{\height,s}\setminus c_{\height,s}$ containing $\zeta_0$. 
Given $0<\lambda\leq 1$, consider the map $f_\lambda\colon(x_1,x_2,x_3)\mapsto(x_1,x_2,\lambda x_3)$ and the union 
\begin{align}\label{eqn:Yls}
Y_{\lambda,s}\vcentcolon=f_\lambda(\Lambda_s)\cup I_{\lambda,s} \cup \rotation_{\zeta_1}^{\pi}f_\lambda(\Lambda_s), 
\end{align}
where $I_{\lambda,s}$ is the planar, rectangular strip of vertical length $(1-\lambda)\height\pi$ and horizontal width $\delta(s)=2\sqrt{1-s^2}$ connecting $f_\lambda(\Lambda_s)$ and $\rotation_{\zeta_1}^{\pi}f_\lambda(\Lambda_s)$ as visualised in Figure \ref{fig:vertical_strip}. 
Since vertical scaling $f_\lambda$ by $0<\lambda\leq 1$ does not increase the area of $\Lambda_s$, we have 
\begin{align}\label{eqn:20240306}
\hsd^2(Y_{\lambda,s})&=2\hsd^2\bigl(f_\lambda(\Lambda_s)\bigr)+\hsd^2(I_{\lambda,s})
<\hsd^2(\heli_{\height,s})+2\height\pi\sqrt{1-s^2} 
\end{align}
for all $0<\lambda\leq1$. 
We may regularise $Y_{\lambda,s}$ to obtain a smooth, $\rotation_{\zeta_1}^{\pi}$-equivariant, properly embedded surface in $Z_\height$ which depends smoothly on $\lambda,s$ and still satisfies inequality \eqref{eqn:20240306}. 
(We do not relabel the regularised surface.)  
By Lemma~\ref{lem:translation} applied to the right hand side in \eqref{eqn:20240306}, there exists $s_0\in\interval{0,1}$ such that $\hsd^2(Y_{\lambda,s})<\hsd^2(\heli_{\height,0})$ for all $s\in\intervaL{-1,-s_0}\cup\Interval{s_0,1}$ and all $0<\lambda\leq1$. 
Let now $\lambda\colon\Interval{s_0,1}\to\intervaL{0,1}$ be a smooth, decreasing function of $s$ such that $\lambda(s_0)=1$ and $\lambda(s)\to0$ as $s\to1$. 
For any $0<t<1$ we define 
\begin{align}
\Gamma_t\vcentcolon=\begin{cases}
\heli_{\height,2t-1} & \text{ if }\abs{2t-1}<s_0, \\
Y_{\lambda(\abs{2t-1}),\,2t-1} & \text{ if }\abs{2t-1}\geq s_0.
\end{cases}
\end{align}
Lemma~\ref{lem:translation} and the choice of $s_0$ imply $\hsd^2(\Gamma_t)\leq\hsd^2(\heli_{\height,0})$ for all $0<t<1$. 
From its explicit parametrisation as a ruled surface or as a graph almost everywhere over $\B^2$ as in \eqref{eqn:graph}, it is evident that 
$\Gamma_t$ converges (in the sense of varifolds) to the union of two horizontal half-discs as $t\to0$ respectively $t\to1$ (see Figure \ref{fig:vertical_strip}, right image). 
Defining $\Gamma_0$ and $\Gamma_1$ as the respective limit, we obtain \ref{lem:sweepout-area}. 
Moreover, $\Gamma_t$ inherits properties \ref{lem:sweepout-sym} and \ref{lem:sweepout-topo} from $\heli_{\height,s}$. 
Finally, we define $F_t$ as the connected component of $Z_\height\setminus\Gamma_t$ containing the end point $(-1,0,\height\pi/2)$ of $\zeta_1$ to prove~\ref{lem:sweepout-perim}. 
\end{proof}

\begin{figure}
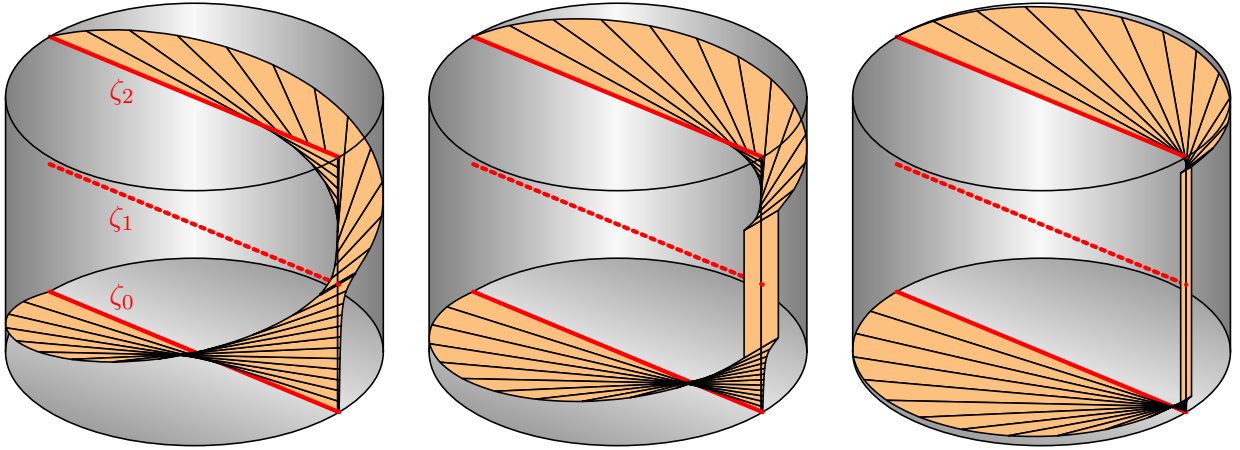
%
\pgfmathsetmacro{\phiO}{40} 
\pgfmathsetmacro{\thetaO}{60}
\tdplotsetmaincoords{\thetaO}{\phiO}
\cylinder{% Surface L
\path[facet]( -1.0000, -0.0000,  0.0000)..controls( -1.0000, -0.0802,  0.0201)and( -0.9869, -0.1801,  0.0453)..( -0.9663, -0.2575,  0.0654)--(  0.9900,  0.0000,  0.0654)..controls(  0.9900,  0.0000,  0.0453)and(  0.9900,  0.0000,  0.0201)..(  0.9900,  0.0000,  0.0000)--cycle;%[1]
\path[facet](  0.9900,  0.0000,  0.0000)..controls(  0.9900,  0.0000,  0.0201)and(  0.9900,  0.0000,  0.0453)..(  0.9900,  0.0000,  0.0654)--(  1.0000,  0.0013,  0.0654)..controls(  1.0000,  0.0009,  0.0453)and(  1.0000,  0.0004,  0.0201)..(  1.0000,  0.0000,  0.0000)--cycle;%[2]
\draw[axis](-1,0,0)--++(2,0,0)node[pos=0.25,above]{$\zeta_0$};
\draw[axis,dotted](-1,0,0.7854)--(  0.9900,0,0.7854)node[pos=0.25,below]{$\zeta_1$};
\path[facet]( -0.9663, -0.2575,  0.0654)..controls( -0.9456, -0.3350,  0.0856)and( -0.9072, -0.4282,  0.1108)..( -0.8674, -0.4977,  0.1309)--(  0.9900,  0.0000,  0.1309)..controls(  0.9900,  0.0000,  0.1108)and(  0.9900,  0.0000,  0.0856)..(  0.9900,  0.0000,  0.0654)--cycle;%[3]
\path[facet](  0.9900,  0.0000,  0.0654)..controls(  0.9900,  0.0000,  0.0856)and(  0.9900,  0.0000,  0.1108)..(  0.9900,  0.0000,  0.1309)--(  1.0000,  0.0027,  0.1309)..controls(  1.0000,  0.0022,  0.1108)and(  1.0000,  0.0017,  0.0856)..(  1.0000,  0.0013,  0.0654)--cycle;%[4]
\path[facet]( -0.8674, -0.4977,  0.1309)..controls( -0.8275, -0.5672,  0.1510)and( -0.7664, -0.6473,  0.1762)..( -0.7100, -0.7042,  0.1963)--(  0.9900,  0.0000,  0.1963)..controls(  0.9900,  0.0000,  0.1762)and(  0.9900,  0.0000,  0.1510)..(  0.9900,  0.0000,  0.1309)--cycle;%[5]
\path[facet](  0.9900,  0.0000,  0.1309)..controls(  0.9900,  0.0000,  0.1510)and(  0.9900,  0.0000,  0.1762)..(  0.9900,  0.0000,  0.1963)--(  1.0000,  0.0041,  0.1963)..controls(  1.0000,  0.0037,  0.1762)and(  1.0000,  0.0031,  0.1510)..(  1.0000,  0.0027,  0.1309)--cycle;%[6]
\path[facet]( -0.7100, -0.7042,  0.1963)..controls( -0.6536, -0.7610,  0.2165)and( -0.5741, -0.8227,  0.2417)..( -0.5050, -0.8631,  0.2618)--(  0.9900,  0.0000,  0.2618)..controls(  0.9900,  0.0000,  0.2417)and(  0.9900,  0.0000,  0.2165)..(  0.9900,  0.0000,  0.1963)--cycle;%[7]
\path[facet](  0.9900,  0.0000,  0.1963)..controls(  0.9900,  0.0000,  0.2165)and(  0.9900,  0.0000,  0.2417)..(  0.9900,  0.0000,  0.2618)--(  1.0000,  0.0058,  0.2618)..controls(  1.0000,  0.0052,  0.2417)and(  1.0000,  0.0046,  0.2165)..(  1.0000,  0.0041,  0.1963)--cycle;%[8]
\path[facet]( -0.5050, -0.8631,  0.2618)..controls( -0.4359, -0.9035,  0.2819)and( -0.3432, -0.9426,  0.3071)..( -0.2662, -0.9639,  0.3272)--(  0.9900,  0.0000,  0.3272)..controls(  0.9900,  0.0000,  0.3071)and(  0.9900,  0.0000,  0.2819)..(  0.9900,  0.0000,  0.2618)--cycle;%[9]
\path[facet](  0.9900,  0.0000,  0.2618)..controls(  0.9900,  0.0000,  0.2819)and(  0.9900,  0.0000,  0.3071)..(  0.9900,  0.0000,  0.3272)--(  1.0000,  0.0077,  0.3272)..controls(  1.0000,  0.0070,  0.3071)and(  1.0000,  0.0063,  0.2819)..(  1.0000,  0.0058,  0.2618)--cycle;%[10]
\path[facet]( -0.2662, -0.9639,  0.3272)..controls( -0.1892, -0.9852,  0.3474)and( -0.0897, -0.9992,  0.3726)..( -0.0100, -1.0000,  0.3927)--(  0.9900,  0.0000,  0.3927)..controls(  0.9900,  0.0000,  0.3726)and(  0.9900,  0.0000,  0.3474)..(  0.9900,  0.0000,  0.3272)--cycle;%[11]
\path[facet](  0.9900,  0.0000,  0.3272)..controls(  0.9900,  0.0000,  0.3474)and(  0.9900,  0.0000,  0.3726)..(  0.9900,  0.0000,  0.3927)--(  1.0000,  0.0100,  0.3927)..controls(  1.0000,  0.0092,  0.3726)and(  1.0000,  0.0083,  0.3474)..(  1.0000,  0.0077,  0.3272)--cycle;%[12]
\path[facet]( -0.0100, -1.0000,  0.3927)..controls(  0.0698, -1.0007,  0.4128)and(  0.1693, -0.9888,  0.4380)..(  0.2463, -0.9692,  0.4581)--(  0.9900,  0.0000,  0.4581)..controls(  0.9900,  0.0000,  0.4380)and(  0.9900,  0.0000,  0.4128)..(  0.9900,  0.0000,  0.3927)--cycle;%[13]
\path[facet](  0.9900,  0.0000,  0.3927)..controls(  0.9900,  0.0000,  0.4128)and(  0.9900,  0.0000,  0.4380)..(  0.9900,  0.0000,  0.4581)--(  0.9999,  0.0129,  0.4581)..controls(  0.9999,  0.0119,  0.4380)and(  0.9999,  0.0107,  0.4128)..(  1.0000,  0.0100,  0.3927)--cycle;%[14]
\path[facet](  0.2463, -0.9692,  0.4581)..controls(  0.3234, -0.9496,  0.4783)and(  0.4161, -0.9128,  0.5035)..(  0.4851, -0.8744,  0.5236)--(  0.9900,  0.0000,  0.5236)..controls(  0.9900,  0.0000,  0.5035)and(  0.9900,  0.0000,  0.4783)..(  0.9900,  0.0000,  0.4581)--cycle;%[15]
\path[facet](  0.9900,  0.0000,  0.4581)..controls(  0.9900,  0.0000,  0.4783)and(  0.9900,  0.0000,  0.5035)..(  0.9900,  0.0000,  0.5236)--(  0.9999,  0.0171,  0.5236)..controls(  0.9999,  0.0155,  0.5035)and(  0.9999,  0.0140,  0.4783)..(  0.9999,  0.0129,  0.4581)--cycle;%[16]
\path[facet](  0.4851, -0.8744,  0.5236)..controls(  0.5542, -0.8361,  0.5437)and(  0.6339, -0.7774,  0.5689)..(  0.6903, -0.7235,  0.5890)--(  0.9900,  0.0000,  0.5890)..controls(  0.9900,  0.0000,  0.5689)and(  0.9900,  0.0000,  0.5437)..(  0.9900,  0.0000,  0.5236)--cycle;%[17]
\path[facet](  0.9900,  0.0000,  0.5236)..controls(  0.9900,  0.0000,  0.5437)and(  0.9900,  0.0000,  0.5689)..(  0.9900,  0.0000,  0.5890)--(  0.9997,  0.0235,  0.5890)..controls(  0.9998,  0.0209,  0.5689)and(  0.9998,  0.0186,  0.5437)..(  0.9999,  0.0171,  0.5236)--cycle;%[18]
\path[facet](  0.6903, -0.7235,  0.5890)..controls(  0.7468, -0.6697,  0.6092)and(  0.8079, -0.5941,  0.6344)..(  0.8480, -0.5300,  0.6545)--(  0.9900,  0.0000,  0.6545)..controls(  0.9900,  0.0000,  0.6344)and(  0.9900,  0.0000,  0.6092)..(  0.9900,  0.0000,  0.5890)--cycle;%[19]
\path[facet](  0.9900,  0.0000,  0.5890)..controls(  0.9900,  0.0000,  0.6092)and(  0.9900,  0.0000,  0.6344)..(  0.9900,  0.0000,  0.6545)--(  0.9994,  0.0350,  0.6545)..controls(  0.9996,  0.0300,  0.6344)and(  0.9997,  0.0260,  0.6092)..(  0.9997,  0.0235,  0.5890)--cycle;%[20]
\path[facet](  0.8480, -0.5300,  0.6545)..controls(  0.8880, -0.4660,  0.6746)and(  0.9267, -0.3818,  0.6998)..(  0.9482, -0.3178,  0.7199)--(  0.9900,  0.0000,  0.7199)..controls(  0.9900,  0.0000,  0.6998)and(  0.9900,  0.0000,  0.6746)..(  0.9900,  0.0000,  0.6545)--cycle;%[21]
\path[facet](  0.9900,  0.0000,  0.6545)..controls(  0.9900,  0.0000,  0.6746)and(  0.9900,  0.0000,  0.6998)..(  0.9900,  0.0000,  0.7199)--(  0.9981,  0.0616,  0.7199)..controls(  0.9989,  0.0485,  0.6998)and(  0.9992,  0.0400,  0.6746)..(  0.9994,  0.0350,  0.6545)--cycle;%[22]
\path[facet](  0.9482, -0.3178,  0.7199)..controls(  0.9696, -0.2538,  0.7401)and(  0.9900, -0.1411,  0.7854)..(  0.9900, -0.1411,  0.7854)--(  0.9900,  0.0000,  0.7854)..controls(  0.9900,  0.0000,  0.7854)and(  0.9900,  0.0000,  0.7401)..(  0.9900,  0.0000,  0.7199)--cycle;%[23]
\path[facet](  0.9900,  0.0000,  0.7199)..controls(  0.9900,  0.0000,  0.7401)and(  0.9900,  0.0000,  0.7854)..(  0.9900,  0.0000,  0.7854)--(  0.9900,  0.1411,  0.7854)..controls(  0.9900,  0.1411,  0.7854)and(  0.9973,  0.0746,  0.7401)..(  0.9981,  0.0616,  0.7199)--cycle;%[24]
\path[facet](  0.9900, -0.1411,  0.7854)..controls(  0.9900, -0.1411,  0.7854)and(  0.9973, -0.0746,  0.8307)..(  0.9981, -0.0616,  0.8508)--(  0.9900,  0.0000,  0.8508)..controls(  0.9900,  0.0000,  0.8307)and(  0.9900,  0.0000,  0.7854)..(  0.9900,  0.0000,  0.7854)--cycle;%[27]
\path[facet](  0.9900,  0.0000,  0.7854)..controls(  0.9900,  0.0000,  0.7854)and(  0.9900,  0.0000,  0.8307)..(  0.9900,  0.0000,  0.8508)--(  0.9482,  0.3178,  0.8508)..controls(  0.9696,  0.2538,  0.8307)and(  0.9900,  0.1411,  0.7854)..(  0.9900,  0.1411,  0.7854)--cycle;%[28]
\path[facet](  0.9981, -0.0616,  0.8508)..controls(  0.9989, -0.0485,  0.8710)and(  0.9992, -0.0400,  0.8962)..(  0.9994, -0.0350,  0.9163)--(  0.9900,  0.0000,  0.9163)..controls(  0.9900,  0.0000,  0.8962)and(  0.9900,  0.0000,  0.8710)..(  0.9900,  0.0000,  0.8508)--cycle;%[29]
\path[facet](  0.9900,  0.0000,  0.8508)..controls(  0.9900,  0.0000,  0.8710)and(  0.9900,  0.0000,  0.8962)..(  0.9900,  0.0000,  0.9163)--(  0.8480,  0.5300,  0.9163)..controls(  0.8880,  0.4660,  0.8962)and(  0.9267,  0.3818,  0.8710)..(  0.9482,  0.3178,  0.8508)--cycle;%[30]
\path[facet](  0.9994, -0.0350,  0.9163)..controls(  0.9996, -0.0300,  0.9364)and(  0.9997, -0.0260,  0.9616)..(  0.9997, -0.0235,  0.9817)--(  0.9900,  0.0000,  0.9817)..controls(  0.9900,  0.0000,  0.9616)and(  0.9900,  0.0000,  0.9364)..(  0.9900,  0.0000,  0.9163)--cycle;%[31]
\path[facet](  0.9900,  0.0000,  0.9163)..controls(  0.9900,  0.0000,  0.9364)and(  0.9900,  0.0000,  0.9616)..(  0.9900,  0.0000,  0.9817)--(  0.6903,  0.7235,  0.9817)..controls(  0.7468,  0.6697,  0.9616)and(  0.8079,  0.5941,  0.9364)..(  0.8480,  0.5300,  0.9163)--cycle;%[32]
\path[facet](  0.9997, -0.0235,  0.9817)..controls(  0.9998, -0.0209,  1.0019)and(  0.9998, -0.0186,  1.0271)..(  0.9999, -0.0171,  1.0472)--(  0.9900,  0.0000,  1.0472)..controls(  0.9900,  0.0000,  1.0271)and(  0.9900,  0.0000,  1.0019)..(  0.9900,  0.0000,  0.9817)--cycle;%[33]
\path[facet](  0.9900,  0.0000,  0.9817)..controls(  0.9900,  0.0000,  1.0019)and(  0.9900,  0.0000,  1.0271)..(  0.9900,  0.0000,  1.0472)--(  0.4851,  0.8744,  1.0472)..controls(  0.5542,  0.8361,  1.0271)and(  0.6339,  0.7774,  1.0019)..(  0.6903,  0.7235,  0.9817)--cycle;%[34]
\path[facet](  0.9999, -0.0171,  1.0472)..controls(  0.9999, -0.0155,  1.0673)and(  0.9999, -0.0140,  1.0925)..(  0.9999, -0.0129,  1.1126)--(  0.9900,  0.0000,  1.1126)..controls(  0.9900,  0.0000,  1.0925)and(  0.9900,  0.0000,  1.0673)..(  0.9900,  0.0000,  1.0472)--cycle;%[35]
\path[facet](  0.9900,  0.0000,  1.0472)..controls(  0.9900,  0.0000,  1.0673)and(  0.9900,  0.0000,  1.0925)..(  0.9900,  0.0000,  1.1126)--(  0.2463,  0.9692,  1.1126)..controls(  0.3234,  0.9496,  1.0925)and(  0.4161,  0.9128,  1.0673)..(  0.4851,  0.8744,  1.0472)--cycle;%[36]
\path[facet](  0.9999, -0.0129,  1.1126)..controls(  0.9999, -0.0119,  1.1328)and(  0.9999, -0.0107,  1.1580)..(  1.0000, -0.0100,  1.1781)--(  0.9900,  0.0000,  1.1781)..controls(  0.9900,  0.0000,  1.1580)and(  0.9900,  0.0000,  1.1328)..(  0.9900,  0.0000,  1.1126)--cycle;%[37]
\path[facet](  0.9900,  0.0000,  1.1126)..controls(  0.9900,  0.0000,  1.1328)and(  0.9900,  0.0000,  1.1580)..(  0.9900,  0.0000,  1.1781)--( -0.0100,  1.0000,  1.1781)..controls(  0.0698,  1.0007,  1.1580)and(  0.1693,  0.9888,  1.1328)..(  0.2463,  0.9692,  1.1126)--cycle;%[38]
\path[facet](  1.0000, -0.0100,  1.1781)..controls(  1.0000, -0.0092,  1.1982)and(  1.0000, -0.0083,  1.2234)..(  1.0000, -0.0077,  1.2435)--(  0.9900,  0.0000,  1.2435)..controls(  0.9900,  0.0000,  1.2234)and(  0.9900,  0.0000,  1.1982)..(  0.9900,  0.0000,  1.1781)--cycle;%[39]
\path[facet](  0.9900,  0.0000,  1.1781)..controls(  0.9900,  0.0000,  1.1982)and(  0.9900,  0.0000,  1.2234)..(  0.9900,  0.0000,  1.2435)--( -0.2662,  0.9639,  1.2435)..controls( -0.1892,  0.9852,  1.2234)and( -0.0897,  0.9992,  1.1982)..( -0.0100,  1.0000,  1.1781)--cycle;%[40]
\path[facet](  1.0000, -0.0077,  1.2435)..controls(  1.0000, -0.0070,  1.2637)and(  1.0000, -0.0063,  1.2889)..(  1.0000, -0.0058,  1.3090)--(  0.9900,  0.0000,  1.3090)..controls(  0.9900,  0.0000,  1.2889)and(  0.9900,  0.0000,  1.2637)..(  0.9900,  0.0000,  1.2435)--cycle;%[41]
\path[facet](  0.9900,  0.0000,  1.2435)..controls(  0.9900,  0.0000,  1.2637)and(  0.9900,  0.0000,  1.2889)..(  0.9900,  0.0000,  1.3090)--( -0.5050,  0.8631,  1.3090)..controls( -0.4359,  0.9035,  1.2889)and( -0.3432,  0.9426,  1.2637)..( -0.2662,  0.9639,  1.2435)--cycle;%[42]
\path[facet](  1.0000, -0.0058,  1.3090)..controls(  1.0000, -0.0052,  1.3291)and(  1.0000, -0.0046,  1.3543)..(  1.0000, -0.0041,  1.3744)--(  0.9900,  0.0000,  1.3744)..controls(  0.9900,  0.0000,  1.3543)and(  0.9900,  0.0000,  1.3291)..(  0.9900,  0.0000,  1.3090)--cycle;%[43]
\path[facet](  0.9900,  0.0000,  1.3090)..controls(  0.9900,  0.0000,  1.3291)and(  0.9900,  0.0000,  1.3543)..(  0.9900,  0.0000,  1.3744)--( -0.7100,  0.7042,  1.3744)..controls( -0.6536,  0.7610,  1.3543)and( -0.5741,  0.8227,  1.3291)..( -0.5050,  0.8631,  1.3090)--cycle;%[44]
\path[facet](  1.0000, -0.0041,  1.3744)..controls(  1.0000, -0.0037,  1.3946)and(  1.0000, -0.0031,  1.4198)..(  1.0000, -0.0027,  1.4399)--(  0.9900,  0.0000,  1.4399)..controls(  0.9900,  0.0000,  1.4198)and(  0.9900,  0.0000,  1.3946)..(  0.9900,  0.0000,  1.3744)--cycle;%[45]
\path[facet](  0.9900,  0.0000,  1.3744)..controls(  0.9900,  0.0000,  1.3946)and(  0.9900,  0.0000,  1.4198)..(  0.9900,  0.0000,  1.4399)--( -0.8674,  0.4977,  1.4399)..controls( -0.8275,  0.5672,  1.4198)and( -0.7664,  0.6473,  1.3946)..( -0.7100,  0.7042,  1.3744)--cycle;%[46]
\path[facet](  1.0000, -0.0027,  1.4399)..controls(  1.0000, -0.0022,  1.4600)and(  1.0000, -0.0017,  1.4852)..(  1.0000, -0.0013,  1.5053)--(  0.9900,  0.0000,  1.5053)..controls(  0.9900,  0.0000,  1.4852)and(  0.9900,  0.0000,  1.4600)..(  0.9900,  0.0000,  1.4399)--cycle;%[47]
\path[facet](  0.9900,  0.0000,  1.4399)..controls(  0.9900,  0.0000,  1.4600)and(  0.9900,  0.0000,  1.4852)..(  0.9900,  0.0000,  1.5053)--( -0.9663,  0.2575,  1.5053)..controls( -0.9456,  0.3350,  1.4852)and( -0.9072,  0.4282,  1.4600)..( -0.8674,  0.4977,  1.4399)--cycle;%[48]
\path[facet](  1.0000, -0.0013,  1.5053)..controls(  1.0000, -0.0009,  1.5255)and(  1.0000, -0.0004,  1.5507)..(  1.0000, -0.0000,  1.5708)--(  0.9900,  0.0000,  1.5708)..controls(  0.9900,  0.0000,  1.5507)and(  0.9900,  0.0000,  1.5255)..(  0.9900,  0.0000,  1.5053)--cycle;%[49]
\path[facet](  0.9900,  0.0000,  1.5053)..controls(  0.9900,  0.0000,  1.5255)and(  0.9900,  0.0000,  1.5507)..(  0.9900,  0.0000,  1.5708)--( -1.0000,  0.0000,  1.5708)..controls( -1.0000,  0.0802,  1.5507)and( -0.9869,  0.1801,  1.5255)..( -0.9663,  0.2575,  1.5053)--cycle;%[50]
\draw[very thick](  0.9900, -0.1411,  0.7854)--(  0.9900,  0.1411,  0.7854);
\draw[axis](  0.9900,0,0.7854)--(1,0,0.7854);
\draw[axis](-1,0,\hoehe)--++(2,0,0)node[pos=0.25,below]{$\zeta_2$};
}\hfill 
\cylinder{%
\draw[axis,dotted](-1,0,0.7854)--(1,0,0.7854); 
\path[facet]( -1.0000, -0.0000,  0.0000)..controls( -1.0000, -0.0802,  0.0101)and( -0.9869, -0.1801,  0.0227)..( -0.9663, -0.2575,  0.0327)--(  0.9900,  0.0000,  0.0327)..controls(  0.9900,  0.0000,  0.0227)and(  0.9900,  0.0000,  0.0101)..(  0.9900,  0.0000,  0.0000)--cycle;%[1]
\path[facet](  0.9900,  0.0000,  0.0000)..controls(  0.9900,  0.0000,  0.0101)and(  0.9900,  0.0000,  0.0227)..(  0.9900,  0.0000,  0.0327)--(  1.0000,  0.0013,  0.0327)..controls(  1.0000,  0.0009,  0.0227)and(  1.0000,  0.0004,  0.0101)..(  1.0000,  0.0000,  0.0000)--cycle;%[2]
\draw[axis](-1,0,0)--++(2,0,0);
\path[facet]( -0.9663, -0.2575,  0.0327)..controls( -0.9456, -0.3350,  0.0428)and( -0.9072, -0.4282,  0.0554)..( -0.8674, -0.4977,  0.0654)--(  0.9900,  0.0000,  0.0654)..controls(  0.9900,  0.0000,  0.0554)and(  0.9900,  0.0000,  0.0428)..(  0.9900,  0.0000,  0.0327)--cycle;%[3]
\path[facet](  0.9900,  0.0000,  0.0327)..controls(  0.9900,  0.0000,  0.0428)and(  0.9900,  0.0000,  0.0554)..(  0.9900,  0.0000,  0.0654)--(  1.0000,  0.0027,  0.0654)..controls(  1.0000,  0.0022,  0.0554)and(  1.0000,  0.0017,  0.0428)..(  1.0000,  0.0013,  0.0327)--cycle;%[4]
\path[facet]( -0.8674, -0.4977,  0.0654)..controls( -0.8275, -0.5672,  0.0755)and( -0.7664, -0.6473,  0.0881)..( -0.7100, -0.7042,  0.0982)--(  0.9900,  0.0000,  0.0982)..controls(  0.9900,  0.0000,  0.0881)and(  0.9900,  0.0000,  0.0755)..(  0.9900,  0.0000,  0.0654)--cycle;%[5]
\path[facet](  0.9900,  0.0000,  0.0654)..controls(  0.9900,  0.0000,  0.0755)and(  0.9900,  0.0000,  0.0881)..(  0.9900,  0.0000,  0.0982)--(  1.0000,  0.0041,  0.0982)..controls(  1.0000,  0.0037,  0.0881)and(  1.0000,  0.0031,  0.0755)..(  1.0000,  0.0027,  0.0654)--cycle;%[6]
\path[facet]( -0.7100, -0.7042,  0.0982)..controls( -0.6536, -0.7610,  0.1082)and( -0.5741, -0.8227,  0.1208)..( -0.5050, -0.8631,  0.1309)--(  0.9900,  0.0000,  0.1309)..controls(  0.9900,  0.0000,  0.1208)and(  0.9900,  0.0000,  0.1082)..(  0.9900,  0.0000,  0.0982)--cycle;%[7]
\path[facet](  0.9900,  0.0000,  0.0982)..controls(  0.9900,  0.0000,  0.1082)and(  0.9900,  0.0000,  0.1208)..(  0.9900,  0.0000,  0.1309)--(  1.0000,  0.0058,  0.1309)..controls(  1.0000,  0.0052,  0.1208)and(  1.0000,  0.0046,  0.1082)..(  1.0000,  0.0041,  0.0982)--cycle;%[8]
\path[facet]( -0.5050, -0.8631,  0.1309)..controls( -0.4359, -0.9035,  0.1410)and( -0.3432, -0.9426,  0.1536)..( -0.2662, -0.9639,  0.1636)--(  0.9900,  0.0000,  0.1636)..controls(  0.9900,  0.0000,  0.1536)and(  0.9900,  0.0000,  0.1410)..(  0.9900,  0.0000,  0.1309)--cycle;%[9]
\path[facet](  0.9900,  0.0000,  0.1309)..controls(  0.9900,  0.0000,  0.1410)and(  0.9900,  0.0000,  0.1536)..(  0.9900,  0.0000,  0.1636)--(  1.0000,  0.0077,  0.1636)..controls(  1.0000,  0.0070,  0.1536)and(  1.0000,  0.0063,  0.1410)..(  1.0000,  0.0058,  0.1309)--cycle;%[10]
\path[facet]( -0.2662, -0.9639,  0.1636)..controls( -0.1892, -0.9852,  0.1737)and( -0.0897, -0.9992,  0.1863)..( -0.0100, -1.0000,  0.1963)--(  0.9900,  0.0000,  0.1963)..controls(  0.9900,  0.0000,  0.1863)and(  0.9900,  0.0000,  0.1737)..(  0.9900,  0.0000,  0.1636)--cycle;%[11]
\path[facet](  0.9900,  0.0000,  0.1636)..controls(  0.9900,  0.0000,  0.1737)and(  0.9900,  0.0000,  0.1863)..(  0.9900,  0.0000,  0.1963)--(  1.0000,  0.0100,  0.1963)..controls(  1.0000,  0.0092,  0.1863)and(  1.0000,  0.0083,  0.1737)..(  1.0000,  0.0077,  0.1636)--cycle;%[12]
\path[facet]( -0.0100, -1.0000,  0.1963)..controls(  0.0698, -1.0007,  0.2064)and(  0.1693, -0.9888,  0.2190)..(  0.2463, -0.9692,  0.2291)--(  0.9900,  0.0000,  0.2291)..controls(  0.9900,  0.0000,  0.2190)and(  0.9900,  0.0000,  0.2064)..(  0.9900,  0.0000,  0.1963)--cycle;%[13]
\path[facet](  0.9900,  0.0000,  0.1963)..controls(  0.9900,  0.0000,  0.2064)and(  0.9900,  0.0000,  0.2190)..(  0.9900,  0.0000,  0.2291)--(  0.9999,  0.0129,  0.2291)..controls(  0.9999,  0.0119,  0.2190)and(  0.9999,  0.0107,  0.2064)..(  1.0000,  0.0100,  0.1963)--cycle;%[14]
\path[facet](  0.2463, -0.9692,  0.2291)..controls(  0.3234, -0.9496,  0.2391)and(  0.4161, -0.9128,  0.2517)..(  0.4851, -0.8744,  0.2618)--(  0.9900,  0.0000,  0.2618)..controls(  0.9900,  0.0000,  0.2517)and(  0.9900,  0.0000,  0.2391)..(  0.9900,  0.0000,  0.2291)--cycle;%[15]
\path[facet](  0.9900,  0.0000,  0.2291)..controls(  0.9900,  0.0000,  0.2391)and(  0.9900,  0.0000,  0.2517)..(  0.9900,  0.0000,  0.2618)--(  0.9999,  0.0171,  0.2618)..controls(  0.9999,  0.0155,  0.2517)and(  0.9999,  0.0140,  0.2391)..(  0.9999,  0.0129,  0.2291)--cycle;%[16]
\path[facet](  0.4851, -0.8744,  0.2618)..controls(  0.5542, -0.8361,  0.2719)and(  0.6339, -0.7774,  0.2845)..(  0.6903, -0.7235,  0.2945)--(  0.9900,  0.0000,  0.2945)..controls(  0.9900,  0.0000,  0.2845)and(  0.9900,  0.0000,  0.2719)..(  0.9900,  0.0000,  0.2618)--cycle;%[17]
\path[facet](  0.9900,  0.0000,  0.2618)..controls(  0.9900,  0.0000,  0.2719)and(  0.9900,  0.0000,  0.2845)..(  0.9900,  0.0000,  0.2945)--(  0.9997,  0.0235,  0.2945)..controls(  0.9998,  0.0209,  0.2845)and(  0.9998,  0.0186,  0.2719)..(  0.9999,  0.0171,  0.2618)--cycle;%[18]
\path[facet](  0.6903, -0.7235,  0.2945)..controls(  0.7468, -0.6697,  0.3046)and(  0.8079, -0.5941,  0.3172)..(  0.8480, -0.5300,  0.3272)--(  0.9900,  0.0000,  0.3272)..controls(  0.9900,  0.0000,  0.3172)and(  0.9900,  0.0000,  0.3046)..(  0.9900,  0.0000,  0.2945)--cycle;%[19]
\path[facet](  0.9900,  0.0000,  0.2945)..controls(  0.9900,  0.0000,  0.3046)and(  0.9900,  0.0000,  0.3172)..(  0.9900,  0.0000,  0.3272)--(  0.9994,  0.0350,  0.3272)..controls(  0.9996,  0.0300,  0.3172)and(  0.9997,  0.0260,  0.3046)..(  0.9997,  0.0235,  0.2945)--cycle;%[20]
\path[facet](  0.8480, -0.5300,  0.3272)..controls(  0.8880, -0.4660,  0.3373)and(  0.9267, -0.3818,  0.3499)..(  0.9482, -0.3178,  0.3600)--(  0.9900,  0.0000,  0.3600)..controls(  0.9900,  0.0000,  0.3499)and(  0.9900,  0.0000,  0.3373)..(  0.9900,  0.0000,  0.3272)--cycle;%[21]
\path[facet](  0.9900,  0.0000,  0.3272)..controls(  0.9900,  0.0000,  0.3373)and(  0.9900,  0.0000,  0.3499)..(  0.9900,  0.0000,  0.3600)--(  0.9981,  0.0616,  0.3600)..controls(  0.9989,  0.0485,  0.3499)and(  0.9992,  0.0400,  0.3373)..(  0.9994,  0.0350,  0.3272)--cycle;%[22]
\path[facet](  0.9482, -0.3178,  0.3600)..controls(  0.9696, -0.2538,  0.3700)and(  0.9900, -0.1411,  0.3927)..(  0.9900, -0.1411,  0.3927)--(  0.9900,  0.0000,  0.3927)..controls(  0.9900,  0.0000,  0.3927)and(  0.9900,  0.0000,  0.3700)..(  0.9900,  0.0000,  0.3600)--cycle;%[23]
\path[facet](  0.9900,  0.0000,  0.3600)..controls(  0.9900,  0.0000,  0.3700)and(  0.9900,  0.0000,  0.3927)..(  0.9900,  0.0000,  0.3927)--(  0.9900,  0.1411,  0.3927)..controls(  0.9900,  0.1411,  0.3927)and(  0.9973,  0.0746,  0.3700)..(  0.9981,  0.0616,  0.3600)--cycle;%[24]
\path[facet](  0.9900, -0.1411,  0.3927)..controls(  0.9900, -0.1411,  0.3927)and(  0.9900, -0.1411,  1.1781)..(  0.9900, -0.1411,  1.1781)--(  0.9900,  0.0000,  1.1781)..controls(  0.9900,  0.0000,  1.1781)and(  0.9900,  0.0000,  0.3927)..(  0.9900,  0.0000,  0.3927)--cycle;%[25]
\path[facet](  0.9900,  0.0000,  0.3927)..controls(  0.9900,  0.0000,  0.3927)and(  0.9900,  0.0000,  1.1781)..(  0.9900,  0.0000,  1.1781)--(  0.9900,  0.1411,  1.1781)..controls(  0.9900,  0.1411,  1.1781)and(  0.9900,  0.1411,  0.3927)..(  0.9900,  0.1411,  0.3927)--cycle;%[26]
\path[facet](  0.9900, -0.1411,  1.1781)..controls(  0.9900, -0.1411,  1.1781)and(  0.9973, -0.0746,  1.2008)..(  0.9981, -0.0616,  1.2108)--(  0.9900,  0.0000,  1.2108)..controls(  0.9900,  0.0000,  1.2008)and(  0.9900,  0.0000,  1.1781)..(  0.9900,  0.0000,  1.1781)--cycle;%[27]
\path[facet](  0.9900,  0.0000,  1.1781)..controls(  0.9900,  0.0000,  1.1781)and(  0.9900,  0.0000,  1.2008)..(  0.9900,  0.0000,  1.2108)--(  0.9482,  0.3178,  1.2108)..controls(  0.9696,  0.2538,  1.2008)and(  0.9900,  0.1411,  1.1781)..(  0.9900,  0.1411,  1.1781)--cycle;%[28]
\path[facet](  0.9981, -0.0616,  1.2108)..controls(  0.9989, -0.0485,  1.2209)and(  0.9992, -0.0400,  1.2335)..(  0.9994, -0.0350,  1.2435)--(  0.9900,  0.0000,  1.2435)..controls(  0.9900,  0.0000,  1.2335)and(  0.9900,  0.0000,  1.2209)..(  0.9900,  0.0000,  1.2108)--cycle;%[29]
\path[facet](  0.9900,  0.0000,  1.2108)..controls(  0.9900,  0.0000,  1.2209)and(  0.9900,  0.0000,  1.2335)..(  0.9900,  0.0000,  1.2435)--(  0.8480,  0.5300,  1.2435)..controls(  0.8880,  0.4660,  1.2335)and(  0.9267,  0.3818,  1.2209)..(  0.9482,  0.3178,  1.2108)--cycle;%[30]
\path[facet](  0.9994, -0.0350,  1.2435)..controls(  0.9996, -0.0300,  1.2536)and(  0.9997, -0.0260,  1.2662)..(  0.9997, -0.0235,  1.2763)--(  0.9900,  0.0000,  1.2763)..controls(  0.9900,  0.0000,  1.2662)and(  0.9900,  0.0000,  1.2536)..(  0.9900,  0.0000,  1.2435)--cycle;%[31]
\path[facet](  0.9900,  0.0000,  1.2435)..controls(  0.9900,  0.0000,  1.2536)and(  0.9900,  0.0000,  1.2662)..(  0.9900,  0.0000,  1.2763)--(  0.6903,  0.7235,  1.2763)..controls(  0.7468,  0.6697,  1.2662)and(  0.8079,  0.5941,  1.2536)..(  0.8480,  0.5300,  1.2435)--cycle;%[32]
\path[facet](  0.9997, -0.0235,  1.2763)..controls(  0.9998, -0.0209,  1.2863)and(  0.9998, -0.0186,  1.2989)..(  0.9999, -0.0171,  1.3090)--(  0.9900,  0.0000,  1.3090)..controls(  0.9900,  0.0000,  1.2989)and(  0.9900,  0.0000,  1.2863)..(  0.9900,  0.0000,  1.2763)--cycle;%[33]
\path[facet](  0.9900,  0.0000,  1.2763)..controls(  0.9900,  0.0000,  1.2863)and(  0.9900,  0.0000,  1.2989)..(  0.9900,  0.0000,  1.3090)--(  0.4851,  0.8744,  1.3090)..controls(  0.5542,  0.8361,  1.2989)and(  0.6339,  0.7774,  1.2863)..(  0.6903,  0.7235,  1.2763)--cycle;%[34]
\path[facet](  0.9999, -0.0171,  1.3090)..controls(  0.9999, -0.0155,  1.3191)and(  0.9999, -0.0140,  1.3317)..(  0.9999, -0.0129,  1.3417)--(  0.9900,  0.0000,  1.3417)..controls(  0.9900,  0.0000,  1.3317)and(  0.9900,  0.0000,  1.3191)..(  0.9900,  0.0000,  1.3090)--cycle;%[35]
\path[facet](  0.9900,  0.0000,  1.3090)..controls(  0.9900,  0.0000,  1.3191)and(  0.9900,  0.0000,  1.3317)..(  0.9900,  0.0000,  1.3417)--(  0.2463,  0.9692,  1.3417)..controls(  0.3234,  0.9496,  1.3317)and(  0.4161,  0.9128,  1.3191)..(  0.4851,  0.8744,  1.3090)--cycle;%[36]
\path[facet](  0.9999, -0.0129,  1.3417)..controls(  0.9999, -0.0119,  1.3518)and(  0.9999, -0.0107,  1.3644)..(  1.0000, -0.0100,  1.3744)--(  0.9900,  0.0000,  1.3744)..controls(  0.9900,  0.0000,  1.3644)and(  0.9900,  0.0000,  1.3518)..(  0.9900,  0.0000,  1.3417)--cycle;%[37]
\path[facet](  0.9900,  0.0000,  1.3417)..controls(  0.9900,  0.0000,  1.3518)and(  0.9900,  0.0000,  1.3644)..(  0.9900,  0.0000,  1.3744)--( -0.0100,  1.0000,  1.3744)..controls(  0.0698,  1.0007,  1.3644)and(  0.1693,  0.9888,  1.3518)..(  0.2463,  0.9692,  1.3417)--cycle;%[38]
\path[facet](  1.0000, -0.0100,  1.3744)..controls(  1.0000, -0.0092,  1.3845)and(  1.0000, -0.0083,  1.3971)..(  1.0000, -0.0077,  1.4072)--(  0.9900,  0.0000,  1.4072)..controls(  0.9900,  0.0000,  1.3971)and(  0.9900,  0.0000,  1.3845)..(  0.9900,  0.0000,  1.3744)--cycle;%[39]
\path[facet](  0.9900,  0.0000,  1.3744)..controls(  0.9900,  0.0000,  1.3845)and(  0.9900,  0.0000,  1.3971)..(  0.9900,  0.0000,  1.4072)--( -0.2662,  0.9639,  1.4072)..controls( -0.1892,  0.9852,  1.3971)and( -0.0897,  0.9992,  1.3845)..( -0.0100,  1.0000,  1.3744)--cycle;%[40]
\path[facet](  1.0000, -0.0077,  1.4072)..controls(  1.0000, -0.0070,  1.4172)and(  1.0000, -0.0063,  1.4298)..(  1.0000, -0.0058,  1.4399)--(  0.9900,  0.0000,  1.4399)..controls(  0.9900,  0.0000,  1.4298)and(  0.9900,  0.0000,  1.4172)..(  0.9900,  0.0000,  1.4072)--cycle;%[41]
\path[facet](  0.9900,  0.0000,  1.4072)..controls(  0.9900,  0.0000,  1.4172)and(  0.9900,  0.0000,  1.4298)..(  0.9900,  0.0000,  1.4399)--( -0.5050,  0.8631,  1.4399)..controls( -0.4359,  0.9035,  1.4298)and( -0.3432,  0.9426,  1.4172)..( -0.2662,  0.9639,  1.4072)--cycle;%[42]
\path[facet](  1.0000, -0.0058,  1.4399)..controls(  1.0000, -0.0052,  1.4500)and(  1.0000, -0.0046,  1.4626)..(  1.0000, -0.0041,  1.4726)--(  0.9900,  0.0000,  1.4726)..controls(  0.9900,  0.0000,  1.4626)and(  0.9900,  0.0000,  1.4500)..(  0.9900,  0.0000,  1.4399)--cycle;%[43]
\path[facet](  0.9900,  0.0000,  1.4399)..controls(  0.9900,  0.0000,  1.4500)and(  0.9900,  0.0000,  1.4626)..(  0.9900,  0.0000,  1.4726)--( -0.7100,  0.7042,  1.4726)..controls( -0.6536,  0.7610,  1.4626)and( -0.5741,  0.8227,  1.4500)..( -0.5050,  0.8631,  1.4399)--cycle;%[44]
\path[facet](  1.0000, -0.0041,  1.4726)..controls(  1.0000, -0.0037,  1.4827)and(  1.0000, -0.0031,  1.4953)..(  1.0000, -0.0027,  1.5053)--(  0.9900,  0.0000,  1.5053)..controls(  0.9900,  0.0000,  1.4953)and(  0.9900,  0.0000,  1.4827)..(  0.9900,  0.0000,  1.4726)--cycle;%[45]
\path[facet](  0.9900,  0.0000,  1.4726)..controls(  0.9900,  0.0000,  1.4827)and(  0.9900,  0.0000,  1.4953)..(  0.9900,  0.0000,  1.5053)--( -0.8674,  0.4977,  1.5053)..controls( -0.8275,  0.5672,  1.4953)and( -0.7664,  0.6473,  1.4827)..( -0.7100,  0.7042,  1.4726)--cycle;%[46]
\path[facet](  1.0000, -0.0027,  1.5053)..controls(  1.0000, -0.0022,  1.5154)and(  1.0000, -0.0017,  1.5280)..(  1.0000, -0.0013,  1.5381)--(  0.9900,  0.0000,  1.5381)..controls(  0.9900,  0.0000,  1.5280)and(  0.9900,  0.0000,  1.5154)..(  0.9900,  0.0000,  1.5053)--cycle;%[47]
\path[facet](  0.9900,  0.0000,  1.5053)..controls(  0.9900,  0.0000,  1.5154)and(  0.9900,  0.0000,  1.5280)..(  0.9900,  0.0000,  1.5381)--( -0.9663,  0.2575,  1.5381)..controls( -0.9456,  0.3350,  1.5280)and( -0.9072,  0.4282,  1.5154)..( -0.8674,  0.4977,  1.5053)--cycle;%[48]
\path[facet](  1.0000, -0.0013,  1.5381)..controls(  1.0000, -0.0009,  1.5481)and(  1.0000, -0.0004,  1.5607)..(  1.0000, -0.0000,  1.5708)--(  0.9900,  0.0000,  1.5708)..controls(  0.9900,  0.0000,  1.5607)and(  0.9900,  0.0000,  1.5481)..(  0.9900,  0.0000,  1.5381)--cycle;%[49]
\path[facet](  0.9900,  0.0000,  1.5381)..controls(  0.9900,  0.0000,  1.5481)and(  0.9900,  0.0000,  1.5607)..(  0.9900,  0.0000,  1.5708)--( -1.0000,  0.0000,  1.5708)..controls( -1.0000,  0.0802,  1.5607)and( -0.9869,  0.1801,  1.5481)..( -0.9663,  0.2575,  1.5381)--cycle;%[50]
\draw[axis](-1,0,\hoehe)--++(2,0,0);
\draw[axis](  0.9900,0,0.7854)--(1,0,0.7854);
}\hfill
\cylinder{%
\draw[axis,dotted](-1,0,0.7854)--(1,0,0.7854);
\path[facet]( -1.0000, -0.0000,  0.0000)..controls( -1.0000, -0.0805,  0.0020)and( -0.9868, -0.1809,  0.0045)..( -0.9660, -0.2587,  0.0065)--(  0.9990,  0.0000,  0.0065)..controls(  0.9990,  0.0000,  0.0045)and(  0.9990,  0.0000,  0.0020)..(  0.9990,  0.0000,  0.0000)--cycle;%[1]
\path[facet](  0.9990,  0.0000,  0.0000)..controls(  0.9990,  0.0000,  0.0020)and(  0.9990,  0.0000,  0.0045)..(  0.9990,  0.0000,  0.0065)--(  1.0000,  0.0001,  0.0065)..controls(  1.0000,  0.0001,  0.0045)and(  1.0000,  0.0000,  0.0020)..(  1.0000,  0.0000,  0.0000)--cycle;%[2]
\draw[axis](-1,0,0)--++(2,0,0);
\path[facet]( -0.9660, -0.2587,  0.0065)..controls( -0.9451, -0.3365,  0.0086)and( -0.9064, -0.4300,  0.0111)..( -0.8662, -0.4998,  0.0131)--(  0.9990,  0.0000,  0.0131)..controls(  0.9990,  0.0000,  0.0111)and(  0.9990,  0.0000,  0.0086)..(  0.9990,  0.0000,  0.0065)--cycle;%[3]
\path[facet](  0.9990,  0.0000,  0.0065)..controls(  0.9990,  0.0000,  0.0086)and(  0.9990,  0.0000,  0.0111)..(  0.9990,  0.0000,  0.0131)--(  1.0000,  0.0003,  0.0131)..controls(  1.0000,  0.0002,  0.0111)and(  1.0000,  0.0002,  0.0086)..(  1.0000,  0.0001,  0.0065)--cycle;%[4]
\path[facet]( -0.8662, -0.4998,  0.0131)..controls( -0.8259, -0.5695,  0.0151)and( -0.7643, -0.6499,  0.0176)..( -0.7074, -0.7068,  0.0196)--(  0.9990,  0.0000,  0.0196)..controls(  0.9990,  0.0000,  0.0176)and(  0.9990,  0.0000,  0.0151)..(  0.9990,  0.0000,  0.0131)--cycle;%[5]
\path[facet](  0.9990,  0.0000,  0.0131)..controls(  0.9990,  0.0000,  0.0151)and(  0.9990,  0.0000,  0.0176)..(  0.9990,  0.0000,  0.0196)--(  1.0000,  0.0004,  0.0196)..controls(  1.0000,  0.0004,  0.0176)and(  1.0000,  0.0003,  0.0151)..(  1.0000,  0.0003,  0.0131)--cycle;%[6]
\path[facet]( -0.7074, -0.7068,  0.0196)..controls( -0.6505, -0.7638,  0.0216)and( -0.5702, -0.8254,  0.0242)..( -0.5005, -0.8657,  0.0262)--(  0.9990,  0.0000,  0.0262)..controls(  0.9990,  0.0000,  0.0242)and(  0.9990,  0.0000,  0.0216)..(  0.9990,  0.0000,  0.0196)--cycle;%[7]
\path[facet](  0.9990,  0.0000,  0.0196)..controls(  0.9990,  0.0000,  0.0216)and(  0.9990,  0.0000,  0.0242)..(  0.9990,  0.0000,  0.0262)--(  1.0000,  0.0006,  0.0262)..controls(  1.0000,  0.0005,  0.0242)and(  1.0000,  0.0005,  0.0216)..(  1.0000,  0.0004,  0.0196)--cycle;%[8]
\path[facet]( -0.5005, -0.8657,  0.0262)..controls( -0.4308, -0.9060,  0.0282)and( -0.3373, -0.9448,  0.0307)..( -0.2596, -0.9657,  0.0327)--(  0.9990,  0.0000,  0.0327)..controls(  0.9990,  0.0000,  0.0307)and(  0.9990,  0.0000,  0.0282)..(  0.9990,  0.0000,  0.0262)--cycle;%[9]
\path[facet](  0.9990,  0.0000,  0.0262)..controls(  0.9990,  0.0000,  0.0282)and(  0.9990,  0.0000,  0.0307)..(  0.9990,  0.0000,  0.0327)--(  1.0000,  0.0008,  0.0327)..controls(  1.0000,  0.0007,  0.0307)and(  1.0000,  0.0006,  0.0282)..(  1.0000,  0.0006,  0.0262)--cycle;%[10]
\path[facet]( -0.2596, -0.9657,  0.0327)..controls( -0.1818, -0.9866,  0.0347)and( -0.0815, -0.9999,  0.0373)..( -0.0010, -1.0000,  0.0393)--(  0.9990,  0.0000,  0.0393)..controls(  0.9990,  0.0000,  0.0373)and(  0.9990,  0.0000,  0.0347)..(  0.9990,  0.0000,  0.0327)--cycle;%[11]
\path[facet](  0.9990,  0.0000,  0.0327)..controls(  0.9990,  0.0000,  0.0347)and(  0.9990,  0.0000,  0.0373)..(  0.9990,  0.0000,  0.0393)--(  1.0000,  0.0010,  0.0393)..controls(  1.0000,  0.0009,  0.0373)and(  1.0000,  0.0008,  0.0347)..(  1.0000,  0.0008,  0.0327)--cycle;%[12]
\path[facet]( -0.0010, -1.0000,  0.0393)..controls(  0.0795, -1.0001,  0.0413)and(  0.1798, -0.9870,  0.0438)..(  0.2576, -0.9663,  0.0458)--(  0.9990,  0.0000,  0.0458)..controls(  0.9990,  0.0000,  0.0438)and(  0.9990,  0.0000,  0.0413)..(  0.9990,  0.0000,  0.0393)--cycle;%[13]
\path[facet](  0.9990,  0.0000,  0.0393)..controls(  0.9990,  0.0000,  0.0413)and(  0.9990,  0.0000,  0.0438)..(  0.9990,  0.0000,  0.0458)--(  1.0000,  0.0013,  0.0458)..controls(  1.0000,  0.0012,  0.0438)and(  1.0000,  0.0011,  0.0413)..(  1.0000,  0.0010,  0.0393)--cycle;%[14]
\path[facet](  0.2576, -0.9663,  0.0458)..controls(  0.3353, -0.9455,  0.0478)and(  0.4288, -0.9070,  0.0503)..(  0.4985, -0.8669,  0.0524)--(  0.9990,  0.0000,  0.0524)..controls(  0.9990,  0.0000,  0.0503)and(  0.9990,  0.0000,  0.0478)..(  0.9990,  0.0000,  0.0458)--cycle;%[15]
\path[facet](  0.9990,  0.0000,  0.0458)..controls(  0.9990,  0.0000,  0.0478)and(  0.9990,  0.0000,  0.0503)..(  0.9990,  0.0000,  0.0524)--(  1.0000,  0.0017,  0.0524)..controls(  1.0000,  0.0016,  0.0503)and(  1.0000,  0.0014,  0.0478)..(  1.0000,  0.0013,  0.0458)--cycle;%[16]
\path[facet](  0.4985, -0.8669,  0.0524)..controls(  0.5682, -0.8268,  0.0544)and(  0.6485, -0.7654,  0.0569)..(  0.7054, -0.7088,  0.0589)--(  0.9990,  0.0000,  0.0589)..controls(  0.9990,  0.0000,  0.0569)and(  0.9990,  0.0000,  0.0544)..(  0.9990,  0.0000,  0.0524)--cycle;%[17]
\path[facet](  0.9990,  0.0000,  0.0524)..controls(  0.9990,  0.0000,  0.0544)and(  0.9990,  0.0000,  0.0569)..(  0.9990,  0.0000,  0.0589)--(  1.0000,  0.0024,  0.0589)..controls(  1.0000,  0.0021,  0.0569)and(  1.0000,  0.0019,  0.0544)..(  1.0000,  0.0017,  0.0524)--cycle;%[18]
\path[facet](  0.7054, -0.7088,  0.0589)..controls(  0.7623, -0.6522,  0.0609)and(  0.8239, -0.5723,  0.0634)..(  0.8642, -0.5032,  0.0654)--(  0.9990,  0.0000,  0.0654)..controls(  0.9990,  0.0000,  0.0634)and(  0.9990,  0.0000,  0.0609)..(  0.9990,  0.0000,  0.0589)--cycle;%[19]
\path[facet](  0.9990,  0.0000,  0.0589)..controls(  0.9990,  0.0000,  0.0609)and(  0.9990,  0.0000,  0.0634)..(  0.9990,  0.0000,  0.0654)--(  1.0000,  0.0037,  0.0654)..controls(  1.0000,  0.0031,  0.0634)and(  1.0000,  0.0027,  0.0609)..(  1.0000,  0.0024,  0.0589)--cycle;%[20]
\path[facet](  0.8642, -0.5032,  0.0654)..controls(  0.9044, -0.4341,  0.0675)and(  0.9431, -0.3415,  0.0700)..(  0.9640, -0.2660,  0.0720)--(  0.9990,  0.0000,  0.0720)..controls(  0.9990,  0.0000,  0.0700)and(  0.9990,  0.0000,  0.0675)..(  0.9990,  0.0000,  0.0654)--cycle;%[21]
\path[facet](  0.9990,  0.0000,  0.0654)..controls(  0.9990,  0.0000,  0.0675)and(  0.9990,  0.0000,  0.0700)..(  0.9990,  0.0000,  0.0720)--(  1.0000,  0.0074,  0.0720)..controls(  1.0000,  0.0052,  0.0700)and(  1.0000,  0.0043,  0.0675)..(  1.0000,  0.0037,  0.0654)--cycle;%[22]
\path[facet](  0.9640, -0.2660,  0.0720)..controls(  0.9848, -0.1904,  0.0740)and(  0.9990, -0.0447,  0.0785)..(  0.9990, -0.0447,  0.0785)--(  0.9990,  0.0000,  0.0785)..controls(  0.9990,  0.0000,  0.0785)and(  0.9990,  0.0000,  0.0740)..(  0.9990,  0.0000,  0.0720)--cycle;%[23]
\path[facet](  0.9990,  0.0000,  0.0720)..controls(  0.9990,  0.0000,  0.0740)and(  0.9990,  0.0000,  0.0785)..(  0.9990,  0.0000,  0.0785)--(  0.9990,  0.0447,  0.0785)..controls(  0.9990,  0.0447,  0.0785)and(  1.0000,  0.0096,  0.0740)..(  1.0000,  0.0074,  0.0720)--cycle;%[24]
\path[facet](  0.9990, -0.0447,  0.0785)..controls(  0.9990, -0.0447,  0.0785)and(  0.9990, -0.0447,  1.4923)..(  0.9990, -0.0447,  1.4923)--(  0.9990,  0.0000,  1.4923)..controls(  0.9990,  0.0000,  1.4923)and(  0.9990,  0.0000,  0.0785)..(  0.9990,  0.0000,  0.0785)--cycle;%[25]
\path[facet](  0.9990,  0.0000,  0.0785)..controls(  0.9990,  0.0000,  0.0785)and(  0.9990,  0.0000,  1.4923)..(  0.9990,  0.0000,  1.4923)--(  0.9990,  0.0447,  1.4923)..controls(  0.9990,  0.0447,  1.4923)and(  0.9990,  0.0447,  0.0785)..(  0.9990,  0.0447,  0.0785)--cycle;%[26]
\path[facet](  0.9990, -0.0447,  1.4923)..controls(  0.9990, -0.0447,  1.4923)and(  1.0000, -0.0096,  1.4968)..(  1.0000, -0.0074,  1.4988)--(  0.9990,  0.0000,  1.4988)..controls(  0.9990,  0.0000,  1.4968)and(  0.9990,  0.0000,  1.4923)..(  0.9990,  0.0000,  1.4923)--cycle;%[27]
\path[facet](  0.9990,  0.0000,  1.4923)..controls(  0.9990,  0.0000,  1.4923)and(  0.9990,  0.0000,  1.4968)..(  0.9990,  0.0000,  1.4988)--(  0.9640,  0.2660,  1.4988)..controls(  0.9848,  0.1904,  1.4968)and(  0.9990,  0.0447,  1.4923)..(  0.9990,  0.0447,  1.4923)--cycle;%[28]
\path[facet](  1.0000, -0.0074,  1.4988)..controls(  1.0000, -0.0052,  1.5008)and(  1.0000, -0.0043,  1.5033)..(  1.0000, -0.0037,  1.5053)--(  0.9990,  0.0000,  1.5053)..controls(  0.9990,  0.0000,  1.5033)and(  0.9990,  0.0000,  1.5008)..(  0.9990,  0.0000,  1.4988)--cycle;%[29]
\path[facet](  0.9990,  0.0000,  1.4988)..controls(  0.9990,  0.0000,  1.5008)and(  0.9990,  0.0000,  1.5033)..(  0.9990,  0.0000,  1.5053)--(  0.8642,  0.5032,  1.5053)..controls(  0.9044,  0.4341,  1.5033)and(  0.9431,  0.3415,  1.5008)..(  0.9640,  0.2660,  1.4988)--cycle;%[30]
\path[facet](  1.0000, -0.0037,  1.5053)..controls(  1.0000, -0.0031,  1.5074)and(  1.0000, -0.0027,  1.5099)..(  1.0000, -0.0024,  1.5119)--(  0.9990,  0.0000,  1.5119)..controls(  0.9990,  0.0000,  1.5099)and(  0.9990,  0.0000,  1.5074)..(  0.9990,  0.0000,  1.5053)--cycle;%[31]
\path[facet](  0.9990,  0.0000,  1.5053)..controls(  0.9990,  0.0000,  1.5074)and(  0.9990,  0.0000,  1.5099)..(  0.9990,  0.0000,  1.5119)--(  0.7054,  0.7088,  1.5119)..controls(  0.7623,  0.6522,  1.5099)and(  0.8239,  0.5723,  1.5074)..(  0.8642,  0.5032,  1.5053)--cycle;%[32]
\path[facet](  1.0000, -0.0024,  1.5119)..controls(  1.0000, -0.0021,  1.5139)and(  1.0000, -0.0019,  1.5164)..(  1.0000, -0.0017,  1.5184)--(  0.9990,  0.0000,  1.5184)..controls(  0.9990,  0.0000,  1.5164)and(  0.9990,  0.0000,  1.5139)..(  0.9990,  0.0000,  1.5119)--cycle;%[33]
\path[facet](  0.9990,  0.0000,  1.5119)..controls(  0.9990,  0.0000,  1.5139)and(  0.9990,  0.0000,  1.5164)..(  0.9990,  0.0000,  1.5184)--(  0.4985,  0.8669,  1.5184)..controls(  0.5682,  0.8268,  1.5164)and(  0.6485,  0.7654,  1.5139)..(  0.7054,  0.7088,  1.5119)--cycle;%[34]
\path[facet](  1.0000, -0.0017,  1.5184)..controls(  1.0000, -0.0016,  1.5205)and(  1.0000, -0.0014,  1.5230)..(  1.0000, -0.0013,  1.5250)--(  0.9990,  0.0000,  1.5250)..controls(  0.9990,  0.0000,  1.5230)and(  0.9990,  0.0000,  1.5205)..(  0.9990,  0.0000,  1.5184)--cycle;%[35]
\path[facet](  0.9990,  0.0000,  1.5184)..controls(  0.9990,  0.0000,  1.5205)and(  0.9990,  0.0000,  1.5230)..(  0.9990,  0.0000,  1.5250)--(  0.2576,  0.9663,  1.5250)..controls(  0.3353,  0.9455,  1.5230)and(  0.4288,  0.9070,  1.5205)..(  0.4985,  0.8669,  1.5184)--cycle;%[36]
\path[facet](  1.0000, -0.0013,  1.5250)..controls(  1.0000, -0.0012,  1.5270)and(  1.0000, -0.0011,  1.5295)..(  1.0000, -0.0010,  1.5315)--(  0.9990,  0.0000,  1.5315)..controls(  0.9990,  0.0000,  1.5295)and(  0.9990,  0.0000,  1.5270)..(  0.9990,  0.0000,  1.5250)--cycle;%[37]
\path[facet](  0.9990,  0.0000,  1.5250)..controls(  0.9990,  0.0000,  1.5270)and(  0.9990,  0.0000,  1.5295)..(  0.9990,  0.0000,  1.5315)--( -0.0010,  1.0000,  1.5315)..controls(  0.0795,  1.0001,  1.5295)and(  0.1798,  0.9870,  1.5270)..(  0.2576,  0.9663,  1.5250)--cycle;%[38]
\path[facet](  1.0000, -0.0010,  1.5315)..controls(  1.0000, -0.0009,  1.5335)and(  1.0000, -0.0008,  1.5361)..(  1.0000, -0.0008,  1.5381)--(  0.9990,  0.0000,  1.5381)..controls(  0.9990,  0.0000,  1.5361)and(  0.9990,  0.0000,  1.5335)..(  0.9990,  0.0000,  1.5315)--cycle;%[39]
\path[facet](  0.9990,  0.0000,  1.5315)..controls(  0.9990,  0.0000,  1.5335)and(  0.9990,  0.0000,  1.5361)..(  0.9990,  0.0000,  1.5381)--( -0.2596,  0.9657,  1.5381)..controls( -0.1818,  0.9866,  1.5361)and( -0.0815,  0.9999,  1.5335)..( -0.0010,  1.0000,  1.5315)--cycle;%[40]
\path[facet](  1.0000, -0.0008,  1.5381)..controls(  1.0000, -0.0007,  1.5401)and(  1.0000, -0.0006,  1.5426)..(  1.0000, -0.0006,  1.5446)--(  0.9990,  0.0000,  1.5446)..controls(  0.9990,  0.0000,  1.5426)and(  0.9990,  0.0000,  1.5401)..(  0.9990,  0.0000,  1.5381)--cycle;%[41]
\path[facet](  0.9990,  0.0000,  1.5381)..controls(  0.9990,  0.0000,  1.5401)and(  0.9990,  0.0000,  1.5426)..(  0.9990,  0.0000,  1.5446)--( -0.5005,  0.8657,  1.5446)..controls( -0.4308,  0.9060,  1.5426)and( -0.3373,  0.9448,  1.5401)..( -0.2596,  0.9657,  1.5381)--cycle;%[42]
\path[facet](  1.0000, -0.0006,  1.5446)..controls(  1.0000, -0.0005,  1.5466)and(  1.0000, -0.0005,  1.5491)..(  1.0000, -0.0004,  1.5512)--(  0.9990,  0.0000,  1.5512)..controls(  0.9990,  0.0000,  1.5491)and(  0.9990,  0.0000,  1.5466)..(  0.9990,  0.0000,  1.5446)--cycle;%[43]
\path[facet](  0.9990,  0.0000,  1.5446)..controls(  0.9990,  0.0000,  1.5466)and(  0.9990,  0.0000,  1.5491)..(  0.9990,  0.0000,  1.5512)--( -0.7074,  0.7068,  1.5512)..controls( -0.6505,  0.7638,  1.5491)and( -0.5702,  0.8254,  1.5466)..( -0.5005,  0.8657,  1.5446)--cycle;%[44]
\path[facet](  1.0000, -0.0004,  1.5512)..controls(  1.0000, -0.0004,  1.5532)and(  1.0000, -0.0003,  1.5557)..(  1.0000, -0.0003,  1.5577)--(  0.9990,  0.0000,  1.5577)..controls(  0.9990,  0.0000,  1.5557)and(  0.9990,  0.0000,  1.5532)..(  0.9990,  0.0000,  1.5512)--cycle;%[45]
\path[facet](  0.9990,  0.0000,  1.5512)..controls(  0.9990,  0.0000,  1.5532)and(  0.9990,  0.0000,  1.5557)..(  0.9990,  0.0000,  1.5577)--( -0.8662,  0.4998,  1.5577)..controls( -0.8259,  0.5695,  1.5557)and( -0.7643,  0.6499,  1.5532)..( -0.7074,  0.7068,  1.5512)--cycle;%[46]
\path[facet](  1.0000, -0.0003,  1.5577)..controls(  1.0000, -0.0002,  1.5597)and(  1.0000, -0.0002,  1.5622)..(  1.0000, -0.0001,  1.5643)--(  0.9990,  0.0000,  1.5643)..controls(  0.9990,  0.0000,  1.5622)and(  0.9990,  0.0000,  1.5597)..(  0.9990,  0.0000,  1.5577)--cycle;%[47]
\path[facet](  0.9990,  0.0000,  1.5577)..controls(  0.9990,  0.0000,  1.5597)and(  0.9990,  0.0000,  1.5622)..(  0.9990,  0.0000,  1.5643)--( -0.9660,  0.2587,  1.5643)..controls( -0.9451,  0.3365,  1.5622)and( -0.9064,  0.4300,  1.5597)..( -0.8662,  0.4998,  1.5577)--cycle;%[48]
\path[facet](  1.0000, -0.0001,  1.5643)..controls(  1.0000, -0.0001,  1.5663)and(  1.0000, -0.0000,  1.5688)..(  1.0000, -0.0000,  1.5708)--(  0.9990,  0.0000,  1.5708)..controls(  0.9990,  0.0000,  1.5688)and(  0.9990,  0.0000,  1.5663)..(  0.9990,  0.0000,  1.5643)--cycle;%[49]
\path[facet](  0.9990,  0.0000,  1.5643)..controls(  0.9990,  0.0000,  1.5663)and(  0.9990,  0.0000,  1.5688)..(  0.9990,  0.0000,  1.5708)--( -1.0000,  0.0000,  1.5708)..controls( -1.0000,  0.0805,  1.5688)and( -0.9868,  0.1809,  1.5663)..( -0.9660,  0.2587,  1.5643)--cycle;%[50]
\draw[axis](-1,0,\hoehe)--++(2,0,0);
\draw[axis](  0.9990,0,0.7854)--(1,0,0.7854);
}
\caption{
The surfaces $Y_{\lambda,s}$ defined in \eqref{eqn:Yls}: 
deforming a translated helicoid into two half-discs.}%
\label{fig:vertical_strip}%
\end{figure}

\begin{corollary}[Sweepout construction]\label{cor:sweepout}
For every $\rho>2$ and every integer $n\geq5(\rho+1)/2$ there exists a $\dih_n$-sweepout $\{\Sigma_t\}_{t\in[0,1]}$ of $\ambient$ with the following properties.
\begin{enumerate}[label={\normalfont(\roman*)}] 
\item\label{cor:sweepout-area} $\hsd^2(\Sigma_0)=\hsd^2(\Sigma_1)=n\pi$ and $\hsd^2(\Sigma_t)<2n\pi$ for every $t\in[0,1]$.  
\item\label{cor:sweepout-sym} 
For every $t\in\interval{0,1}$ the surface $\Sigma_t$ contains the segment $\xi_k$ defined in \eqref{eqn:axes} if $k$ is even and intersects it orthogonally if $k$ is odd.
\item\label{cor:sweepout-topo}
For every $t\in\interval{0,1}$ the surface $\Sigma_t$ is a Möbius band if $n$ is odd, or an annulus if $n$ is~even. 
\item\label{cor:sweepout-perim} 
For every $t\in[0,1]$ there exists a set $E_t\subset \ambient_{1/n}$ of finite perimeter such that 
\begin{itemize}[nosep]
\item 
$\Sigma_t\cap \ambient_{1/n}$ is the relative perimeter of $E_t$ in $\ambient_{1/n}$ for every $0<t<1$;
\item $E_0=\emptyset$ and $E_1=\ambient_{1/n}$; 
\item $\displaystyle\lim_{t\to t_0}\hsd^3(E_t\symdiff E_{t_0})=0$ for any $t_0\in[0,1]$.
\end{itemize}
\end{enumerate}
\end{corollary}

\begin{proof}
Let the real number $\rho>2$ and the integer $n\geq5(\rho+1)/2$ be fixed. 
Choosing $\height=2(\rho+1)$, let $\{\Gamma_t\}_{t\in[0,1]}$ be the sweepout of $Z_{\height/n}$ constructed in Lemma
\ref{lem:sweepout} (which we apply with $\height/n$ in place of $\height$). 
By assumption $\height/n\leq4/5$, therefore Lemma \ref{lem:sweepout} \ref{lem:sweepout-area} 
and Lemma \ref{lem:helicoid-area} imply 
$\hsd^2(\Gamma_t)<2\pi$ for every $t\in[0,1]$. 
Recalling the surjective map $\Phi_\height\colon Z_\height\to \ambient$ from \eqref{eqn:covering}, we define for every $t\in[0,1]$
\begin{align}\label{eqn:Sigma_t}
\hat\Gamma_t&\vcentcolon=\bigcup_{k=0}^{n-1}\translation_{(\height\pi k/n)e_3}(\Gamma_t),
&
\Sigma_t&\vcentcolon=\Phi_\height(\hat\Gamma_t). 
\end{align}
Lemma \ref{lem:biLipschitz} and our choice of $\height$ imply 
$\hsd^2(\Sigma_t)\leq\hsd^2(\hat\Gamma_t)=n\hsd^2(\Gamma_t)<2n\pi$. 
By construction, $\Sigma_0$ and $\Sigma_1$ coincide with the union 
\[
\bigcup_{k=0}^{n-1}\rotation^{2\pi k/n}_{e_3}\ambient_0
\]
of $n$ vertical discs in $\ambient$. 
This completes the proof of claim \ref{cor:sweepout-area}.
Claim \ref{cor:sweepout-sym} follows directly from Lemma~\ref{lem:sweepout}~\ref{lem:sweepout-sym} and Lemma \ref{lem:symmetry_compatibility}, which also imply that $\Sigma_t$ is a properly embedded, $\dih_n$-equivariant surface in $\ambient$. 
The boundary $\partial\Gamma_t\cap\partial Z$ is a half-twist of a double-helix (see Figures~\ref{fig:helicoid_translation}--\ref{fig:vertical_strip}). 
Therefore, $\Sigma_t$ comprises $n$ such half-twists and claim \ref{cor:sweepout-topo} follows.  
Recalling the finite perimeter set $F_t\subset Z_{\height/n}$ from Lemma~\ref{lem:sweepout}~\ref{lem:sweepout-perim} we  define $E_t\vcentcolon=\Phi_\height(F_t)$.
Then claim \ref{cor:sweepout-perim} follows directly from Lemma~\ref{lem:sweepout}~\ref{lem:sweepout-perim} and the fact that $\Phi_\height$ restricts to a diffeomorphism $Z_{\height/n}\to\ambient_{1/n}$.
\end{proof}

\section{Width Estimate}\label{sec:width}

Let $\{\Sigma_t\}_{t\in[0,1]}$ be the $\dih_n$-sweepout of $\ambient$ constructed in Corollary~\ref{cor:sweepout}. 
Its $\dih_n$-saturation $\Pi$ is defined as the set of all $\{f(t,\Sigma_t)\}_{t\in[0,1]}$, where $f\colon[0,1]\times \ambient\to \ambient$ is smooth such that $f(t,\cdot)$ is a diffeomorphism which commutes with the $\dih_n$-action for all $t\in[0,1]$ and coincides with the identity for $t\in\{0,1\}$ 
(cf. \cite[Definition~1.3]{FranzSchulz2023}). 
The corresponding min-max width of $\Pi$ is  
\[
W_\Pi\vcentcolon=\adjustlimits\inf_{\{\Lambda_t\}\in \Pi~}\sup_{t\in[0,1]}\hsd^2({\Lambda_t}).
\]
The width estimate $W_\Pi>\max\{\hsd^2(\Sigma_0),\hsd^2(\Sigma_1)\}$ is the essential condition for applying the min-max theorem (cf.~\cite[Theorem~1.4]{FranzSchulz2023} and references therein). 
The proof of the width estimate typically relies on the (relative) isoperimetric inequality in the ambient space. 

\begin{lemma}[Isoperimetric inequality in the unit ball \cite{BokowskiSperner1979}]
\label{lem:isoball}
Any surface $\Sigma$ dividing the Euclidean unit ball $\B^3\subset\R^3$ into two equal volumes has area $\hsd^2(\Sigma)\geq\pi$. 
\end{lemma}

\begin{lemma}[Isoperimetric inequality in cylinders]
\label{lem:isocyl}
Given $\height>0$, let $F\subset Z_\height=\B^2\times[0,\height\pi]$ be any set with finite perimeter and Lebesgue measure $\hsd^3(F)=\frac{1}{2}\hsd^3(Z_\height)$. 
Then its relative perimeter satisfies 
\(P(F;Z_\height)\geq\min\{\pi,2\pi\height\}\).
\end{lemma}

\begin{proof}
Theorem~23\;(d) in \cite{Ros2005} implies that among all hypersurfaces in $Z_\height$ dividing it in two equal volumes, the area is minimised either by $\B^2\times\{\height\pi/2\}$ or by $[-1,1]\times[0,\height\pi]$.   
\end{proof}

If $\height<\frac{1}{2}$ then Lemma~\ref{lem:isocyl} is insufficient to establish the width estimate for the sweepout $\{\Gamma_t\}_{t\in[0,1]}$ constructed in Lemma~\ref{lem:sweepout}, because $2\pi\height<\pi=\hsd^2(\Gamma_0)$. 
The following crucial insight allows us to overcome this limitation of the isoperimetric inequality in cylinders.

\begin{lemma}\label{lem:projection}
Given $\height>0$ let $\varpi\colon Z\to\tilde{Z}_\height$ be the quotient map associated with \eqref{eqn:quotient}. 
Let $\tilde\Sigma\subset\tilde{Z}_\height$ be any properly embedded Möbius band which allows an ambient diffeomorphism $\varphi\colon\tilde{Z}_\height\to\tilde{Z}_\height$ such that 
$\varphi(\tilde\Sigma)=\varpi(Z\cap\frac{\height}{n}X)$ for some $n\in\Z\setminus\{0\}$, where $X$ denotes the helicoid as in Section~\ref{sec:sweepout}. 
Then the projection map $p\colon\tilde\Sigma\to\B^2$ given by 
$(x_1,x_2,x_3)\mapsto(x_1,x_2)$ is surjective.
\end{lemma}

\begin{proof}
Towards a contradiction, suppose that there exists $y\in\B^2\setminus p(\tilde\Sigma)$. 
Then $\tilde\gamma\vcentcolon=\varpi(\{y\}\times\R)$ is a simple closed curve in $\tilde{Z}_\height\setminus\tilde\Sigma$. 
Let $h\in\R$ such that the disc $\tilde{D}\vcentcolon=\varpi(\B^2\times\{h\})\subset\tilde{Z}_\height$ intersects the curve $\varphi(\tilde\gamma)$ transversally. 
On the one hand, the curve $\varphi(\tilde\gamma)$ has \emph{odd} intersection number with $\tilde{D}$ because by definition the homotopy class of $\tilde\gamma$ generates the fundamental group of $\tilde{Z}_\height$. 
On the other hand, since $\varphi(\tilde\Sigma)$ is a helicoidal Möbius band by assumption, $\tilde{Z}_\height\setminus\varphi(\tilde\Sigma)$ is connected with the topology of a solid torus and 
$\tilde{D}\setminus\varphi(\tilde\Sigma)$ has exactly two connected components (two half-discs, separated by a line segment) which divide $\tilde{Z}_\height\setminus\varphi(\tilde\Sigma)$ into exactly \emph{two} connected components. 
Therefore the closed curve $\varphi(\tilde\gamma)\subset\tilde{Z}_\height\setminus\varphi(\tilde\Sigma)$ has \emph{even} intersection number with $\tilde{D}\setminus\varphi(\tilde\Sigma)$. 
This contradiction proves the surjectivity of $p$.   
\end{proof}

\begin{lemma}[Width estimate]\label{lem:width}
Let $\Pi$ be the $\dih_n$-saturation of the 
$\dih_n$-sweepout $\{\Sigma_t\}_{t\in[0,1]}$ of $\ambient$ constructed in Corollary~\ref{cor:sweepout} for $\rho>2$ and $n\geq5(\rho+1)/2$. 
Then its min-max width satisfies 
\[n\pi<W_\Pi<2n\pi.\] 
\end{lemma}

\begin{proof}
The upper bound follows directly from Corollary~\ref{cor:sweepout}~\ref{cor:sweepout-area}.  
For the lower bound, let $\{\Lambda_t\}_{t\in[0,1]}\in\Pi$ be arbitrary. 
By definition, there exists a smooth function $f\colon[0,1]\times \ambient\to \ambient$, where $f(t,\cdot)$ is a 
$\dih_n$-equivariant diffeomorphism for all $t\in[0,1]$ which coincides with the identity for $t\in\{0,1\}$, such that $\Lambda_t=f(t,\Sigma_t)$. 
Corollary~\ref{cor:sweepout}~\ref{cor:sweepout-sym} and the properties of $f(t,\cdot)$ imply that $\xi_k\subset\Lambda_t$ for every even $k$, where we recall the segments $\xi_k$ from \eqref{eqn:axes}. 
By Corollary~\ref{cor:sweepout}~\ref{cor:sweepout-perim} there exists a continuous family $\{E_t\}_{t\in[0,1]}$ of finite perimeter sets in the torus sector $\ambient_{1/n}$ such that $\Sigma_t\cap\ambient_{1/n}$ is the relative perimeter of the set $E_t\subset\ambient_{1/n}$ for all $0<t<1$. 
We also define $E_t'\vcentcolon=\rotation_{e_1}^{\pi}(\ambient_{1/n}\setminus E_t)$. 
Let $\Phi_\height\colon Z\to\ambient$ be as defined in \eqref{eqn:covering}. 
In the proof of Corollary \ref{cor:sweepout} we chose the value $2(\rho+1)$ for $\height$ in order to obtain upper area bounds using Lemma~\ref{lem:biLipschitz}. 
Now, aiming for lower area bounds, we instead choose $\height=2(\rho-1)$ and define for any $0<t<1$ 
\begin{align}\label{eqn:setFt}
L_t&\vcentcolon=\Phi_\height^{-1}(\Lambda_t), 
&
F_t&\vcentcolon=\Phi_\height^{-1}\bigl(f(t,E_t)\bigr), & 
F_t'&\vcentcolon=\Phi_\height^{-1}\bigl(f(t,E_t')\bigr). 
\end{align} 
Then $L_t$ is a properly embedded, $\height\pi/n$-periodic surface in $Z=\B^2\times\R$ containing $\zeta_k$ defined in \eqref{eqn:axes2} for every even $k$ by Lemma~\ref{lem:symmetry_compatibility}. 
In particular, it passes to a properly embedded surface $\tilde L_t$ in the quotient $\tilde{Z}_{\height/n}$ as defined in \eqref{eqn:quotient}. 
Up to a vertical translation (which does not change area) we may assume that $L_t$ intersects $Z_\height$ transversally. 
Lemma~\ref{lem:biLipschitz} and our choice of $\height$ then imply 
\begin{align}\label{eqn:20240907-2}
n\hsd^2(\tilde L_t)
=\hsd^2(L_t\cap Z_\height)
&\leq\hsd^2\bigl(\Phi_\height(L_t\cap Z_\height)\bigr)
=\hsd^2(\Lambda_t).
\end{align}
Similarly, $\Phi_\height^{-1}(\Sigma_t)$ passes to a properly embedded surface $\tilde\Gamma_t$ in $\tilde{Z}_{\height/n}$. 
Moreover, $\tilde L_t$ and $\tilde\Gamma_t$ are related by an ambient diffeomorphism. 
Recalling the definition \eqref{eqn:Sigma_t} of $\Sigma_t$ and Lemma~\ref{lem:sweepout}~\ref{lem:sweepout-topo} we know that $\tilde\Gamma_t$ and hence $\tilde L_t$ is a topological Möbius band satisfying the hypothesis of Lemma~\ref{lem:projection}. 
Hence, the projection map $p\colon\tilde L_t\to\B^2$ given by $(x_1,x_2,x_3)\mapsto(x_1,x_2)$ is surjective. 
Given any $r>0$ to be chosen, let $T_r=\{x\in\tilde{Z}_{\height/n}\st x_1^2+x_2^2\leq r^2\}$. 
Since $p$ is surjective and does not increase area, 
\begin{align}\label{eqn:20240907-1}
\hsd^2(\tilde L_t\setminus T_r)
&\geq\hsd^2\bigl(p(\tilde L_t\setminus T_r)\bigr)
=(1-r^2)\pi.
\end{align}
It remains to prove a refined estimate for $\hsd^2(\tilde L_t\cap T_r)$. 
Let $B_r(p_k)\subset Z$ denote the Euclidean ball of radius $r>0$ around the point $p_k=(0,0,k\height\pi/(2n))\in\zeta_k$ for all $k\in\Z$.  
We may choose $r\in\interval{\height\pi/(5n),\height\pi/(4n)}$ such that $L_t$ intersects $\partial B_r(p_k)$ transversally 
and such that the balls $B_r(p_k)$ are disjoint for varying $k$, as shown in Figure~\ref{fig:balls}. 
We claim that there exists $t_0\in\interval{0,1}$ such that for every $k\in\Z$  
\begin{align}\label{eqn:lower_area_bound}
\hsd^2\bigl(L_{t_0}\cap B_r(p_k)\bigr)\geq\pi r^2.
\end{align}

\begin{figure}
\centering
\pgfmathsetmacro{\scalepar}{5}
\begin{tikzpicture}[line cap=round,line join=round,scale=\scalepar,semithick]
\pgfmathsetmacro{\rhopar}{2.5}
\pgfmathsetmacro{\npar}{ceil(5*(\rhopar+1)/2)}
\pgfmathsetmacro{\lpar}{2*\rhopar-2}
\pgfmathsetmacro{\rpar}{\lpar*pi/\npar/4.25}
\pgfmathsetmacro{\Rpar}{\lpar*pi/\npar/2}
\pgfmathsetmacro{\ymax}{1.3}
\pgfmathsetmacro{\xmax}{\textwidth/\scalepar/2cm-0.001}
\pgfmathsetmacro{\ymin}{-0.2}
\shade[top color=white,bottom color=white,middle color=black!20]
(1,\ymax)rectangle(-1,\ymin);
\draw[->](-\xmax,0)--(\xmax,0)node[inner sep=0,below left={1ex and 0pt}]{$x_1$};
\draw[->](0,0)--(0,\ymax)node[inner sep=0,below right]{~$x_3$};	
\draw(-1,\ymax)node[rotate=90,inner sep=0,left](NW){$\ldots$};
\draw( 1,\ymax)node[rotate=90,inner sep=0,left](NE){$\ldots$};
\draw(NW.west)++(0,-1pt)-|(-1,0)node[sloped,pos=1,right]{$\ldots$};
\draw(NE.west)++(0,-1pt)-|( 1,0)node[sloped,pos=1,right]{$\ldots$}; 
\draw plot[plus](0,0);
\draw plot[plus](-1,0);  
\draw plot[plus](1,0); 
\begin{scope}[axis]
\draw(-1,0)--(1,0)node[near end,below]{$\zeta_0$};
\path(-1,  \Rpar)--++(2,0)node[near end,below]{$\zeta_1$};
\draw(-1,2*\Rpar)--++(2,0)node[near end,below]{$\zeta_2$};
\draw[loosely dotted](0,\Rpar)--++(1,0)(0,\Rpar)--++(-1,0);
\end{scope}
\draw plot[bullet](0,2*\Rpar)node[below right]{$p_2$};
\draw plot[bullet](0,  \Rpar)node[below right]{$p_1$}circle(\rpar);
\draw plot[bullet](0,0)node[below right]{$p_0$};
\draw (\rpar,0)arc(0:180:\rpar);
\draw[dotted](\rpar,0)arc(0:-20:\rpar)(-\rpar,0)arc(180:200:\rpar);
\draw (\rpar,2*\Rpar)arc(0:-180:\rpar);
\draw[dotted](\rpar,2*\Rpar)arc(0:20:\rpar)(-\rpar,2*\Rpar)arc(180:160:\rpar);
\draw[dotted](-\rpar,0)--++(0,2*\Rpar)(\rpar,0)--++(0,2*\Rpar);
\begin{scope}[very thick]
\draw[latex-latex](-1,0)--++(0,\Rpar)node[left,midway]{$\dfrac{\height\pi}{2n}$};
\draw[latex-latex](-\rpar,0)--(0,0)node[below,midway]{$r$};
\end{scope} 
\draw(-2/3,3*\rpar)node{$Z$};
\end{tikzpicture}
\caption{Vertical cross-section of the Euclidean unit cylinder $Z$.}%
\label{fig:balls}%
\end{figure}
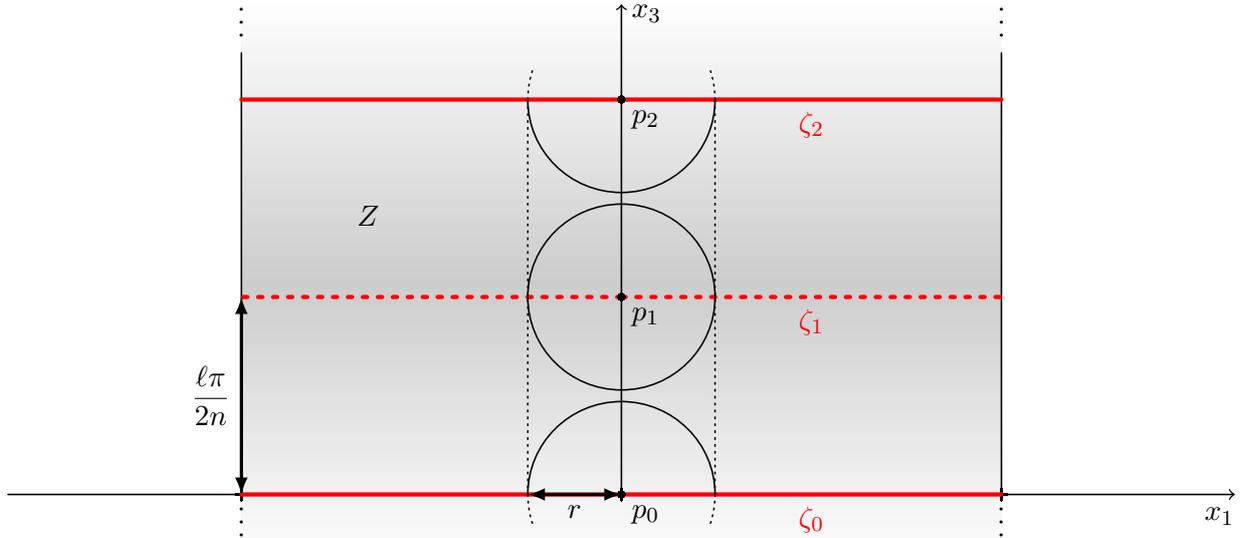

If $k$ is even, \eqref{eqn:lower_area_bound} follows from Lemma~\ref{lem:isoball} (see also \cite[Lemma~4.4]{SchulzEllipsoids}) and the fact that 
$L_{t_0}$ contains the segment $\zeta_k$ and is equivariant with respect to the rotation of angle $\pi$ around $\zeta_k$ for any choice of $t_0\in\interval{0,1}$. 
Thus, it suffices to prove \eqref{eqn:lower_area_bound} for $k=1$, recalling the periodicity of $L_t$.

The surface $L_t$ divides the cylinder $Z$ in two connected components:  
Let $\Omega_t$ denote the closure of the connected component of $Z\setminus L_t$ containing the interior of the set $F_t$ defined in \eqref{eqn:setFt}. 
Then $\Omega_t$ is $2\height\pi/n$-periodic and the closure of any connected component of $F_t\cup F_t'$ is a fundamental domain for $\Omega_t$.   
The intersection $Q_t\vcentcolon=\Omega_t\cap B_r(p_1)$ is a possibly empty set of finite perimeter. 
As stated in Corollary~\ref{cor:sweepout}~\ref{cor:sweepout-perim}, 
$\{E_t\}_{t\in\interval{0,1}}$ and thus $\{E_t'\}_{t\in\interval{0,1}}$ are continuous families of sets of finite perimeter. 
This property is inherited by 
$F_t\cup F_t'$ and thus by $Q_t$, {i.\,e.}  
\begin{align}\label{eqn:continuity}
\displaystyle\lim_{t\to t_0}\hsd^3(Q_t\symdiff Q_{t_0})=0 
\end{align} 
for any $t_0\in\interval{0,1}$. 	
By definition, $\Phi_\height(B_r(p_1))\subset\ambient_{1/n}$. 
Since $E_0=\emptyset$ and $E_1=\ambient_{1/n}$, 
and since $f(t,\cdot)$ is the identity for $t\in\{0,1\}$,  
we have $Q_0=\emptyset$ and $Q_1=B_r(p_1)$. 
By \eqref{eqn:continuity} there exists $t_0\in\interval{0,1}$ such that 
$\hsd^3(Q_{t_0})=\tfrac{1}{2}\hsd^3\bigl(B_r(p_1)\bigr)$. 
Claim \eqref{eqn:lower_area_bound} then follows from Lemma \ref{lem:isoball}. 

By choice of $r$, estimate \eqref{eqn:lower_area_bound} implies 
$\hsd^2(\tilde L_t\cap T_r)\geq2\pi r^2$ (cf.~Figure~\ref{fig:balls}). 
Combined with inequality \eqref{eqn:20240907-1} we obtain $\hsd^2(\tilde L_t)\geq(1+r^2)\pi$ and thus $\hsd^2(\Lambda_t)\geq n(1+r^2)\pi$ by \eqref{eqn:20240907-2}.
Since $\{\Lambda_t\}_{t\in[0,1]}\in\Pi$ is arbitrary, and since $r>\height\pi/(5n)=2(\rho-1)\pi/(5n)$ we obtain 
\begin{align}\label{eqn:20240917}
W_\Pi&>n\biggl(1+ \Bigl(\frac{2(\rho-1)\pi}{5n}\Bigr)^2\biggr)\pi
\end{align}
implying $W_\Pi>n\pi$ as claimed. 
\end{proof}

\begin{remark} 
The width estimate \eqref{eqn:20240917} being quantitative in terms of the parameter $\rho$ and the order $2n$ of the symmetry group is a novelty compared with the respective arguments in \cite{Ketover2016FBMS,CarlottoFranzSchulz2022,SchulzEllipsoids}. 
Moreover, it is remarkable that the upper and lower bounds on the width \emph{increase} with the order of the symmetry group rather than being uniform in $n$.  
\end{remark}

\section{Uniqueness of free boundary minimal discs in cylinders and toroids}
\label{sec:uniqueness}
 
Nitsche \cite{Nitsche1985} proved that the planar, equatorial disc is unique  
in the class of immersed free boundary minimal discs in the Euclidean unit ball $\B^3$ up to ambient isometries. 
Fraser and Schoen \cite{FraserSchoen2015} later extended this result to higher codimensions. 
In more general ambient manifolds, free boundary minimal discs are harder to characterise. 
In certain ellipsoids for example a surprising variety of nonplanar free boundary minimal discs have been discovered \cite{Petrides,HaslhoferKetover,SchulzEllipsoids}. 
In contrast, a free boundary minimal disc in a cylinder is necessarily planar and orthogonal to the cylinder axis by Corollary 5 in \cite[§\,5.5]{DHKW1992}.   
This interlude section contains an alternative proof of this result in a more general setting. 
Based on this, we characterise all free boundary minimal discs in the toroid $\ambient$, which is a key ingredient for the proof of Theorem~\ref{thm:main}.

\begin{theorem}\label{thm:uniqueness_cylinder1}
Let $(\B^2,g_{\B^2})$ be the compact Euclidean unit disc in $\R^2$ and let $f\colon\B^2\to\interval{0,\infty}$ be smooth. 
Let $\Sigma$ be any compact, connected free boundary minimal surface in the 
Riemannian manifold $(Z,g)$, where $Z=\B^2\times\R$ is a cylinder equipped with Cartesian coordinates $(x_1,x_2,x_3)$ and $g=g_{\B^2}+f^2\,dx_3^2$. 
Then $\Sigma=\B^2\times\{h\}$ for some $h\in\R$. 
\end{theorem}

\begin{proof} 
We may assume that $\Sigma$ is the image of some conformal map $u=(u_1,u_2,u_3)\colon\Omega\to Z$. 
Denoting the partial derivatives of $u$ by $\partial_1 u$ and $\partial_2u$, we have 
$\abs{\partial_1u}_g^2=\abs{\partial_2u}_g^2$ and $g(\partial_1u,\partial_2u)=0$.  
Consider the diffeomorphism $\varphi_t\colon Z\to Z$ given by $\varphi_t(x_1,x_2,x_3)=(x_1,x_2,e^{t}x_3)$. 
By the area formula \cite[Theorem~8.1]{Maggi2012}
\begin{align*}
\hsd^2\bigl(\varphi_t\bigl(\Sigma\bigr)\bigr)
&=\int_{\Omega}\sqrt{
\abs{D\varphi_t\partial_1u}_g^2
\abs{D\varphi_t\partial_2u}_g^2
-g(D\varphi_t\partial_1u,D\varphi_t\partial_2u)^2 
}.
\intertext{Since $\abs{D\varphi_t\partial_ju}_g^2
=\abs{\partial_ju}_g^2+(e^{2t}-1)(f\partial_ju_3)^2$
and 
$g(D\varphi_t\partial_1u,D\varphi_t\partial_2u)
=(e^{2t}-1)(f\partial_1u_3)(f\partial_2u_3)$,  
and since $f$ and thus $g$ are independent of $x_3$, the stationarity of $\Sigma$ implies that 
}
0=\frac{\partial}{\partial t}\Big\vert_{t=0}\hsd^2\bigl(\varphi_t\bigl(\Sigma\bigr)\bigr)
&=\int_{\Omega}
\frac{\abs{\partial_1u}_g^2(f\partial_2u_3)^2+\abs{\partial_2u}_g^2(f\partial_1u_3)^2}{\abs{\partial_1u}_g\abs{\partial_2u}_g}
=\int_{\Omega}\Bigl((\partial_1u_3)^2+(\partial_2u_3)^2\Bigr)f^2.
\end{align*}
We conclude that $u_3$ is constant, relying on the assumption that the domain is connected.  
\end{proof}

\begin{corollary}\label{cor:uniqueness}
Any embedded free boundary minimal disc $\Sigma$ in the toroid $\ambient$ defined in \eqref{eqn:ambient} coincides with the vertical planar slice $\ambient_0$ up to a rotation of the coordinate system. 
\end{corollary}

\begin{proof}
Since $\Sigma\subset\ambient$ is a properly embedded disc, the boundary curve $\partial\Sigma$ is contractible inside the toroid $\ambient$ (including its interior). 
Restricted to the boundary $\partial\ambient$, the curve $\partial\Sigma$ is therefore either contractible in $\partial\ambient$ or homotopic to the circle $\partial\ambient_0$. 
In any case, 
recalling the map $\Phi_\height\colon Z\to\ambient$ from \eqref{eqn:covering} and setting $\height=1$, any connected component $\Gamma$ of $\Phi_\height^{-1}(\Sigma)\subset Z$ is compact. 
Since $\Gamma$ is a compact, connected free boundary minimal surface in the Riemannian manifold $(Z,\Phi_\height^*g_{\ambient})$ and since $\Phi_\height^*g_{\ambient}=g_{\B^2}+f^2\,d x_3^2$ for $f=2(\rho+x_1)/\height$ the claim follows from Theorem~\ref{thm:uniqueness_cylinder1}.
\end{proof}

\section{Topological control}\label{sec:proof}

In this section we prove Theorem~\ref{thm:main} 
using $\dih_n$-equivariant min-max methods employing the sweepout constructed in Section~\ref{sec:sweepout} and the width estimate obtained in Section~\ref{sec:width}. 
The control of the topology of the resulting limit surface relies on the author's previous work with Franz, specifically \cite[Theorem~1.8]{FranzSchulz2023} stating that the first Betti number $\betti_1$ and the genus complexity $\gsum$ defined in \cite[Definition~1.6]{FranzSchulz2023} are lower semicontinuous along min-max sequences.  
However, if $\Sigma$ is an annulus or a Möbius band then $\betti_1(\Sigma)=1$ and $\gsum(\Sigma)=0$ in both cases.
To distinguish between them, additional information is required.
The following lemma implies that one cannot obtain an annulus from a Möbius band through surgery. 
It involves the boundary complexity $\bsum$ defined in \cite[Definition~1.6]{FranzSchulz2023} as the sum of the number of boundary components minus~$1$ over each connected component with boundary. 
For example, an arbitrary finite union $\Sigma_0$ of topological spheres, discs and Möbius bands satisfies $\bsum(\Sigma_0)=0$ while an annulus $\Sigma_1$ satisfies $\bsum(\Sigma_1)=1$. 
Note that in general, the boundary complexity can increase through surgery.

\begin{lemma}\label{lem:surgery}
Given a smooth, compact, topological Möbius band $\Sigma$ which is properly embedded in some three-dimensional ambient manifold $M$, let $\hat\Sigma$ be obtained from $\Sigma$ through surgery in the sense of \cite[Definition~3.1]{FranzSchulz2023}. 
Then $\bsum(\hat\Sigma)=0$. 
In particular, $\hat\Sigma$ does not contain any annuli. 
\end{lemma}

\begin{proof}
Focusing on the boundary complexity, it suffices to consider a surface $\hat\Sigma$ which is obtained from $\Sigma$ by cutting away a half-neck in the sense of \cite[Definition~3.1\,(b)]{FranzSchulz2023}. 
In this case, the Euler characteristics of $\Sigma$ and $\hat\Sigma$ are related by $\chi(\hat\Sigma)=\chi(\Sigma)+1=1$. 
We denote by $\hat{c}_{\mathrm{O}}$ respectively $\hat{c}_{\mathrm{N}}$ the number of orientable respectively nonorientable connected components of $\hat\Sigma$ and by $\hat{c}_{\mathrm{b}}$ the number of its boundary components. 
By \cite[Lemma~3.4]{FranzSchulz2023}, the genus complexity of $\hat\Sigma$ necessarily vanishes 
and \cite[Corollary~A.2]{FranzSchulz2023} then implies  
\begin{align}\label{eqn:Euler}
1=\chi(\hat\Sigma)=2\hat{c}_{\mathrm{O}}+\hat{c}_{\mathrm{N}}-\hat{c}_{\mathrm{b}}.
\end{align}  
Since $\partial\Sigma$ is connected we have either $\hat{c}_{\mathrm{b}}=1$ or $\hat{c}_{\mathrm{b}}=2$. 
In the first case, the surface stays connected, i.\,e. $\hat{c}_{\mathrm{O}}+\hat{c}_{\mathrm{N}}=1$, and \eqref{eqn:Euler} implies $\hat{c}_{\mathrm{O}}=1$ and $\hat{c}_{\mathrm{N}}=0$. 
Thus, $\hat\Sigma$ is a topological disc and any further surgery operation cannot increase the boundary complexity. 
In the second case, $2\hat{c}_{\mathrm{O}}+\hat{c}_{\mathrm{N}}=3$ and since necessarily $\hat{c}_{\mathrm{N}}\leq2$ we have 
$\hat{c}_{\mathrm{O}}=1=\hat{c}_{\mathrm{N}}$. 
Since each of the two connected components must have boundary, $\hat\Sigma$ is the union of a disc and a Möbius band. 
The claim follows by iterating the argument. 
\end{proof}

\begin{proof}[Proof of Theorem~\ref{thm:main}]
Given $\rho>2$ and $n\geq5(\rho+1)/2$ let $\{\Sigma_t\}_{t\in[0,1]}$ be the $\dih_n$-equivariant sweepout of $\ambient$ constructed in Corollary~\ref{cor:sweepout}. 
The width estimate stated in Lemma~\ref{lem:width} and the mean-convexity of $\partial\ambient$ imply that the min-max theorem \cite[Theorem~1.4]{FranzSchulz2023} applies:  
There exists a min-max sequence $\{\Sigma^j\}_{j\in\N}$ converging in the sense of varifolds to 
\(
\Gamma\vcentcolon= \sum_{i=1}^k m_i\Gamma_i
\)
for some $k\in\N$, 
where the varifolds $\Gamma_1,\ldots,\Gamma_k$ are induced by pairwise disjoint, connected, embedded free boundary minimal surfaces in $\ambient$ and where the multiplicities $m_1,\ldots,m_k$ are positive integers. 
Moreover, $\Gamma$ is $\dih_n$-equviarant. 
Recalling Lemma~\ref{lem:width}, we have
\begin{align}\label{eqn:npi2npi}
\abs{\Gamma}\vcentcolon=
\sum_{i=1}^k m_i\hsd^2(\Gamma_i)
&=W_\Pi\in\interval{n\pi,2n\pi}.
\end{align}
Corollary~\ref{cor:sweepout}~\ref{cor:sweepout-topo} implies that all the surfaces in the min-max sequence have first Betti number $\betti_1(\Sigma^j)=1$. 
Applying the topological lower semicontinuity result \cite[Theorem~1.8]{FranzSchulz2023}, we obtain 
\begin{align}\label{eqn:betti1}
\sum_{i=1}^k \betti_1(\Gamma_i)\leq 1.
\end{align}
Since $\ambient\subset\R^3$ does not contain any closed minimal surfaces, \eqref{eqn:betti1} implies that each $\Gamma_i$ is either a topological disc, an annulus or a Möbius band, and that at most one of them is not a disc.  
 
Towards a contradiction suppose that $\betti_1(\Gamma_i)=0$ for some $i\in\{1,\ldots,k\}$. 
Up to relabelling we may assume $i=1$. 
Then $\Gamma_1$ is a free boundary minimal disc in $\ambient$. 
The uniqueness result Corollary~\ref{cor:uniqueness} implies that $\Gamma_1$ coincides with the vertical slice $\ambient_0$ up to a rotation of the coordinate system. 
In particular, $\hsd^2(\Gamma_1)=\pi$. 
Being $\dih_n$-equivariant, the union $\Gamma_1\cup\ldots\cup\Gamma_k$ also contains 
$V=\bigcup_{j=1}^n\rotation^{2\pi j/n}_{e_3}\Gamma_1$ having area $\hsd^2(V)=n\pi$. 
If there exists $j\in\{1,\ldots,k\}$ such that $\betti_1(\Gamma_j)=1$, then $\Gamma_j$ is disjoint from $V$ which by $\dih_n$-equivariance implies that the support of $\Gamma$ contains at least $n$ pairwise disjoint copies of $\Gamma_j$ contradicting \eqref{eqn:betti1}.
 
We conclude that if $\Gamma_1$ is a disc then all $\Gamma_1,\ldots,\Gamma_k$ are planar, vertical discs isometric to $\ambient_0$.  
In this case, the $\dih_n$-equivariance implies that $\abs{\Gamma}$ is an integer multiple of $n\pi$ contradicting \eqref{eqn:npi2npi}. 
As a result we obtain $k=1$ and $\Gamma=m_1\Gamma_1$ with $\betti_1(\Gamma_1)=1$. 
In particular, the free boundary minimal surface $\sol_n\vcentcolon=\Gamma_1$ is either an annulus or a Möbius band. 
It remains to prove properties~\ref{thm:main-i}--\ref{thm:main-iii} stated in the theorem. 
We will also prove $m_1=1$ to obtain the area estimate stated in \ref{thm:main-iii}. 

\begin{itemize}[wide]
\item[\ref{thm:main-i}] 
Let the integer $k\in\{0,\ldots,2n-1\}$ be even. 
By construction, every surface along the min-max sequence contains the segment $\xi_k$ defined in~\eqref{eqn:axes}. 
Consequently, $\xi_k\subset \sol_n$ which also implies that the multiplicity $m_1$ is odd (see \cite[§\,7.3]{Ketover2016FBMS}). 

Now let the integer $k\in\{0,\ldots,2n-1\}$ be odd. 
By \cite[Lemma~3.4~(2)]{Ketover2016Equivariant} (applied in a suitable ball in $\ambient$ containing $\xi_k$) we have either $\xi_k\subset\sol_n$ or $\sol_n\cap\xi_k$ is finite (possibly empty) and every intersection is orthogonal. 
By Lemma~\ref{cor:sweepout}~\ref{cor:sweepout-sym}, every surface along the min-max sequence intersects $\xi_k$ orthogonally. 
The equivariance with respect to rotation by angle $\pi$ around $\xi_k$ then implies that if $\xi_k\subset\sol_n$, the multiplicity $m_1\in\N$ is even by \cite[Theorem~3.2.iv]{Ketover2016FBMS} (see also \cite[Theorem~1.3.f]{Ketover2016Equivariant}). 
This however contradicts the fact shown above that $m_1$ is odd. 
Hence, the intersection $\sol_n\cap\xi_k$ is finite and orthogonal if not empty. 
It remains to prove that $\sol_n\cap\xi_k$ is in fact nonempty. 

The quotient $\tilde\ambient=\ambient/\Z_n$ obtained by factoring out the cyclic subgroup of order $n$ is a smooth, mean convex manifold. 
Let $\varpi\colon\ambient\to\tilde\ambient$ denote the quotient map. 
Since the surfaces $\sol_n$ and $\Sigma^j$ are $\Z_n$-equivariant in $\ambient$, 
their quotients $\tilde\sol_n=\varpi(\sol_n)$ and $\tilde\Sigma^j=\varpi(\Sigma^j)$ are smooth, properly embedded surfaces in $\tilde\ambient$.  
Moreover, the sequence $\{\tilde{\Sigma}^j\}_{j\in\N}$ converges in the sense of varifolds to $m_1\tilde\sol_n$ as $j\to\infty$ and is $\Z_2$-almost minimizing, where $\Z_2$ is the action of the group $\dih_n$ reduced the quotient $\tilde\ambient$. 
Recalling Lemma~\ref{lem:sweepout}~\ref{lem:sweepout-topo} and Corollary~\ref{cor:sweepout}, the surface $\tilde\Sigma^j$ has the topology of a Möbius band. 
We claim that $\tilde\sol_n$ is also a Möbius band in $\tilde\ambient$. 
A variant of the Riemann--Hurwitz formula (see e.\,g.~\cite[§\,IV.3]{Freitag2011}) implies $\betti_1(\tilde\sol_n)=\betti_1(\sol_n)=1$. 
Suppose that $\tilde\sol_n$ is not a Möbius band, in which case it must be an annulus. 
Let $U\subset\tilde\ambient$ be a thin tubular neighborhood around $\tilde\sol_n$. 
By \cite[Theorem~4.11]{FranzSchulz2023} we can apply a topological surgery procedure to all surfaces $\tilde\Sigma^j$ with sufficienlty large $j$, resulting in $\Z_2$-equivariant surfaces $\hat\Sigma^j\subset U$
such that the sequence $\{\hat\Sigma^j\}_{j}$ still converges to $m_1\tilde\sol_n$ in the sense of varifolds and such that the boundary complexity $\bsum$ defined in \cite[Definition~1.6]{FranzSchulz2023} satisfies
\begin{align*}
\bsum(\tilde\sol_n)&\leq\liminf_{j\to\infty}\bsum(\hat\Sigma^j).  
\end{align*}
If $\tilde\sol_n$ is an annulus, then $\bsum(\tilde\sol_n)=1$ implying that $\bsum(\hat\Sigma^j)\geq1$ for all sufficiently large $j$. 
Using Lemma~\ref{lem:surgery} we obtain a contradiction to the fact that $\hat\Sigma^j$ is obtained from the Möbius band $\tilde\Sigma^j$ via surgery. 
Consequently $\tilde\sol_n$ is a topological Möbius band. 

Let $\tilde\xi_k=\varpi(\xi_k)$ for $k\in\{0,1\}$. 
Statement \ref{thm:main-i} implies $\tilde\xi_0\subset\tilde\sol_n$ and $\tilde\sol_n\setminus\tilde\xi_0$ is a topological disc. 
We claim that $\tilde\xi_1$ intersects $\tilde\sol_n$. 
Let $\rotation$ be the generator of the $\Z_2$-action on $\tilde\ambient$. 
Note that $\tilde\xi_0\cup\tilde\xi_1$ is the singular locus of this group action. 
Given $p\in\tilde\sol_n\setminus(\tilde\xi_0\cup\tilde\xi_1)$ we have $p\neq\rotation p\in\tilde\sol_n$. 
Since $\tilde\sol_n\setminus\tilde\xi_0$ is connected, a curve $\gamma\subset\tilde\sol_n\setminus\tilde\xi_0$ connects $p$ and $\rotation p$. 
If $\gamma$ is disjoint from $\tilde\xi_1$ then $\gamma\cup\rotation\gamma$ contains a simple closed curve winding around $\tilde\xi_1$. 
This curve is contractible in $\tilde\sol_n\setminus\tilde\xi_0$ because $\tilde\sol_n\setminus\tilde\xi_0$ is a topological disc. 
Therefore, $\tilde\sol_n\cap\tilde\xi_1$ must be nonempty. 
Consequently, the intersection $\sol_n\cap\xi_k$ is nonempty for all odd $k$ as claimed.

\pagebreak[2]

\item[\ref{thm:main-ii}] 
We now investigate the orientability of $\sol_n$ depending on $n$. 
Since the quotient $\tilde\sol_n$ is a topological Möbius band, there exists a curve in $\partial\sol_n$ connecting the start point of the segment $\xi_0$ (measured by distance from the origin) with the end point of the segment $\xi_2$ without intersecting any $\xi_k$ for $k\notin\{0,2\}$. 
This means that $\sol_n$ makes an odd number $j$ of half-twists between $\xi_0$ and $\xi_2$. 
Thus, $\sol_n$ comprises $n j$ half-twists in total, implying \ref{thm:main-ii}. 

Alternatively one can argue as follows:  
The singular locus of the subgroup generated by $\rotation_{e_1}^{\pi}\in\dih_n$ acting on $\ambient$ is given by 
$\xi_0\cup\xi_n$. 
Let $n$ be odd and suppose that $\sol_n$ is orientable. 
Then there exists a globally smooth choice of a unit normal vector field $\nu$ on $\sol_n$. 
The map $\rotation_{e_1}^{\pi}$ acts as an isometry on $\sol_n$ and preserves $\nu$ at $\xi_n\cap\sol_n$ because the intersection is orthogonal and nonemtpy as shown in \ref{thm:main-i}. 
However, $\rotation_{e_1}^{\pi}$ reverses $\nu$ on $\xi_0\subset\sol_n$ which contradicts the assumption of orientability. 
Let now $n$ be even. 
Then $\sol_n$ contains both $\xi_0$ and $\xi_n$ by \ref{thm:main-i}. 
Since $\sol_n$ is either an annulus or a Möbius band, the set $\sol_n\setminus(\xi_0\cup\xi_n)$ has two connected components $V_1$ and $V_2=\rotation_{e_1}^{\pi}V_1$, both of them topological discs.  
Given a smooth unit normal vector field $\nu$ on $V_1$ we obtain a compatible unit normal vector field on $V_2$ by mapping $-\nu$ to $V_2$ via $\rotation_{e_1}^{\pi}$. 
Therefore, $\sol_n$ is orientable. 

\item[\ref{thm:main-iii}] For the area estimate, we set $\height=2(\rho-1)$ as in \eqref{eqn:setFt} and consider the surface 
\(
L_n\vcentcolon=\Phi_\height^{-1}(\sol_n)
\) 
which, by Lemma~\ref{lem:symmetry_compatibility}, is $\height\pi/n$-periodic in $Z=\B^2\times\R$ containing $\zeta_k$ defined in \eqref{eqn:axes2} for every even $k$. 
It passes to a properly embedded surface $\tilde L_n$ in the quotient $\tilde{Z}_{\height/n}$ as defined in \eqref{eqn:quotient}. 
Moreover, $\tilde L_n$ has the same topology as the Möbius band $\tilde\sol_n$. 
As in the proof of Lemma~\ref{lem:width} we may apply Lemmata~\ref{lem:projection} and~\ref{lem:biLipschitz} to obtain
\begin{align}
\label{eqn:20240909}
\hsd^2(\sol_n)
&=n\hsd^2(\tilde\sol_n)\geq n\hsd^2(\tilde L_n)\geq n\pi.
\end{align}
Since $m_1\hsd^2(\sol_n)=W_\Pi\in\interval{n\pi,2n\pi}$ by \eqref{eqn:npi2npi}, estimate \eqref{eqn:20240909} implies $m_1=1$ and thus we have 
$n\pi<\hsd^2(\sol_n)<2n\pi$ as claimed. 
\qedhere
\end{itemize}
\end{proof}

\section{Conjectures and simulations}\label{sec:conjectures}

We conjecture that the surfaces $\sol_n\subset\ambient$ constructed in Theorem~\ref{thm:main}~\ref{thm:main-i}--\ref{thm:main-ii} exist in fact for \emph{any} number $n\in\N$ of half-twists.
We visualise the conjectural surfaces for $n\in\{1,3\}$ in Figure~\ref{fig:1}. 
However, we expect that the upper area bound in claim \ref{thm:main-iii} does require a lower bound on~$n$ as stated in the theorem. 
Moreover, if $n$ is too small depending on the radius $\rho$, we expect the minimal surfaces $\sol_n$ to have $\dih_n$-equivariant index greater than $1$. 
Their variational construction would thus require multi-parameter sweepouts, which are beyond the scope of this article.

Furthermore, it seems that -- unlike discs -- free boundary minimal annuli and Möbius bands in $\ambient$ are surprisingly nonunique even if one prescribes their full symmetry group (see Figure~\ref{fig:2}).  
Given any integer $n\geq0$, there may exist \emph{infinitely} many embedded solutions with $n$ half-twists. 
Some inspiration for this conjecture comes from a method called ``stacking'' (cf.~\cite{CSWstackings}). 
The simulations visualised in Figure~\ref{fig:3} suggest that a union of (sufficiently many) pairwise disjoint free boundary minimal discs in $\ambient$ can be connected with  \emph{a single} half-catenoidal bridge between each pair of adjacent discs, and then perturbed into a free boundary minimal surface. 
(This might be surprising in view of \cite{CSWstackings}, where the stacking construction requires \emph{many} half-catenoidal bridges between each layer.)
The relative positions of these bridges then determine the number $n$ of half-twists.

\begin{figure}%
\begin{torus}
\FBMS{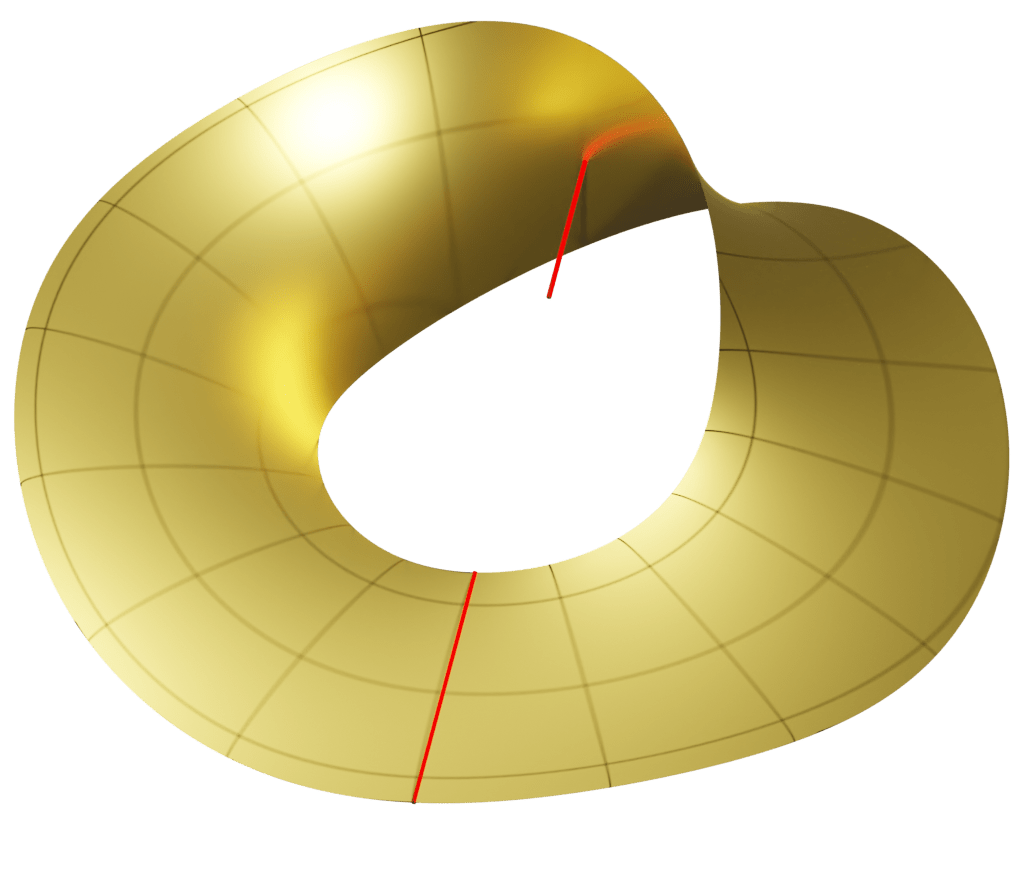}
\end{torus}
\hfill
\begin{torus}
\FBMS{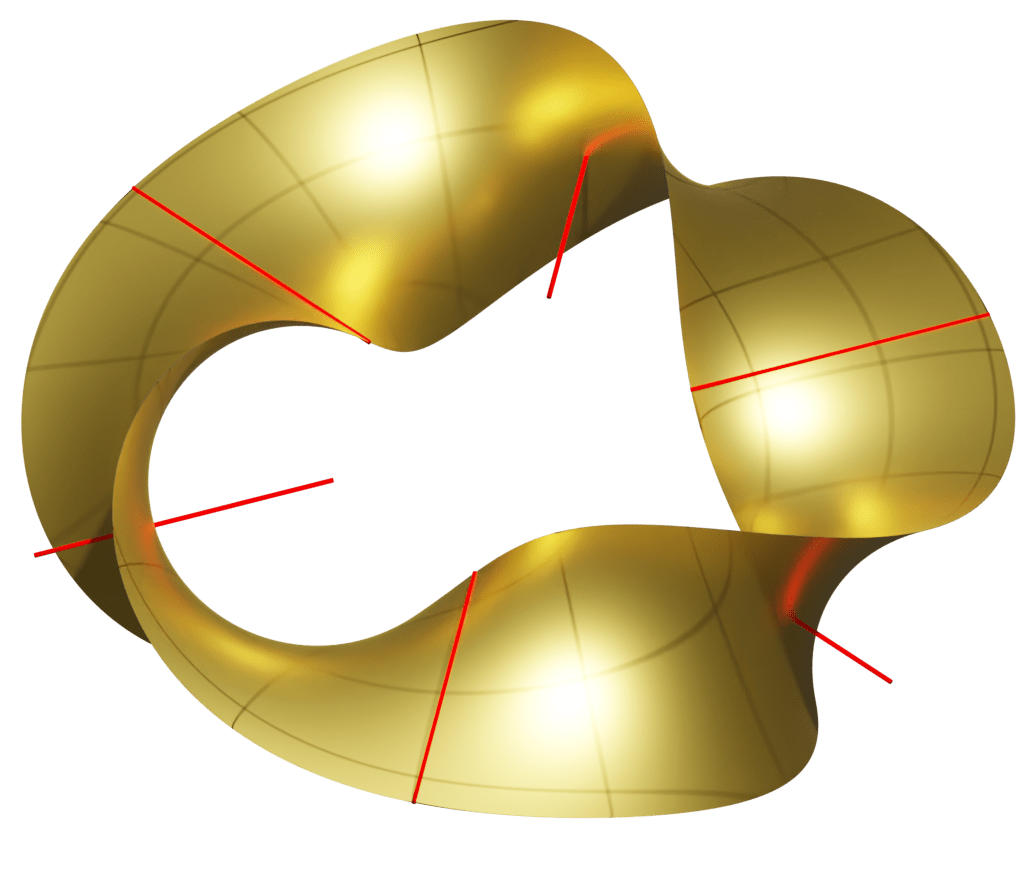}
\end{torus}
\caption{Free boundary minimal Möbius bands $\sol_n\subset\ambient$ with $n\in\{1,3\}$ half-twists.}%
\label{fig:1}%
\medskip
\begin{torus}
\FBMS{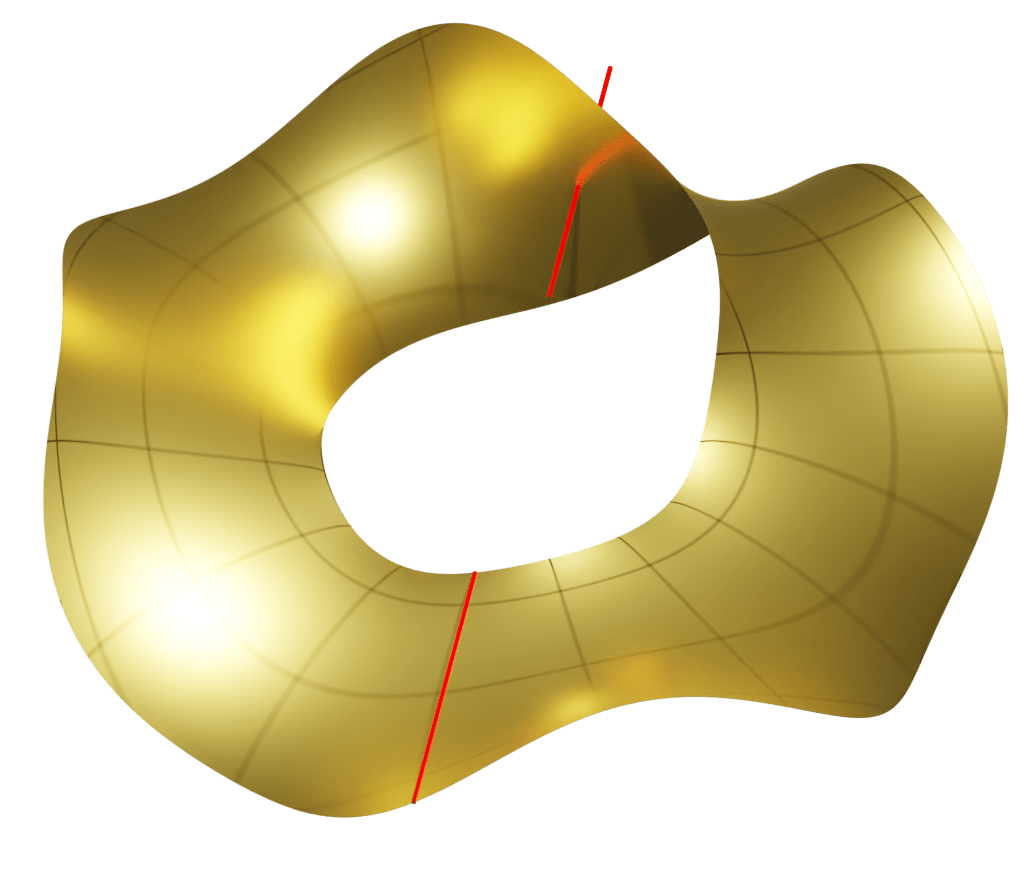}
\end{torus}
\hfill
\begin{torus}
\FBMS{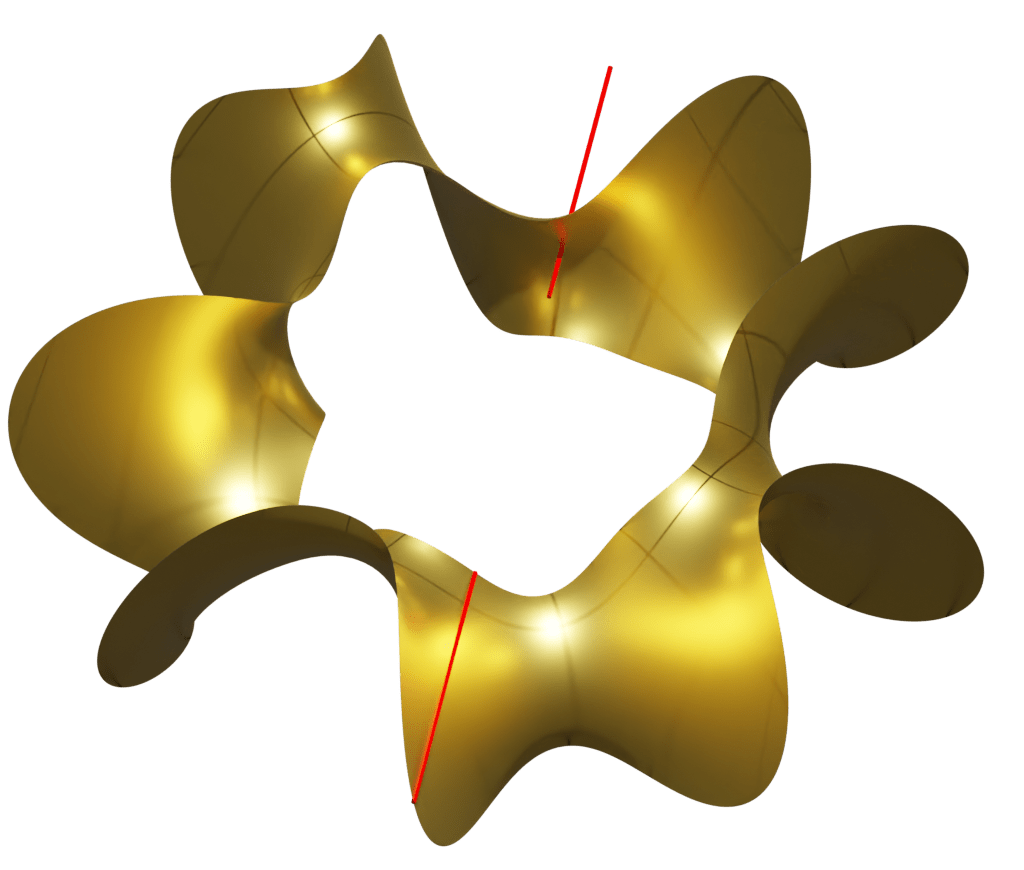}
\end{torus}
\caption{Nonuniqueness of $\dih_1$-equivariant free boundary minimal Möbius bands in $\ambient$.}%
\label{fig:2}%
\medskip
\begin{torus}
\FBMS{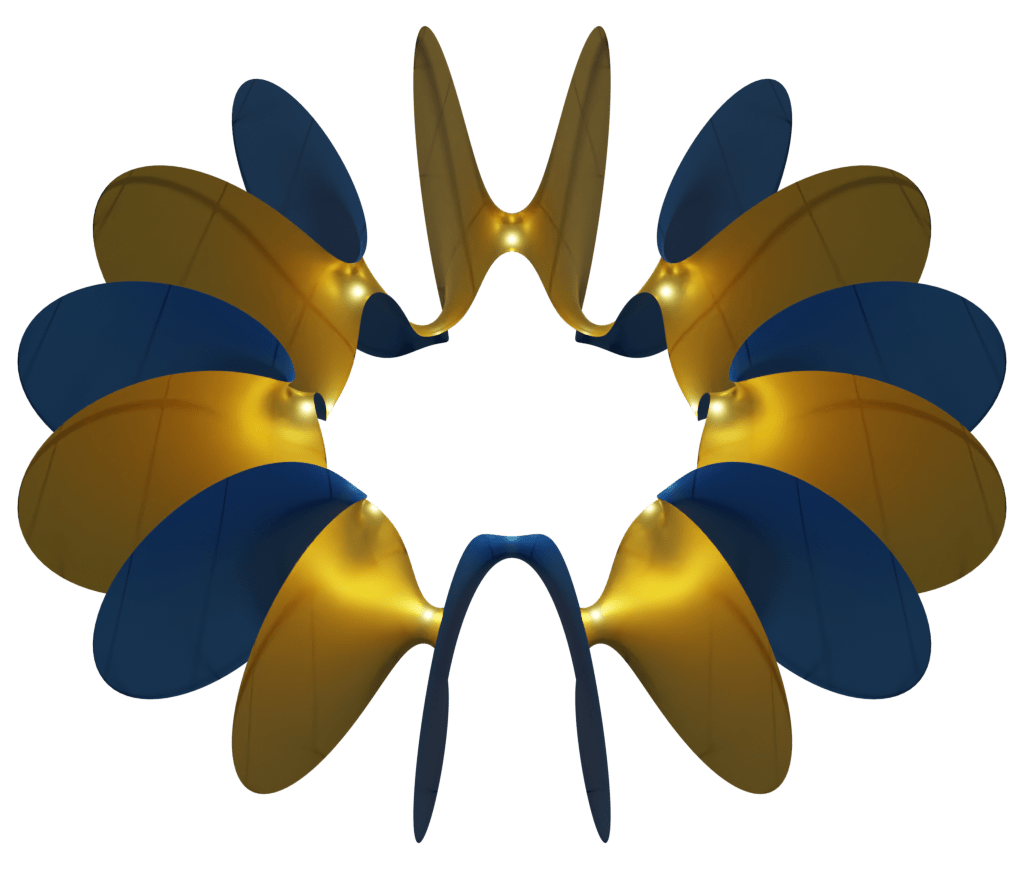}
\end{torus}
\hfill
\begin{torus}
\FBMS{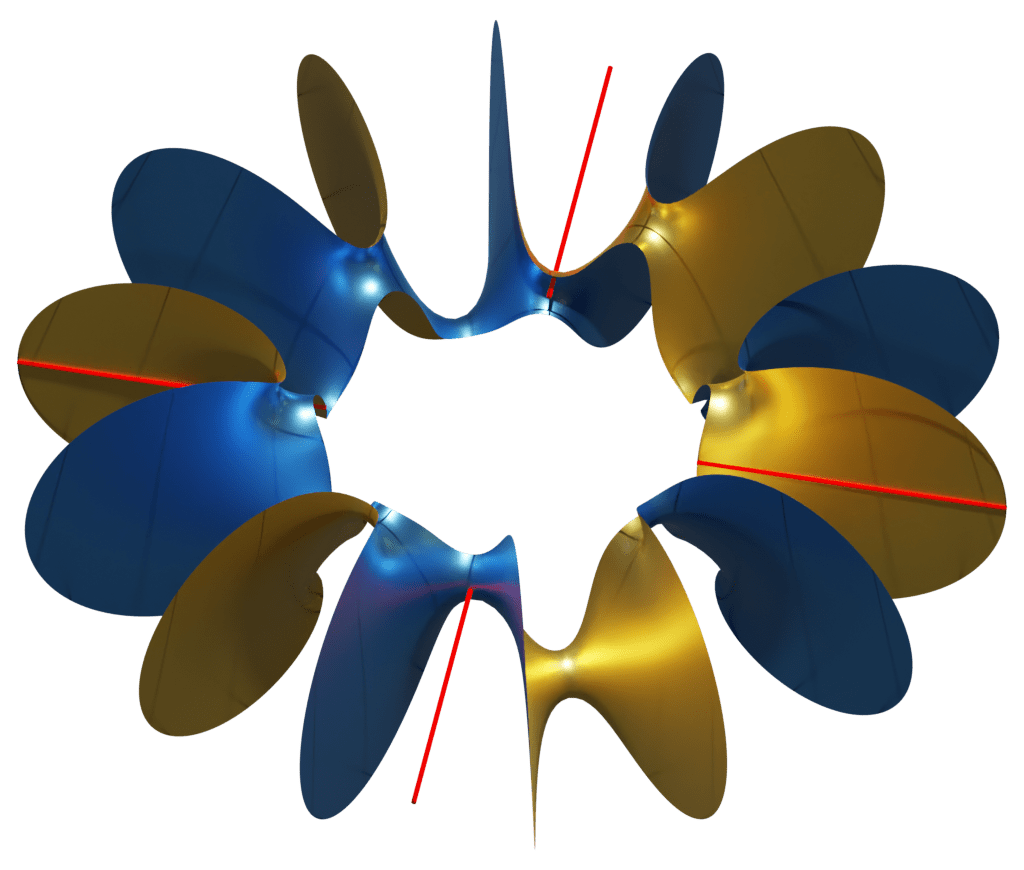}
\end{torus}
\caption{Free boundary minimal disc stackings in $\ambient$ with $0$ respectively $2$ half-twists.}%
\label{fig:3}%
\end{figure}

%===== BIBLIOGRAPHY ============================================================
\clearpage
\bibliography{fbms-bibtex}

\printaddress

\end{document}